\documentclass{amsart}

\usepackage{amsmath}
\usepackage{amssymb}
\usepackage{paralist}
\usepackage[english]{babel}
\usepackage[colorlinks=true]{hyperref}
\usepackage{enumitem}
\usepackage{algorithm,algorithmic}

\usepackage{array}
\usepackage[latin1]{inputenc}
\usepackage[dvipsnames]{xcolor}
\usepackage{esint}
\usepackage{latexsym,amsfonts}
\usepackage{amsthm}
\usepackage{cleveref}

\usepackage{mathtools}
\usepackage{verbatim}
\usepackage{graphics}

\usepackage{tikz}
\usetikzlibrary{decorations.pathreplacing}
\usepackage{pgfplots}
\pgfplotsset{compat=newest}
\usepackage{tikz-3dplot}
\usetikzlibrary{arrows}
\usetikzlibrary{intersections,calc,angles}

\numberwithin{equation}{section}

\newtheorem{theorem}{Theorem}[section]
\newtheorem*{theorem7}{Theorem~\ref{thm:cald}}
\newtheorem{lemma}{Lemma}[section]
\newtheorem{proposition}{Proposition}[section]
\newtheorem{corollary}{Corollary}[section]

\theoremstyle{definition}
\newtheorem{definition}{Definition}[section]
\newtheorem{remark}{Remark}[section]
\newtheorem{example}{Example}[section]
\newtheorem{hypothesis}{Hypothesis}[section]

\newcommand{\set}[2]{\left\{#1 \; :\; #2\right\}} 
\newcommand{\prodin}[2]{\left \langle #1,#2 \right\rangle} 

\DeclareMathOperator*{\diam}{diam} 
\DeclareMathOperator*{\dist}{dist} 
\DeclareMathOperator*{\conv}{conv}
\DeclareMathOperator*{\ran}{ran}
\DeclareMathOperator*{\diver}{div}

\newcommand{\R}{\mathbb R} 
\newcommand{\N}{\mathbb N} 
\newcommand{\F}{\mathcal F}
\renewcommand{\L}{\mathcal L}
\newcommand{\A}{\mathcal A}
\renewcommand{\S}{\mathbb{S}}
\newcommand{\eps}{{\varepsilon}}

\newcommand{\Norm}[1]{\|\hspace{-1pt}|#1\|\hspace{-1pt}| }
\newcommand{\V}{\mathcal V}
\newcommand{\C}{\mathbb C}
\renewcommand{\P}{\mathcal P}

\begin{document}

\title{Inverse problems on low-dimensional manifolds}

\author{Giovanni S. Alberti}
\address{MaLGa Center, Department of Mathematics, University of Genoa, Via Dodecaneso 35, 16146 Genova, Italy.}
\email{giovanni.alberti@unige.it}

\author{\'Angel Arroyo}
\address{MOMAT Research Group, Interdisciplinary Mathematics Institute, Department of Applied Mathematics and Mathematical Analysis, Universidad Complutense de Madrid, 28040 Madrid, Spain.}
\email{ar.arroyo@ucm.es}

\author{Matteo Santacesaria}
\address{MaLGa Center, Department of Mathematics, University of Genoa, Via Dodecaneso 35, 16146 Genova, Italy.}
\email{matteo.santacesaria@unige.it}



\thanks{This work has been carried out at the Machine Learning Genoa (MaLGa) center, Universit\`a di Genova (IT). The authors are members of GNAMPA, INdAM. This material is based upon work supported by the Air Force Office of Scientific Research under award number FA8655-20-1-7027. \'AA is  partially supported by the grants MTM2017-85666-P and 2017 SGR 395.}

\begin{abstract}

We consider abstract inverse problems between infinite-dimen\-sio\-nal Banach spaces. These inverse problems are typically nonlinear and ill-posed, making the inversion with limited and noisy measurements a delicate process. In this work, we assume that the unknown belongs to a finite-dimensional manifold: this assumption arises in many real-world scenarios where natural objects have a low intrinsic dimension and  belong to a certain submanifold of a much larger ambient space. We prove uniqueness and  H\"older and Lipschitz stability results in this general setting, also in the case when only a finite discretization of the measurements is available. Then, a Landweber-type reconstruction algorithm from a finite number of measurements is proposed, for which we prove global convergence, thanks to a new criterion for finding a suitable initial guess. 

These general results are then applied to several examples, including two classical nonlinear ill-posed inverse boundary value problems. The first is Calder\'on's inverse conductivity problem, for which we prove a Lipschitz stability estimate from a finite number of measurements for piece-wise constant conductivities with discontinuities on an unknown triangle. A similar stability result is then obtained for Gel'fand-Calder\'on's problem for the Schr\"odinger equation, in the case of piece-wise constant potentials with discontinuities on a finite number of non-intersecting balls.

\end{abstract}

\maketitle

\setcounter{tocdepth}{1}
\tableofcontents

\section{Introduction}

We consider the problem of inverting the operator equation
\[
F(x) = y,
\]
where $F$ is a possibly nonlinear map between Banach spaces modeling a measurement (forward) operator, $x$ is an unknown quantity to be recovered and $y$ is the measured data.
This serves as a model of well-known inverse problems such as computed tomography, magnetic resonance, ultrasonography, and inverse problems for partial differential equations (PDE).

In an ideal setting, with infinite-precision measurements and no noise, many important inverse problems can be solved -- i.e.\ there is a unique solution -- and stability holds, possibly in a very weak form, for the reconstruction. However, in real applications, the measurements are affected by noise and only a finite-dimensional projection can be acquired. Another issue is the ill-posedness of the map $F$, that can amplify the noise in the measurements if not properly taken into account. This can be observed in the stability estimates, which may be of logarithmic type \cite{alessandrini1988} when no particular assumptions on the unknown are imposed: this is an intrinsic phenomenon of some ill-posed problems \cite{mandache2001} and it reflects in poor numerical reconstructions. On the other hand, Lipschitz and H\"older stability estimates have great impact on applications, since they allow for good numerical reconstructions \cite{dehoop2012}. It is therefore a fundamental question to find explicit conditions on the problem that guarantee good stability properties.

There is a wide literature on Lipschitz stability results for nonlinear ill-posed inverse problems (mostly inverse boundary value problems for PDE) under the assumption that the unknown belongs to a known finite-dimensional subspace or submanifold of a Banach space \cite{2005-alessandrini-vessella,2011-beretta-francini,beretta2013,beretta2015,gaburro2015,beretta2017,alessandrini2017,
alessandrini2017u,alessandrini2018,beretta2019,2020-ruland-sincich,beretta2020preprint}. All these results require infinitely many measurements, even though the number of degrees of freedom to recover are finite. 
Similar results have been obtained with a finite number of measurements, mostly regarding inverse boundary value problems \cite{friedman1989,alberti2017infinite,alberti2018,Harrach_2019,ruland2018,ALB-SAN,harrach2019uniqueness,alberti2020calderon,Liu_2020} and scattering problems in the case when the unknown has a periodic, polygonal or polyhedral structure  \cite{cheng2003,alessandrini2005b,bao2011,hu2016,blaasten2016,blaasten2017,liu2017}.

In this paper, we consider an abstract setting in which most of these works fall in: we assume the 
unknown to belong, or to be close, to a finite-dimensional submanifold $M$ of a Banach space. This setup arises in many real-world scenarios where the objects of interest have a low intrinsic dimensionality compared to the large dimension of the ambient space \cite{2009-peyre}. Priors expressed with manifolds are very popular in machine learning, for example in manifold learning and nonlinear dimensionality reduction \cite{2007-lee-verleysen,2008-lin-zha,2009-baraniuk-wakin,bourgain2015toward}, and especially when machine learning is applied to solving inverse problems \cite{2017-adler-oktem,2017-mccann-jin-unser,2018-lucas-etal,2019-arridge-etal}. In particular, the methods developed in \cite{hyun2020deep} are based on the so-called $\mathcal{M}$-RIP property, which is exactly the Lipschitz stability estimate we derive in this work. We note that low-dimensional manifolds are also used for image denoising \cite{osher2017}.

Under this assumption we are able to obtain a general H\"older and Lipschitz stability result with infinite-precision measurements and also a similar result from a finite number of measurements. 
The main finding is that the ill-posedness of an inverse problem can be mitigated  by imposing \textit{nonlinear} constraints. More precisely, we only require the unknown to belong to a $C^1$ manifold of a Banach space $X$, not necessarily embedded in $X$ but only satisfying some H\"older or Lipschitz embedding properties, as described in the next section. Most of the conditions on the map $F$ and on the $C^1$ manifold are sharp, as we show with many examples and counterexamples. The approach is based on the inverse function theorem, by using the differential of $F$, as in many of the papers mentioned above, which are often based on the use of the Fr\'echet derivative of $F$ composed with the parametrisation of the manifold. Thus, the results of this work may be seen as a generalisation of those ad-hoc methods to the abstract manifold setting.
The particular case in which $M$ is a finite-dimensional subspace was studied in  \cite{ALB-SAN}, by using some ideas appeared in Lipschitz stability results for nonlinear ill-posed problems \cite{bacchelli2006,stefanov-uhlmann-2009,BOU} combined with finite-dimensional approximations used in \cite{alberti2018,harrach2019uniqueness}.

We then apply these general stability results to two classical nonlinear ill-posed problems: Calder\'on's inverse conductivity problem and Gel'fand-Calder\'on problem for the Schr\"odinger equation. Both are inverse boundary value problems that suffer from exponential instability \cite{mandache2001,isaev2011,koch2020}, but there are very few Lipschitz stability results for these problems in case the unknown conductivity/potential belongs to a finite-dimensional manifold \cite{beretta2008,beretta2015,beretta2019,aspri2022lipschitz}. Thanks to our abstract stability estimates, we are now able to include the known results as part of a general framework, where precise conditions on the finite-dimensional manifold are crucial and stability can be obtained also with a discretization of the measurements. 

For the Calder\'on problem, we consider piece-wise constant conductivities with discontinuities on a triangle, with a finite number of measurements. The same argument extends to more general polygons and  is based on a similar stability results in the case where the full infinite-dimensional measurements are available \cite{beretta2019}. We then consider the Gel'fand-Calder\'on problem for piecewise constant potentials with discontinuities on a finite number of non-intersecting balls. To the best of our knowledge, this setting has never been considered, and thus the Lipschitz stability we obtain is new also for the case of full boundary measurements. Another novelty consists in the fact that
not only is the location and size of each ball unknown, but also the values of the potential in each ball are to be determined and they belong to a continuous interval. The key step to achieve this improvement was the use of the Runge approximation property for solutions of elliptic equations.

This work is structured as follows. In \Cref{sec:Lip} we present the framework and the main stability results. Many examples and counterexamples are given to motivate the assumptions. The Lipschitz stability results for Calder\'on's and Gel'fand-Calder\'on's problems from a finite number of measurements are presented in Section \ref{sec:appl}. In  \Cref{sec:rec} we design a globally convergent reconstruction algorithm based on the Lipschitz stability estimate and on \cite{dehoop2012}.
In \Cref{sec:exa} we present several examples of manifolds in Banach spaces that verify the assumptions of the stability estimates. Moreover, as an application, we obtain a Lipschitz stability estimate for the ill-posed inverse problem of differentiation. Sections~\ref{sec:proofs-1} and \ref{sec:pfII} are devoted to the proofs of the main abstract results.
Sections \ref{sec:cald} and \ref{sec:gelf} contain the proofs of the stability results for the two inverse problems presented in Section \ref{sec:appl}. Section~\ref{sec:conclusion} is devoted to concluding remarks and open questions.
In \Cref{sec:tangent} we recall some basic concepts related to the tangent space of a manifold and the differential of maps between manifolds. Finally, \Cref{sec:estimates} contains some technical estimates regarding the size of the symmetric difference between balls in $\R^d$.

\section{Abstract stability results} \label{sec:Lip}

Let $X$ and $Y$ be Banach spaces with norms $\|\cdot\|_X$ and $\|\cdot\|_Y$, and let $A\subseteq X$ be an open set. In this work, we focus on solving the inverse problem
\[
F(x)=y,
\]
where $F\colon A\to Y$ is the forward map, possibly nonlinear. As discussed in the introduction, we assume $x\in M$, for a certain known manifold $M\subseteq X$, as detailed below.

We always assume that $F$ is  \emph{Fr\'echet differentiable}, namely,  for every $x\in A$ 
there exists a bounded linear operator $F'(x)\in\L(X,Y)$ such that
\begin{equation*}
				\lim_{y\to x}\frac{\|F(y)-F(x)-F'(x)(y-x)\|_Y}{\|y-x\|_X}
				=
				0,
\end{equation*}
and that the map $F'\colon A\to\L(X,Y)$ is continuous: in short, we write $F\in C^1(A,Y)$.

It is worth clarifying why the results of this paper are most meaningful for low-dimensional manifolds $M$. Indeed, they will involve a stability constant $C>0$ and, for those in \Cref{sub:finite}, a parameter $N\in\N$ quantifying the number of measurements needed for stability. Both these constants depend on the manifold $M$, and  grow as $\dim M$ grows, for ill-posed problems;  the dependence on $\dim M$ is typically of exponential type for severely ill-posed problems \cite{rondi2006,beretta2013,beretta2017,alberti2018,ALB-SAN}. Therefore, since in practice small values of $C$ and $N$ are needed, these estimates are interesting for low-dimensional manifolds $M$.

\subsection{Stability with infinite measurements}\label{subsec:infinity-meas}

Our starting point is a Lipschitz stability result that can be found in \cite[Proposition 5]{bacchelli2006}, \cite[Theorem~2]{stefanov-uhlmann-2009} and \cite[Theorem 2.1]{BOU}, where sufficient conditions for Lipschitz stability are provided in the case in which $M=W\subseteq X$ is an $n$-dimensional subspace and $x\in K$, where $K\subseteq W\cap A$ is a known compact and convex subset.  
However, the convexity of $K$ is not a reasonable assumption towards a generalization of this result to a differentiable manifold $M\subseteq X$. In our first result we show that the convexity of $K$ can actually be dropped from the assumptions of \cite[Theorem 2.1]{BOU}. Indeed, the property of a line segment being included in $K$ is needed only in the case in which the endpoints are close enough. The following result is then obtained by constructing a suitable neighbourhood of $K$ which contains all line segments up to a certain length.

\begin{theorem}\label{THM infinite measurements without compactness}
Let $X$ and $Y$ be Banach spaces, $A\subseteq X$ be an open set, $W\subseteq X$ be an $n$-dimensional subspace and $K\subseteq W\cap A$ be a compact subset. Consider a Fr\'echet differentiable map $F\in C^1(A,Y)$ satisfying that
\begin{enumerate}
\item $F\big|_K$ is injective;
\item $F'(x)\in\mathcal{L}(X,Y)$ is injective on $W$ for every $x\in W\cap A$.
\end{enumerate}
Then there exists $C>0$ such that
\begin{equation}\label{Lip-stability-1}
				\|x-y\|_X \le C  \|F(x)-F(y)\|_Y,\qquad x,y\in K.
\end{equation}
\end{theorem}

It is worth observing that the assumption on the continuity of $F'$ may not be dropped, as shown in the following example.
\begin{example}
Let us take $A=X=Y=\R$, $K=[0,1]$ and $F=f\colon\R\to\R$ the function defined as
\begin{equation*}
		f(x)
		:\,=
		\begin{cases}
			\displaystyle x^2\Big(\operatorname{sign}(x)+\sin\frac{1}{x}\Big)+x
			&
			\text{ if } x\neq 0,
			\\
			0
			&
			\text{ if } x=0.
		\end{cases}
\end{equation*}
It turns out that $f$ is continuous and differentiable, but $f'$ is not continuous at $x=0$. More precisely, we have
\begin{equation*}
		f'(x)
		=
		2x\Big(\operatorname{sign}(x)+\sin\frac{1}{x}\Big)-\cos\frac{1}{x}+1,\quad x\neq 0,\quad\qquad 		f'(0)
		=
		\lim_{x\to 0}\frac{f(x)}{x}
		=
		1.
\end{equation*}
Then $f'(x)>0$ for every $x\in \R$ and, in particular, $f$ is injective. However, \eqref{Lip-stability-1} does not hold. Indeed, for $x_k=\frac{1}{2k\pi}$ we have that
\begin{equation*}
		\lim_{k\to\infty}f'(x_k)
		=
		\lim_{k\to\infty}\frac{1}{k\pi}
		=
		0
	.
\end{equation*}
If $f$ were Lipschitz stable, there would exist a constant $C>0$ such that $f'(x)\ge C$ for every $x$, a contradiction.
\end{example}

Next, we introduce some concepts that will be needed for the generalization of the Lipschitz stability estimate to manifolds. We start by the definition of a differentiable manifold in a Banach space.

\begin{definition}\label{def:manifold}
Let $X$ be a Banach space. We say that $M\subseteq X$ is an \emph{$n$-dimensional differentiable manifold in $X$} if there exists an \emph{atlas} $\{(U_i,\varphi_i)\}_{i\in I}$, with $U_i\subseteq M$, $\bigcup_{i\in I} U_i = M$ and $\varphi_i\colon U_i\to\R^n$, such that
\begin{enumerate}
\item  for every $i\in I$, $U_i$ and $\varphi_i(U_i)$ are open sets, with respect to the topologies of $M$ inherited from $X$ and of $\R^n$, respectively;
\item for every $i\in I$, $\varphi_i\colon U_i\to\R^n$ is a homeomorphism onto its image $\varphi_i(U_i)$;
\item for every $i,j\in I$, the transitions maps
\begin{equation*}
	\varphi_j\circ\varphi_i^{-1}\colon\varphi_i(U_i\cap U_j)\to\varphi_j(U_i\cap U_j)
\end{equation*}
are continuously differentiable.
\end{enumerate}
\end{definition}

Therefore, given $x\in M$ and $i\in I$ such that $U_i\ni x$, the map $\varphi_i^{-1}$ can be seen as a parametrization of $M$ in a neighbourhood of $x$. We shall denote the tangent space of $M$ at $x\in M$ by $T_x M$. For the precise definition, see \Cref{sec:tangent}.

Note that $M$ is always a topological submanifold of $X$. However, in general $M$ is not a differentiable submanifold of $X$, since the two differential structures may not be compatible, meaning that the tangent space $T_x M$ may not be contained in $X$. In other words, the differentiability of $M$ has to be understood in an intrinsic sense, independently of the structure of the ambient space where $M$ lies. In fact, it is easy to construct differentiable manifolds which are not differentiable as objects contained in a given Banach space $X$ (e.g., $M=\{(x,|x|):x\in\R\}\subseteq\R^2=X$ or see the example in \Cref{example-1}).
Therefore, as it will be made clear later,  we will need to require certain regularity with respect to the norm in $X$.

\begin{definition}\label{Lipschitz in X}
Let $X$ be a Banach space and $M\subseteq X$ be an $n$-dimensional differentiable manifold. We say that $M$ is \emph{$\alpha$-H\"older in $X$}, with $\alpha\in (0,1]$, if there exists $\ell\in(0,1]$ such that, for every $i\in I$, the map $\varphi_i^{-1}\colon\varphi_i(U_i)\to M$ is $\alpha$-H\"older continuous with respect to $\|\cdot\|_X$, that is
\begin{equation}\label{holder parametrization}
	\frac{|\varphi_i(x)-\varphi_i(y)|^\alpha}{\|x-y\|_X}
	\ge
	\ell,\qquad x,y\in U_i.
\end{equation}
If $\alpha=1$, we say that $M$ is Lipschitz in $X$.
\end{definition}

For example, the manifold $\{(x,|x|):x\in\R\}$ introduced above is Lipschitz in $\R^2$.

Next, we recall the definition of differentiability of a function on a differentiable manifold.

\begin{definition}\label{F in M diffble}
Let $X$ and $Y$ be Banach spaces and $M\subseteq X$ be an $n$-dimensional differentiable manifold. We say that a mapping $F\colon M\to Y$ is \emph{differentiable} if $F\circ\varphi_i^{-1}\colon\varphi_i(U_i)\to Y$ is Fr\'echet differentiable for every $i\in I$. In addition, if the maps $(F\circ\varphi_i^{-1})'\colon\varphi_i(U_i)\to\L(\R^n,Y)$ are continuous for every $i\in I$, we write $F\in C^1(M,Y)$.
\end{definition}

Note that, even though $F$ is not necessarily Fr\'echet differentiable from $M$ to $Y$, it is always continuous, since $M$ is a topological submanifold of $X$.

We now state the main H\"older and Lipschitz stability result of this subsection.

\begin{theorem}\label{MAINTHM1}
Let $X$ and $Y$ be Banach spaces, $\alpha\in (0,1]$, $M\subseteq X$ be an $n$-dimensional differentiable manifold  $\alpha$-H\"older in $X$ and $K\subseteq M$ be a compact set. Consider a differentiable map $F\in C^1(M,Y)$ satisfying that
\begin{enumerate}
\item $F$ is injective;
\item the differential $dF_x\colon T_xM\to Y$ is injective for every $x\in M$.
\end{enumerate}
Then there exists a constant $C>0$ such that
\begin{equation}\label{Lip-stability}
				\|x-y\|_X
				\le C
			\|F(x)-F(y)\|_Y^\alpha,
				\qquad x,y\in K.
\end{equation}
\end{theorem}

Observe that the second assumption in the previous theorem is a hypothesis regarding the differential $dF_x$ (see \Cref{sub:diff}). This means that the condition on $F$ is tied to the structure of the manifold where the function is defined.
The differentiability of $F$ may be uncoupled from the differential structure of $M$ when the manifold $M$ is embedded in $X$, as we now discuss.

\begin{remark}\label{rem:embedded}
We say that an $n$-dimensional differentiable manifold $M$ is \emph{embedded} in $X$ if the inclusion map $M\hookrightarrow X$ is an embedding between differentiable manifolds, that is, if $\varphi_i\colon U_i\to\varphi_i(U_i)$ is a diffeomorphism for every $i\in I$, namely, $\varphi_i^{-1}\colon\varphi_i(U_i)\to X$ is continuously differentiable and $\bigl(\varphi_i^{-1}\bigr)'(\varphi_i(x))\colon\R^n\to X$ is injective for every $x\in U_i$. In other words, $M$ is embedded in $X$ if it inherits the differential structure of $X$. 
It is worth observing that a manifold $M$ may be Lipschitz in $X$ but not be embedded in $X$, as e.g.\ $\{(x,|x|):x\in\R\}\subseteq\R^2$, or see \Cref{ex:Lip-emb} below.

Thanks to a simple calculation (see \Cref{sub:diff}), when $M$ is embedded in $X$ and $F$ is Fr\'echet differentiable, the differential of $F$ coincides with its Fr\'echet derivative restricted  to $T_xM$, which may be identified with a subspace of $X$. More precisely, we have
\begin{equation*}
dF_x = F'(x)|_{T_x M}.
\end{equation*}
Thus, when $M$ is embedded in $X$, assumption (2) in \Cref{MAINTHM1} may be replaced by the condition
\[
\textit{$F\in C^1(A,Y)$ and $F'(x)$ is injective on $T_x M\subseteq X$ for every $x\in M$,}
\]
where $A\supseteq M$ is an open set of $X$.
\end{remark}

The following result is a consequence of \Cref{MAINTHM1} and \Cref{rem:embedded}.
\begin{corollary}\label{cor:F(M)embedded}
Under the assumptions from \Cref{MAINTHM1}, if $\{(U_i,\varphi_i)\}$ is an atlas for $M$, then $F(M)\subseteq Y$ is an $n$-dimensional differentiable manifold with the atlas $\{(F(U_i),\varphi_i\circ F^{-1})\}$. Furthermore, $F(M)$ is embedded in $Y$ and $dF_x\colon T_xM\to T_{F(x)}F(M)$.
\end{corollary}
\begin{proof}
The H\"older stability estimate \eqref{Lip-stability} yields that $F\colon M\to F(M)\subseteq Y$ is a homeomorphism. Moreover, since $\{(U_i,\varphi_i)\}$ is an atlas for $M$ satisfying the conditions from \Cref{def:manifold}, then $F\circ\varphi_i^{-1}$ is also an homeomorphism, and we can define an atlas for $F(M)$ by $\{(F(U_i),\varphi_i\circ F^{-1})\}$ so that $F(M)$ is an $n$-differentiable manifold in $Y$.
By \Cref{rem:embedded}, to see that $F(M)$ is embedded in $Y$ we just need to check that $\varphi_i\circ F^{-1}\colon F(U_i)\to \varphi_i(U_i)$ is a diffeomorphism, that is, $F\circ\varphi_i^{-1}\colon\varphi_i(U_i)\to Y$ is continuously differentiable and $(F\circ\varphi_i^{-1})'(\varphi_i(x))\colon\R^n\to Y$ is injective for every $x\in U_i$, which are granted by the hypothesis $F\in C^1(M,Y)$ (\Cref{F in M diffble}) and the fact that $dF_x$ is injective by assumption, respectively.
\end{proof}

In the following example we show that, if the manifold is not embedded, the Fr\'echet differentiability of $F$ alone is not a sufficient assumption, even if $F'(x)$ is injective on the whole $X$ for every $x\in M$.

\begin{example}\label{ex:Lip-emb}
Let $M=\{\chi_{[t,t+1]}\,:\,t\in\R\}\subseteq L^1=L^1(\R)$ and consider the function $\varphi\colon M\to\R$ defined as $\varphi(\chi_{[t,t+1]})=t$ for every $t\in\R$. We first check that the conditions from \Cref{def:manifold} are satisfied so that $M$ is a $1$-dimensional differentiable manifold together with the atlas $\{(M,\varphi)\}$. Since $\varphi^{-1}(t)=\chi_{[t,t+1]}$, a direct computation shows that
\begin{equation}\label{isometry}
		\|\varphi^{-1}(s)-\varphi^{-1}(t)\|_{L^1}
		=
		2\min(1,|s-t|).
\end{equation}
Thus, $\varphi$ is a homeomorphism. Finally, we just simply observe that, since the atlas associated to $M$ has only one chart, the unique transition map $\varphi\circ\varphi^{-1}$ is the identity operator on $\R$, so $M$ is a differentiable manifold. Furthermore, in view of \Cref{Lipschitz in X}, by \eqref{isometry} it turns out that $M$ is Lipschitz in $L^1$. However, $M$ is not embedded in $L^1$ because $\varphi^{-1}\colon\R\to L^1$ is not differentiable.

Now we consider a function $F\colon L^1\to L^1$ satisfying all the hypotheses in \Cref{MAINTHM1} except the injectivity of $(F\circ\varphi^{-1})'(\varphi(x))\in\L(\R,L^1)$, which is replaced by the injectivity of $F'(x)$ in $L^1$, but failing to be $\alpha$-H\"older stable for every $\alpha\in(0,1]$.

Let $g\in L^\infty(\R)$ be the $1$-periodic function defined by
\[
g(t)=
e^{-\frac{1}{t}}
,\qquad t\in (0,1].
\]
Then we consider $F(u):\,=gu$ for each $u\in L^1$. It is easy to check that $F$ is an injective linear bounded map from $L^1$ to $L^1$. Thus $F$ belongs to $C^1(\R,L^1)$ and, by linearity, the Fr\'echet derivative of $F$ coincides with $F$, i.e. $F'(u)\equiv F$ for every $u\in L^1$. However, $F$ is not $\alpha$-H\"older stable in $M$. Indeed, by the linearity of $F$ we have
\begin{equation*}
		F(\varphi^{-1}(t))-F(\varphi^{-1}(0))
		=
		F\big(\varphi^{-1}(t)-\varphi^{-1}(0)\big)
		=
		g\cdot\big(\chi_{[t,t+1]}-\chi_{[0,1]}\big)
\end{equation*}
for every $t\in(0,1)$, and taking the $L^1$-norm we obtain
\begin{equation*}
		\|F(\varphi^{-1}(t))-F(\varphi^{-1}(0))\|_{L^1}
		=
		\int_0^tg(s)\ ds+\int_1^{1+t}g(s)\ ds
		=
		2\int_0^t g(s)\,ds
		\le
		2te^{-\frac{1}{t}}.
\end{equation*}
The last term is infinitesimal of infinite order as $t\to 0^+$ and so, by \eqref{isometry},
\begin{equation*} 
		\frac{\|F(\varphi^{-1}(t))-F(\varphi^{-1}(0))\|_{L^1}^\alpha}{\|\varphi^{-1}(t)-\varphi^{-1}(0)\|_{L^1}}
		\le
		2^{\alpha-1}\frac{e^{-\frac{\alpha}{t}}}{t^{1-\alpha}}\xrightarrow[t\to 0^+]{}0.
\end{equation*}
This proves that $F$ is not $\alpha$-H\"older stable in $M$, namely, \eqref{Lip-stability} is not satisfied.
\end{example}

\subsection{Stability with finite measurements}\label{sub:finite}
We now consider the case where, instead of the infinite-dimensional measurements $F(x)$, we have at our disposal only a finite-dimensional approximation. As in \cite{ALB-SAN}, inspired by the theory of regularization by projection \cite{scherzer_book}, we write the new measurements as $Q_N F(x)$ for a suitable operator $Q_N$.

\begin{hypothesis}\label{Q_N}
For each $N\in\N$, let $Q_N\colon Y\to Y$ be a bounded linear map. Assume that there exist a subspace $\widetilde Y\subseteq Y$ and $D>0$ such that
\begin{enumerate}
\item\label{ass:1} $\|Q_N\big|_{\widetilde Y}\|_{\L(\widetilde Y,Y)}\le D$ for every $N\in\N$;
\item\label{ass:2} $Q_N\big|_{\widetilde Y}\to I_{\widetilde Y}$ as $N\to\infty$ with respect to the strong operator topology, i.e.
\begin{equation*}
		\lim_{N\to\infty}\|y-Q_Ny\|_Y
		=
		0
\end{equation*}
for every $y\in\widetilde Y$.
\end{enumerate}
\end{hypothesis}

Note that condition \eqref{ass:1} is implied by \eqref{ass:2} when $\widetilde Y$ is a closed subspace of $Y$, by the uniform boundedness principle.

Let us now list some examples of operators $Q_N$.\smallskip

Any family of bounded operators $Q_N$ such that $Q_N\to I_Y$ strongly may be considered (in particular, we do not require convergence with respect to the operator norm). This situation can arise in practice when $Y$ is an infinite-dimensional separable Hilbert space and $Q_N$ is the orthogonal projection onto the space spanned by the first $N$ elements of an orthonormal basis (see \cite[Example 1]{ALB-SAN} for more details).

Another example, that is adapted to inverse boundary value problems, is when the measurements themselves are operators. In this case, the maps $Q_N$ do not converge strongly to the identity on the whole $Y$. Let us present it in more details.

\begin{example}\label{ex:gal}
Let $Y = \L (Y^1,Y^2)$ be the space of bounded linear operators from $Y^1$ to $Y^2$, where $Y^1$ and $Y^2$ are Banach spaces. Let  $P^k_N \colon Y^k \to Y^k$ be bounded maps such that $P^{2}_N \to I_{Y^k}$ and $(P^{1}_N)^* \to I_{Y^k}$ strongly as $N\to +\infty$. 

Let $\widetilde Y=\{T\in Y:\text{$T$ is compact}\}$ and $Q_N \colon Y \to Y$  be the maps defined by
\[
Q_N (y) = P^2_N y P^1_N,\qquad y \in Y.
\]
Then, even if $Q_N \not\to I_Y$ strongly,  Hypothesis~\ref{Q_N} is satisfied. We refer to \cite[Example 2]{ALB-SAN} for full details.
\end{example}

We now present the following Lipschitz stability result with finite measurements in the case of linear subspaces, which can be found in \cite[Theorem 2]{ALB-SAN} under the additional assumption that $K$ is convex.

\begin{theorem}\label{FM-1}
Let $X$ and $Y$ be Banach spaces, $A\subseteq X$ be an open set, $W\subseteq X$ be an $n$-dimensional subspace, $K\subseteq W\cap A$ be a compact set and  $Q_N\colon Y\to Y$ be bounded linear maps satisfying \Cref{Q_N}. Consider a Fr\'echet differentiable map $F\in C^1(A,Y)$ such that
\begin{enumerate}
\item $F(x)-F(y) \in \widetilde Y$ for every $x,y \in K$;
\item $\ran(F'(x)\big|_W)\subseteq\widetilde Y$ for every $x\in W\cap A$, where $\ran$ denotes the range;
\item the Lipschitz stability estimate
\begin{equation*}
				\|x-y\|_X\le C\|F(x)-F(y)\|_Y
			,\qquad x,y\in K,
\end{equation*}
is satisfied for some $C>0$. 
\end{enumerate}
Then $Q_NF$ satisfies the Lipschitz stability estimate
\begin{equation}\label{LS-QNF}
				\|x-y\|_X\le 2c_KC \|Q_NF(x)-Q_NF(y)\|_Y
				,\qquad x,y\in K,
\end{equation}
for some $c_K>0$ depending only on $K$ and every  sufficiently large $N\in\N$. If $K$ is convex, we can choose $c_K=1$. In the general case, we have
\[
c_K=\frac{\diam K}{\delta_K},
\]
where $\delta_K$ is the constant given in Lemma~\ref{delta-K} below.
\end{theorem}

\begin{remark}\label{rem:N}
It is worth to note that the values of $N\in\N$ for which the Lipschitz stability estimate in the previous theorem holds depend on the a priori data explicitly. More precisely, as it will be clear from the proof, the inequality \eqref{LS-QNF} is true for those $N\in\N$ satisfying
\begin{align*}
&	\sup_{\xi,\eta\in K}\|(I_Y-Q_N)(F(\xi)-F(\eta))\|_Y\le\frac{\delta_K}{2C},\\
&	\sup_{\xi\in\widehat K}\|(I_Y-Q_N)F'(\xi)\|_{\L(W,Y)}
	\le
	\frac{1}{2C},
\end{align*}
where $\widehat K=\{\xi\in W: \dist_X(\xi,K)\le\delta_K\}=\overline{B_X(K,\delta_K)}\cap W$ is the compact neighbourhood of $K$ constructed  in \Cref{delta-K} below. If $K$ is convex, the first of these two inequalities may be dropped.
\end{remark}

We now consider the case of a manifold.

\begin{theorem}\label{FM-2}
Let $X$ and $Y$ be Banach spaces, $M\subseteq X$ be an $n$-dimensional differentiable manifold, $K\subseteq M$ be a compact set, $Q_N\colon Y\to Y$ be bounded linear maps satisfying \Cref{Q_N} and $C>0$.
Assume in addition that $M$ is $\alpha$-H\"older in $X$ for some $\alpha\in (0,1]$ with constant $\ell\in(0,1]$ as in \eqref{holder parametrization}. 
Consider a differentiable map $F\in C^1(M,Y)$ such that
\begin{enumerate}
\item $F(x)-F(y) \in \widetilde Y$ for every $x,y \in K$;
\item $\ran(dF_x)\subseteq\widetilde Y$ for every $x\in M$;
\item and either
\begin{itemize}
\item[(3a)] $\varphi_i$ is $\ell^{-1}$-Lipschitz for every  $i\in I$ and
\begin{equation*}
				\|x-y\|_X\le C\|F(x)-F(y)\|_Y
				,\qquad x,y\in K,
\end{equation*} 
\item[(3b)] \label{3b} or for every $i\in I$
\begin{equation*}
	|\varphi_i(x)-\varphi_i(y)|\le \ell^{-1}C \|F(x)-F(y)\|_Y,\qquad x,y\in U_i\cap K.
\end{equation*}
\end{itemize}
\end{enumerate}
Then $Q_NF$ satisfies the H\"older stability estimate
\begin{equation*}
				\|x-y\|_X\le c_{K,M} (2C)^\alpha  \|Q_NF(x)-Q_NF(y)\|^\alpha_Y
				,\qquad x,y\in K,
\end{equation*}
for some $c_{K,M}>0$ depending only on $K$ and $M$ and every sufficiently large $N\in\N$.
\end{theorem}

Note that it is possible to obtain a constructive estimate for the parameter $N$. This involves conditions similar to the one given in \Cref{rem:N}, but related to the composition of $F$ with the charts. For the sake of exposition we decided not to include this expression here;  it can be easily obtained from the proof.

\begin{remark}
As in \Cref{MAINTHM1} and \Cref{rem:embedded}, in the case when $M$ is embedded in $X$ and $F$ is Fr\'echet differentiable, assumption (2) in \Cref{FM-2} may be replaced by the condition
\[
\textit{$F\in C^1(A,Y)$ and $\ran(F'(x)\big|_{T_xM})\subseteq\widetilde Y$  for every $x\in M$,}
\]
where $A\supseteq M$ is an open set of $X$.
\end{remark}

\begin{remark}\label{rem:frechet}
In the case in which $\widetilde Y\subseteq Y$ is a closed subspace, the condition $\ran\{dF_x\}\subseteq\widetilde Y$ is a consequence of the condition
\begin{equation}\label{eq:difference in Y}
   F(x)-F(y) \in \widetilde Y,\qquad  x,y \in M.
\end{equation}
Indeed, fix $x\in M$ and $i\in I$ such that $x\in U_i$. Then
\begin{equation*}
	dF_xh
	=
	(F\circ\varphi_i^{-1})'(\varphi_i(x))h
	=
	\lim_{t\to 0}\frac{F\circ\varphi_i^{-1}(\varphi_i(x)+th)-F\circ\varphi_i^{-1}(\varphi_i(x))}{t}
\end{equation*}
for every $h\in\R^n$. Observe that the right-hand side is the limit of elements in $\widetilde Y$. Since $\widetilde Y$ is closed, then the limit is also in $\widetilde Y$, so $\ran\{dF_x\}\subseteq\widetilde Y$.
\end{remark}

As a combination of \Cref{MAINTHM1} and \Cref{FM-2} we also derive the following result.

\begin{theorem}\label{QNF-stability}
Let $X$ and $Y$ be Banach spaces, $M\subseteq X$ be an $n$-dimensional differentiable manifold  Lipschitz in $X$, $K\subseteq M$ be a compact set and $Q_N\colon Y\to Y$ be bounded linear maps satisfying \Cref{Q_N}. Consider a differentiable map $F\in C^1(M,Y)$ satisfying that
\begin{enumerate}
\item $F$ is injective;
\item the differential $dF_x\colon T_xM\to Y$ is injective for every $x\in M$;
\item $F(x)-F(y) \in \widetilde Y$ for every  $x,y \in K$;
\item $\ran(dF_x)\subseteq\widetilde Y$ for every $x\in M$.
\end{enumerate}
Then $Q_NF$ satisfies the Lipschitz stability estimate
\begin{equation}\label{QNF-stability-estimate}
			\|x-y\|_X \le C	\|Q_NF(x)-Q_NF(y)\|_Y
								,\qquad x,y\in K,
\end{equation}
for some $C>0$ and every sufficiently large $N\in\N$.
\end{theorem}

\begin{remark}\label{ref:deriv_compact}
Assume that $Y = \L (Y^1,Y^2)$, with $Y^1$ and $Y^2$ Banach spaces, $Q_N$ is given as in Example~\ref{ex:gal} and
\begin{equation*}
dF_x h \colon Y^1\to Y^2\;\; \text{is compact for every $x \in M, h \in T_x M$.}
\end{equation*}
In view of \Cref{rem:frechet}, this assumption may be replaced by
\begin{equation*}
F(x_1)-F(x_2)\;\; \text{is compact for every $x_1,x_2 \in M$.}
\end{equation*}
Then Hypothesis~\ref{Q_N} is satisfied with $\widetilde Y=\{T\in Y:\text{$T$ is compact}\}$ and by Theorem~\ref{QNF-stability} we obtain a Lipschitz stability estimate from a finite number of measurements.
\end{remark}

It is worth observing that \Cref{MAINTHM1,FM-2,QNF-stability} can be readily extended to the case with a finite number of pairwise disjoint manifolds, possibly with different dimensions. For example, the extension of \Cref{QNF-stability} reads as follows.

\begin{theorem}\label{QNF-stability-multiple}
Let $X$ and $Y$ be Banach spaces, $Q_N\colon Y\to Y$ be bounded linear maps satisfying \Cref{Q_N} and $P\in\N$. For $p=1,\dots,P$, let  $M_p\subseteq X$ be  $n_p$-dimensional differentiable manifolds  Lipschitz in $X$ and $K_p\subseteq M_p$ be pairwise disjoint compact sets. Set $M=\cup_p M_p$ and $K=\cup_p K_p$. Consider a  map $F\colon M\to Y$  satisfying for every $p=1,\dots,P$ that
\begin{enumerate}
\item $F\in C^1(M_p,Y)$;
\item $F$ is injective;
\item the differential $dF_x\colon T_xM_p\to Y$ is injective for every $x\in M_p$;
\item $F(x)-F(y) \in \widetilde Y$ for every  $x,y \in K$;
\item $\ran(dF_x)\subseteq\widetilde Y$ for every $x\in M_p$.
\end{enumerate}
Then $Q_NF$ satisfies the Lipschitz stability estimate
\begin{equation}\label{QNF-stability-estimate2}
			\|x-y\|_X \le C	\|Q_NF(x)-Q_NF(y)\|_Y
								,\qquad x,y\in K,
\end{equation}
for some $C>0$ and every sufficiently large $N\in\N$.
\end{theorem}

\subsection{Stability when $x,y \notin K$}
It is easy to see that all H\"older and Lipschitz stability estimates derived in this section may be easily extended to deal with the case $x,y\notin K$, which models mismodeling errors, namely, the case when $d(x,K)$ and $d(y,K)$ are small but possibly nonzero.

Suppose that $F\colon A\subseteq X\to Y$ is Lipschitz and satisfies \eqref{eq:difference in Y} and the H\"older stability estimate
\begin{equation*}
			\|x-y\|_X \le c	\|Q_NF(x)-Q_NF(y)\|_Y^\alpha
								,\qquad x,y\in K,
\end{equation*}
where $K$ is a compact set (possibly a subset of a manifold $M$).
Take now $x,y\in A$ such that
\begin{equation}\label{eq:distance-delta}
d(x,K)\le\delta,\qquad d(y,K)\le \delta,
\end{equation}
for some $\delta\in [0,1]$. We claim that
\begin{equation}\label{eq:mismod}
    \|x-y\|_X\le c \|Q_NF(x)-Q_NF(y)\|_Y^\alpha + 2(1+c(DL)^\alpha)\delta^\alpha,
\end{equation}
where $L>0$ denotes the Lipschitz constant of $F$.

Indeed, let $x_K,y_K\in K$ be such that $d(x,K)=\|x-x_K\|_X$ and $d(y,K)=\|y-y_K\|_X$, whose existence follows by the compactness of $K$. We readily derive
\begin{equation*}
    \begin{split}
     \|x-y\|_X&\le \|x-x_K\|_X+\|x_K-y_K\|_X+\|y_K-y\|_X\\
     &\le c	\|Q_NF(x_K)-Q_NF(y_K)\|_Y^\alpha + d(x,K)+ d(y,K).
    \end{split}
\end{equation*}
By using the fact that $t\mapsto t^\alpha$ is subadditive, together with condition \eqref{ass:1} from Hypothesis~\ref{Q_N}, we obtain
\[
    \|Q_NF(x_K)-Q_NF(y_K)\|_Y^\alpha \le (DL)^\alpha (\|x-x_K\|_X^\alpha 
    +\|y-y_K\|_X^\alpha) +\|Q_NF(x)-Q_NF(y)\|_Y^\alpha.
\]
Altogether, we have
\begin{multline*}
   \|x-y\|_X\le c \|Q_NF(x)-Q_NF(y)\|_Y^\alpha+ d(x,K)+ d(y,K) \\
   + c(DL)^\alpha(d(x,K)^\alpha+ d(y,K)^\alpha)
\end{multline*}
for every $x,y\in A$. Thus, by using \eqref{eq:distance-delta}, \eqref{eq:mismod} immediately follows.

\section{Stability for two nonlinear inverse problems}\label{sec:appl}

In this section we apply the abstract stability results to two severely ill-posed nonlinear inverse boundary value problems: the Calder\'on problem and the Gel'fand-Calder\'on problem. We present here the main results and leave all the proofs to  Sections~\ref{sec:cald} and \ref{sec:gelf}.

\subsection{The Calder\'on problem with a triangular inclusion} \label{sub:cald}
In this section we focus on the Calder\'on problem \cite{calderon1980,2002-borcea,2009-uhlmann}, which is the mathematical model for electrical impedance tomography  (EIT) \cite{eit-1999}.

Consider a bounded Lipschitz domain $\Omega \subseteq \R^d$, with $d\ge 2$, equipped with an electrical conductivity $\sigma \in L_+^{\infty}(\Omega)$, where
\[
L_+^{\infty}(\Omega):=\{f\in L^\infty(\Omega):f \ge \lambda \text{ a.e.\ in $\Omega$, for some $\lambda>0$}\}.
\]
The corresponding Dirichlet-to-Neumann (DN) or voltage-to-current map is the operator $\Lambda_{\sigma}\colon H^{1/2}(\partial \Omega) \to H^{-1/2}(\partial \Omega)$, defined by
\begin{equation}
\Lambda_{\sigma}(f) = \sigma \left. \frac{\partial u_\sigma^f}{\partial \nu}\right|_{\partial \Omega},
\end{equation}
where $\nu$ is the unit outward normal to $\partial \Omega$ and $u^f_\sigma\in H^1(\Omega)$ is the unique weak solution of the Dirichlet problem for the conductivity equation
\begin{equation}\label{schr}
\left\{
\begin{array}{ll}
-\nabla \cdot (\sigma \nabla u^f_\sigma) = 0 \quad &\text{in } \Omega,\\
u^f_\sigma =f \quad &\text{on } \partial \Omega.
\end{array}
\right.
\end{equation}
The following inverse boundary value problem arises from this framework, see \cite{calderon1980,2002-borcea,2009-uhlmann} and references therein.

\vspace*{.2cm}

{\bf Inverse conductivity problem.} Given $\Lambda_{\sigma}$, find $\sigma$ in $\Omega$.

\vspace*{.2cm}

 It is well known that the knowledge of $\Lambda_\sigma$ determines $\sigma$ uniquely if $d=2$ \cite{Nachman1996,astala2006} or if $\sigma$ is smooth enough \cite{Sylvester1987,haberman2015,caro2016}. The inverse problem is severely ill-posed, and only logarithmic stability holds true \cite{alessandrini1988,mandache2001,clop2010,caro2013}.

In recent years several Lipschitz stability estimates have been obtained for this inverse problem under certain a priori assumptions on $\sigma$, such as for $\sigma$ piecewise constant \cite{2005-alessandrini-vessella}, piecewise linear \cite{alessandrini2017}, for $\sigma$ belonging to a finite-dimensional subspace of piecewise analytic functions \cite{Harrach_2019,eberle2021lipschitz} (see \cite{gaburro2015,alessandrini2017u} for the anisotropic case), or under nonlinear assumptions, namely for $\sigma$ piece-wise constant  on polygons or polygonal partitions with a known upper bound on the number of (unknown) vertices/edges \cite{beretta2017differentiability,beretta2019,beretta2020preprint,aspri2022lipschitz}. In all these cases, full boundary measurements are required, i.e.\ all possible combinations of current/voltage data.

When only a finite number of measurements is available, several results have been recently obtained \cite{alberti2018,Harrach_2019,ruland2018,ALB-SAN,alberti2020calderon} where the conductivity is always assumed to lie in a certain  finite-dimensional \textit{linear} subspace $W$ of $L^\infty(\Omega)$.
Related works consider the problem of locating small inhomogeneities inside the medium, which again can be seen as a way of reducing the dimensionality of the unknown space \cite{bruhl-hanke-vogelius-2003,ammari-kang-2004,ammari-kang-2007}.\medskip

We now show how the results of Section~\ref{sub:finite} can be used to derive Lipschitz stability estimates for EIT under \textit{nonlinear} assumptions with finite measurements.

For this, it is useful to denote  the family of triangles in $\R^2$ by $\triangle^2$, namely,
\begin{equation*}
	\triangle^2
	:\,=
	\{\conv\{v_0,v_1,v_2\}\,:\,v_0,v_1,v_2\in\R^2,\ \det(v_1-v_0,v_2-v_0)\neq 0\},
\end{equation*}
where $\conv S$ stands for the convex hull of a set $S$, i.e.\ the smallest set containing all convex combinations of elements in $S$. The condition $\det(v_1-v_0,v_2-v_0)\neq 0$ in the definition of $\triangle^2$ ensures that the triangles are not degenerate.\smallskip

Let us now consider the following two-dimensional setup.
\begin{itemize}
\item $\Omega \subset \R^2$ is a bounded Lipschitz domain.
\item $X = L^1(\Omega)$.
\item $Y = \L(H^{\frac 1 2}(\partial \Omega), H^{-\frac 1 2}(\partial \Omega))$.
\item $M = \{ \sigma_T = 1 + (k-1)\chi_{T} : T \in \triangle^2, T \subseteq \Omega, \dist(T,\partial \Omega) > d_0/2\}$ for some fixed $d_0 > 0$ and $k >0$, $k \neq 1$.
\item $K = \{\sigma_T \in M : \dist(T,\partial \Omega) \ge d_0, |T| \ge d_1 \}$, for some $d_1 >0$.
\item $P^1_N\colon H^{\frac 1 2}(\partial \Omega) \to H^{\frac 1 2}(\partial \Omega)$ and $P^2_N\colon H^{-\frac 1 2}(\partial \Omega) \to H^{-\frac 1 2}(\partial \Omega)$ are bounded linear maps for $N \in \N$, such that $(P^j_N)^* = P^j_N$, $j=1,2,$ and $P^1_N \to I_{H^{\frac 1 2}(\partial \Omega)}$ and $P^2_N \to I_{H^{-\frac 1 2}(\partial \Omega)}$ strongly as $N \to +\infty$. In particular, since $H^s(\partial \Omega)$ is a Hilbert space for $s \in \R$, the $P^j_N$ can be chosen as orthogonal projections onto the space spanned by the first N elements of any orthonormal bases.
\item $Q_N y = P^2_N y P^1_N$ for $y \in Y$, as in Example \ref{ex:gal}.
\end{itemize}

This is a special case of the one considered in \cite{beretta2019}, where the following Lipschitz stability was proved:
\begin{equation}\label{eq:stab28}
\|\sigma_1-\sigma_2\|_{L^1(\Omega)} \le C \|\Lambda_{\sigma_1}-\Lambda_{\sigma_2}\|_{H^{\frac 1 2}(\partial \Omega) \to H^{-\frac 1 2}(\partial \Omega)}, \qquad \sigma_1, \sigma_2 \in K,
\end{equation}
where $C>0$ is a constant depending on $K$ and $M$. The techniques developed in this paper allows us to easily obtain the same stability in the case of a finite number of measurements.

\begin{theorem}\label{thm:cald}
Under the above assumptions, there exists $C>0$ depending only on $\Omega$, $d_0$, $d_1$ and $k$ such that 
\begin{equation}\label{est:Cald0}
\|\sigma_1-\sigma_2\|_{L^1(\Omega)} \le C \|Q_N(\Lambda_{\sigma_1})-Q_N(\Lambda_{\sigma_2})\|_{H^{\frac 1 2}(\partial \Omega) \to H^{-\frac 1 2}(\partial \Omega)},\qquad \sigma_1, \sigma_2 \in K,
\end{equation}
for every sufficiently large $N \in \N$.
\end{theorem}

The proof is presented in Section~\ref{sec:cald}. After having showed that $M$ is indeed a manifold with the required properties, the result follows from \eqref{eq:stab28} and Theorem~\ref{FM-2}. As far as we know, the only available stability result in this setting with finite measurements is \cite{Liu_2020} (see also \cite{friedman1989,isakov-powell-1990,1994-barcelo-fabes-seo} for uniqueness results), which, in contrast to ours, cannot be extended beyond polygonal-type assumptions.

\begin{remark}
We chose to consider a manifold of piecewise constant conductivities on triangles in order to keep the exposition as simple as possible. Since the Lipschitz stability result with the full DN map \cite{beretta2019} holds for general polygons, it is natural to ask if a discretized estimate \eqref{est:Cald0} could be obtained in that setting. We believe that this is possible thanks to Theorem~\ref{QNF-stability-multiple}. Indeed, the set of piece-wise constant conductivities with discontinuities on a single polygon, as considered in \cite{beretta2019}, can be seen as a (finite) disjoint union of manifolds of different dimensions.
\end{remark}

\subsection{The Gel'fand-Calder\'on problem with spherical inclusions} 

This is a close relative of the Calder\'on problem.
We consider the Dirichlet problem for the Schr\"odinger equation
\begin{equation*}
	\left\{\begin{array}{rl}
	-\Delta u+(\beta +q)u=0 & \text{ in } \Omega,
	\\
	u=\phi & \text{ on } \partial\Omega,
	\end{array}\right.
\end{equation*}
where $\phi\in H^{1/2}(\partial\Omega)$, $q\in L^\infty(\Omega)$ and $\beta\in L^\infty(\Omega)$ is a known background potential, with $\Omega \subseteq \R^3$ a bounded Lipschitz domain. We assume that $0$ is not a Dirichlet eigenvalue of $-\Delta +\beta +q$, so that the boundary value problem is well posed.

Under this assumption, we can define the Dirichlet-to-Neumann map
\[
\Lambda_q\colon H^{1/2}(\partial \Omega)\to H^{-1/2}(\partial \Omega),\qquad \Lambda_q(\phi) = \left.\frac{\partial u^\phi_q}{\partial \nu}\right|_{\partial \Omega},
\]
where $u^\phi_q$ is the unique solution of the Dirichlet problem above. Let us now state the related inverse boundary value problem.
 
\vspace*{.2cm}
\textbf{Gel'fand-Calder\'on's problem.} Given $\Lambda_q$, find $q$ in $\Omega$.
\vspace*{.2cm}

Also for this problem, many uniqueness and logarithmic stability results have been obtained \cite{Novikov1988,alessandrini1988,bukhgeim2008,novikov2010}, as well as many Lipschitz stability results for potentials belonging to a finite-dimensional space or with polyhedral type discontinuities (see \cite{beretta2013,beretta2015,alessandrini2018,2020-ruland-sincich} and references therein). All these results require the knowledge of the full DN map, i.e.\ an infinite number of measurements.\bigskip

As for the Calder\'on problem, we show how the results of Section~\ref{sec:Lip} can be used to derive a Lipschitz stability estimate with a finite number of measurements. 

We consider potentials $q$ that are piecewise constant with discontinuities on a finite number of (disjoint) balls, with unknown centers and radii. We also assume that the coefficient on each ball is unknown and varies continuously in a bounded interval. To the best of our knowledge, the latter assumption has never been considered in the literature and provides an additional challenge to the problem. Thus, in this setting, the stability with an infinite number of measurements is new as well.\smallskip

Let $a_{\rm max}>0$ and $0<\varrho_{\rm min}< \varrho_{\rm max}<+\infty$ and consider $L^\infty(\Omega)$ potentials of the form $\beta+q$, where
\begin{equation*}
	q = \sum_{k=1}^p\lambda_k\chi_{B(a_k,r_k)},
\end{equation*}
where $1\le p \le N_B$, for some known $N_B\in \N$, and $|\lambda_k|,r_k \in [\varrho_{\rm min},\varrho_{\rm max}]$ and $a_k \in \R^3$ with $|a_k|\le a_{\rm max}$ for $k=1,\dots,p$. We assume that $\overline{B(0,a_{\rm max}+\varrho_{\rm max})}\subseteq\Omega$ and
\begin{equation*}
	r_k+r_\ell
    \le \eta
	|a_k-a_\ell|
	\ \text{ for every } \ k\neq\ell,
\end{equation*}
for some $\eta\in (0,1)$, so that the balls are contained in $\Omega$ (and at positive distance from $\partial\Omega$) and are pairwise disjoint. Moreover, we assume that there is a constant $D>0$ such that
\begin{equation*}
\|(-\Delta + \beta+q)^{-1}\|_{\L(H^{-1}(\Omega),H^1_0(\Omega))}\le D,
\end{equation*}
which, in particular, tells us that $\Lambda_q$ is well defined. We denote  the set of potentials $q$ satisfying the above conditions by $K$.\smallskip

We now consider the projections $Q_N y = P^2_N y P^1_N$ for $y \in \L(H^{\frac 12}(\partial \Omega),H^{-\frac 12}(\partial \Omega))$ as in section~\ref{sub:cald}.

We can now state the main stability result of the section.
\begin{theorem}\label{thm:gelf}
Under these assumptions, there exists $C>0$ depending only on $K$ such that 
\begin{equation*}
\|q_1-q_2\|_{L^1(\Omega)} \le C \|Q_N(\Lambda_{q_1})-Q_N(\Lambda_{q_2})\|_{H^{\frac 1 2}(\partial \Omega) \to H^{-\frac 1 2}(\partial \Omega)},\qquad q_1, q_2 \in K,
\end{equation*}
for every sufficiently large $N \in \N$.
\end{theorem}

The proof is given in Section~\ref{sec:gelf} and it is based on Theorem~\ref{QNF-stability-multiple}. Several preliminary lemmas are needed before we can check the assumptions of Theorem~\ref{QNF-stability-multiple}. The crucial step is the injectivity of the Frech\'et derivative of the DN map. This is obtained by combining a family of complex geometrical optics (CGO) solutions with a Runge approximation argument, together with some properties of Bessel functions.

\section{Reconstruction algorithm} \label{sec:rec}

The $\alpha$-H\"older stability estimate from \Cref{FM-2} can be used to design a reconstruction algorithm when $\alpha\in(\frac{1}{2},1]$. In this section, we slightly strengthen the assumptions of \Cref{FM-2} (regarding the regularity of $F$, $M$ and the map $Q$), and let
\begin{itemize}
\item $X$ and $Y$ be Banach spaces;
\item $Q:Y\to Y$ be a continuous finite-rank operator;
\item $A\subseteq X$ be an open set;
\item $M\subseteq A$ be an $n$-dimensional differentiable manifold $\alpha$-H\"older in $X$, for some $\alpha\in (\frac 12,1]$ with constant $\ell\in(0,1]$ as in \eqref{holder parametrization};
\item $K\subseteq M$ be a compact set;
\item $F\in C^1(M,Y)$ be such that
\begin{itemize}
\item $F|_K$ is Lipschitz continuous, namely,
\begin{equation}\label{F-lip}
	\|F(x)-F(y)\|_Y
	\le
	L\|x-y\|_X,
	\qquad x,y\in K,
\end{equation}
for some $L>0$;
\item $(F\circ\varphi_i^{-1})':\varphi_i(U_i)\to \mathcal{L}(\R^n,Y)$ is Lipschitz continuous for every $i\in I$, where $\{(U_i,\varphi_i)\}_{i\in I}$ is an atlas for $M$;
\item there exists $C>0$ such that
\begin{equation}\label{RA-QF-stability}
	\|x-y\|_X \le C \|QF(x)-QF(y)\|_Y^\alpha
	,
	\qquad x,y\in M.
\end{equation}
\end{itemize}
\end{itemize}

\begin{remark}
Since the range of $Q$ is finite dimensional, it is isomorphic to a finite-dimensional Euclidean space. In particular, in what follows, and without loss of generality, we assume that $Y$ is a Hilbert space.
\end{remark}

We denote  the unknown signal by $x^\dagger\in K$ and  the corresponding measurements by $QF(x^\dagger)$. We derive a global reconstruction algorithm which allows for the recovery of $x^\dagger$ from the knowledge of $QF(x^\dagger)$.

\subsection{The initial guess $x_0$}
To reconstruct $x^\dagger$ from its corresponding measurement $QF(x^\dagger)$, we perform an iterative method based on the Landweber iteration (see \cite[Theorem 3.2]{dehoop2012} and \cite[Proposition~10]{ALB-SAN}). Before doing this, we need to find a starting point $x_0\in K$ sufficiently close to the unknown $x^\dagger$ so that the following two conditions are satisfied.
\begin{enumerate}
\item\label{(1)} Both $x_0$ and $x^\dagger$ are contained in the same compact set $K_0\subseteq U_i$ for some $i\in I$. More precisely, by \Cref{one chart lemma} below, if $x_0\in K$ is a point satisfying
\begin{equation*}
	\|x_0-x^\dagger\|_X<\delta_{K,M}
\end{equation*}
($\delta_{K,M}$ is a positive constant depending only on $K$ and $M$), then there exists a compact set $K_j\subseteq K$ such that $x_0\in K_j$ and
\begin{equation*}
	x^\dagger
	\in
	K_0
	:\,=
	\overline{B_X(K_j,\delta_{K,M})}\cap K
	\subseteq
	U_i
\end{equation*}
for some $i\in I$.
\item\label{(2)} The Landweber iteration related to the minimization of
\begin{equation*}
	\min_{h\in\varphi_i(K_0)}\|Q(F\circ\varphi_i^{-1})(h)-QF(x^\dagger)\|_Y^2,
\end{equation*}
starting at $h_0=\varphi_i(x_0)\in\R^n$ converges. This is ensured when
\begin{equation*}
	\|x_0-x^\dagger\|_X
	<
	\omega^{-1}(\rho),
\end{equation*}
for a certain fixed $\rho>0$ (see \Cref{Landweber} below), where $\omega$ is the modulus of continuity given in \Cref{rem:modulus} below.
\end{enumerate}

In the following lemma we establish a condition to discriminate whether a point $x_0\in K$ satisfies these requirements by comparing the value $QF(x_0)$ with the measurement $QF(x^\dagger)$.

\begin{lemma}
If $x_0\in K$ satisfies
\begin{equation}\label{x0-condition}
	\|QF(x_0)-QF(x^\dagger)\|_Y
	<
	\bigg(\frac{\min\{\omega^{-1}(\rho),\delta_{K,M}\}}{C}\bigg)^{1/\alpha},
\end{equation}
then
\begin{equation}\label{x0-xt}
	\|x_0-x^\dagger\|_X
	<
	\min\{\omega^{-1}(\rho),\delta_{K,M}\}.
\end{equation}
\end{lemma}

\begin{proof}
The bound \eqref{x0-xt} is a direct consequence of \eqref{x0-condition} and the H\"older stability estimate for $QF$, \eqref{RA-QF-stability}.
\end{proof}

It only remains to show a procedure to choose a good candidate for the initial guess $x_0\in K$ satisfying \eqref{x0-condition}.

\begin{lemma}
Let $\{x_j\}_{j\in J}\subseteq K$ be a finite set of points satisfying
\begin{equation}\label{K-lattice}
	K
	\subseteq
	\bigcup_{j\in J}B_X(x_j,r),
\end{equation}
with
\begin{equation}\label{r}
	r
	=
	\frac{1}{L\|Q\|_{\L(Y,Y)}}\bigg(\frac{\min\{\omega^{-1}(\rho),\delta_{K,M}\}}{C}  \bigg)^{1/\alpha}.
\end{equation}
Then, inequality \eqref{x0-condition} holds for at least one point $x_0=x_j$ in the set.
\end{lemma}

The finite set of points $\{x_j\}_{j\in J} \subseteq K$, which exists since $K$ is compact, can be constructed by considering sufficiently fine lattices in $\varphi_i (U_i \cap K)$ mapped back to $M$ via $\varphi_i^{-1}$.

\begin{proof}
Since $x^\dagger\in K$, there exists $x_j$ in the lattice at a distance of at most $r>0$ from $x^\dagger$. Recalling \eqref{F-lip} we have
\begin{equation*}
	\|F(x_j)-F(x^\dagger)\|_Y
	\le
	L\|x_j-x^\dagger\|_X.
\end{equation*}
Then we can estimate
\begin{equation*}
	\|QF(x_j)-QF(x^\dagger)\|_Y
	\le
	L\|Q\|_{\mathcal{L}(Y,Y)}\|x_j-x^\dagger\|_X
	<
	L\|Q\|_{\mathcal{L}(Y,Y)}r,
\end{equation*}
so \eqref{x0-condition} follows for $x_0=x_j$.
\end{proof}

\subsection{Local reconstruction}

In the following proposition we provide an iterative method based on \cite{dehoop2012} to reconstruct $x^\dagger$ starting from a good approximation $x_0$.

\begin{proposition}\label{Landweber}
There exist $\rho,\mu>0$ and $c\in(0,1)$ such that the following is true. Let $x^\dagger\in K$ and $QF(x^\dagger)\in Y$. Suppose that $x_0\in K$ satisfies
\begin{equation*}
	\|x_0-x^\dagger\|_X
	<
	\min\{\omega^{-1}(\rho),\delta_{K,M}\}.
\end{equation*}
Let $i\in I$ be the index from condition \eqref{(1)} above and $\{x_k\}_k$ be the sequence of points defined by the recursive relation
\begin{equation*}
	\varphi_i(x_{k+1})
	=
	\varphi_i(x_k)-\mu(F\circ\varphi_i^{-1})'(\varphi_i(x_k))^*Q^*\big(QF(x_k)-QF(x^\dagger)\big).
\end{equation*}
Then $x_k\to x^\dagger$. More precisely, the convergence rate is given by
\begin{equation*}
	\|x_k-x^\dagger\|_X
	\le
	\frac{\rho c^k}{\ell},
	\qquad k\in\N,
\end{equation*}
if $\alpha=1$,  and
\begin{equation*}
	\|x_k-x^\dagger\|_X
	\le
	\frac{1}{\ell}\bigg(ck\frac{1-\alpha}{\alpha}+\rho^{-\frac{1-\alpha}{\alpha}}\bigg)^{-\frac{\alpha^2}{2(1-\alpha)}},
	\qquad k\in\N,
\end{equation*}
if $\alpha\in (\frac12,1)$.
\end{proposition}

\begin{proof}
First recall that the index $i\in I$ and the compact set $K_0\subseteq U_i$ were fixed in condition \eqref{(1)}.
Next define $\widetilde F=F\circ\varphi_i^{-1}\colon\varphi_i(U_i)\to Y$ and let $h_0=\varphi_i(x_0)$. The Landweber iteration for the minimization of
\begin{equation*}
	\min_{h\in\varphi_i(U_i)}\|Q\widetilde F(h)-QF(x^\dagger)\|_Y^2
\end{equation*}
reads
\begin{equation*}
	h_{k+1}
	=
	h_k-\mu\widetilde F'(h_k)^*Q^*\big(Q\widetilde F(h_k)-QF(x^\dagger)\big),
	\qquad k\in\N,
\end{equation*}
where $\mu>0$ is the step size. Since $\varphi_i|_{K_0}$ has modulus of continuity $\omega$ (see \Cref{rem:modulus} below), we have that
\begin{equation*}
	|h_0-\varphi_i(x^\dagger)|
	\le
	\omega(\|x_0-x^\dagger\|_X)
	<
	\rho,
\end{equation*}
and by \cite[Theorem 3.2]{dehoop2012} (see also \cite[Proposition~10]{ALB-SAN}), $h_k\to\varphi_i(x^\dagger)$. Moreover, if $\alpha=1$, the convergence rate is
\begin{equation*}
	|h_k-\varphi_i(x^\dagger)|
	\le
	\rho c^k,
\end{equation*}
while for $\alpha\in(\frac{1}{2},1)$,
\begin{equation*}
	|h_k-\varphi_i(x^\dagger)|
	\le
	\bigg(ck\frac{1-\alpha}{\alpha}+\rho^{-\frac{1-\alpha}{\alpha}}\bigg)^{-\frac{\alpha}{2(1-\alpha)}},
	\qquad k\in\N.
\end{equation*}

Finally, let us define $x_k=\varphi_i^{-1}(h_k)$ for every $k\in\N$ and observe that by continuity $x_k\to x^\dagger$. Moreover, by \eqref{holder parametrization}, the convergence rate is
\begin{equation*}
	\|x_k-x^\dagger\|_X
	\le
	\frac{1}{\ell}|h_k-\varphi_i(x^\dagger)|^\alpha,
	\qquad k\in\N,
\end{equation*}
and the proof is complete.
\end{proof}

\subsection{Global reconstruction}

We combine these two steps to obtain a global reconstruction algorithm, see \Cref{alg:rec}. Note that this algorithm can be split into an offline part and an online part. The offline part consists of the computation of $Q(F(x_j))$ for $j\in J$, which has to be done only once and can be done in parallel.
\begin{algorithm}
  \caption{Reconstruction of $x^\dag$ from $Q(F(x^\dag))$}
  \label{alg:rec}
  \begin{algorithmic}[1]
    \STATE Input $X$, $Y$, $M$, $\{(U_i,\varphi_i)\}_{i\in I}$, $K$, $Q$, $F$, $Q(F(x^\dag))$, $\alpha$, $\ell$, $\rho$, $\mu$, $C$ and $M$.
    \STATE Equip $\R^n$ and $Q(Y)$ with equivalent Euclidean scalar products.
    \STATE Find a finite set $\{x_j\}_{j\in J}\subseteq K$ so that \eqref{K-lattice} is satisfied with $r$ as in \eqref{r}.
     \FOR {$j\in J$}
    \STATE Compute $Q(F(x_j))$.
    \IF{\eqref{x0-condition} is satisfied with $x_0=x_j$} 
    \STATE Set $x_0=x_j$.
    \STATE Exit \textbf{for}.
    \ENDIF 
    \ENDFOR
    \STATE Choose $i\in I$ and $K_0\subseteq U_i$ as in condition \eqref{(1)}.
    \STATE Let $h_0=\varphi_i(x_0)$.
    \FOR {$k=0,1,2,\ldots$}
     \STATE Set $h_{k+1}=h_k-\mu(F\circ\varphi_i^{-1})'(h_k)^*Q^*\big(Q(F\circ\varphi_i^{-1})(h_k)-QF(x^\dagger)\big)$.
     \STATE Check the stopping criterion.
    \ENDFOR
    \STATE Output $x_{k+1}=\varphi_i^{-1}(h_{k+1})$.
  \end{algorithmic}
\end{algorithm}

\section{Toy examples} \label{sec:exa}

In this section, we consider several examples of manifolds $M$ and, in one case, of an operator $F$ satisfying the assumptions of our results. We do not aim at studying complicated and real-world scenarios, but rather at illustrating our results with simple, if not toy, case studies.

\subsection{Indicator functions on balls with variable centres and radii}\label{example-1}
Let
\[
M=\{\chi_{B(a,r)}\,:\, a\in\R^d,\; |a|<A,\; \varrho_0<r<\varrho_1\},
\]
where $A>0$ and $0<\varrho_0<\varrho_1$ are fixed parameters, and consider the function $\varphi\colon M\to\R^d\times\R$ given by $\varphi(\chi_{B(a,r)})=(a,r)$.

 \subsubsection*{$M$ is a differentiable manifold H\"older in $L^p=L^p(\R^d)$ for $p\in [1,+\infty)$}
 We show that the set $M$ together with the atlas $\{(M,\varphi)\}$ is an $(d+1)$-dimensional differentiable manifold in $L^p$ for finite $p$. Since $M\subseteq L^p$ for each $1\le p\le\infty$ and $\varphi\circ\varphi^{-1}= Id$ is the unique transition map, it only remains to study for which values of $p$ the function $\varphi$ is indeed a homeomorphism.
For $p=\infty$, the difference in $L^\infty$ between any two different elements in $M$ is equal to $1$, that is
\begin{equation*}
	\|\chi_{B(a_1,r_1)}-\chi_{B(a_2,r_2)}\|_{L^\infty}
	=
	1,
\end{equation*}
so $\varphi^{-1}$ is not continuous in $L^\infty$. As a consequence, $M\subseteq L^\infty$ is not a differentiable manifold in $L^\infty$.
For $1\le p<\infty$, it turns out that
\begin{equation*}
\begin{split}
	\|\chi_{B(a_1,r_1)}-\chi_{B(a_2,r_2)}\|_{L^p}
	=
	~&
	\left(\int_{\R^d}\big|\chi_{B(a_1,r_1)}(x)-\chi_{B(a_2,r_2)}(x)\big|^p\ dx\right)^{1/p}
	\\
	=
	~&
	|B(a_1,r_1)\triangle B(a_2,r_2)|^{1/p},
\end{split}
\end{equation*}
where $\triangle$ here denotes the symmetric difference.
In order to obtain estimates for the right-hand side above, we observe that 
\begin{multline*}
	\frac{1}{C}|(a_1,r_1)-(a_2,r_2)|
	\le
	|B(a_1,r_1)\triangle B(a_2,r_2)|
	\le
	C|(a_1,r_1)-(a_2,r_2)|,
	\\
	\quad |a_1|,|a_2|<A,\; r_1,r_2\in (\varrho_0,\varrho_1),	
\end{multline*}
for some constant $C=(d,A,\varrho_0,\varrho_1)>1$ (see \Cref{sec:estimates}).
Thus
\begin{equation}\label{B1-B2 in L^p}
	\|\chi_{B(a_1,r_1)}-\chi_{B(a_2,r_2)}\|_{L^p}
	\asymp\footnote{$f(x) \asymp g(x)$ for $x \in X$ means that there exists a constant $c\ge 1$ independent of $x$ such that $\frac 1 c g(x) \leq f(x) \leq c g(x)$ for every $x \in X$.}
	|(a_1,r_1)-(a_2,r_2)|^{1/p},
	\quad (a_1,r_1),(a_2,r_2)\in \varphi(M).
\end{equation}
In consequence, $\varphi$ is a homeomorphism, so $M\subseteq L^p$ is an $(d+1)$-dimensional differentiable manifold in $L^p$ for each $1\le p<\infty$. In addition, by \Cref{Lipschitz in X}, it turns out that $M$ is $\frac{1}{p}$-H\"older in $L^p$ for $1\le p<\infty$. 

\subsubsection*{$M$ is not embedded in $L^p$ for $p\in (1,+\infty)$}
Since
\begin{equation*}
		\frac{\|\chi_{B(a_1,r_1)}-\chi_{B(a_2,r_2)}\|_{L^p}}{|(a_1,r_1)-(a_2,r_2)|}
		\asymp
		|(a_1,r_1)-(a_2,r_2)|^{-(p-1)/p},
\end{equation*}
the function $\varphi^{-1}$ is not locally Lipschitz for any $1<p<\infty$. Hence $\varphi$ fails to be a diffeomorphism and, by \Cref{rem:embedded}, $M$ is not embedded in $L^p$ for any $1<p<\infty$.

\subsubsection*{$M$ is Lipschitz but not embedded in $L^1$}
For $p=1$, both $\varphi$ and $\varphi^{-1}$ are Lipschitz continuous, and according to \Cref{Lipschitz in X}, $M$ is Lipschitz in $L^1$. However, $M$ is not embedded in $L^1$. To see this, fix $r_1=r_2=1$ (without loss of generality, assume that $\varrho_0<1<\varrho_1$) and observe that $\varphi^{-1}$ is not Fr\'echet differentiable, so $\varphi$ is not a diffeomorphism in $L^1$. Indeed, the following holds for $h\in\R^d\setminus\{0\}$,
\begin{equation*}
				\lim_{t\to 0^+}
				\frac{\chi_{B(th,1)}(z)-\chi_{B(0,1)}(z)}{t}
				=
				\begin{cases}
				+\infty & \mbox{ if } |z|=1 \mbox{ and } \prodin{z}{h}>0,
				\\
				-\infty & \mbox{ if } |z|=1 \mbox{ and } \prodin{z}{h}<0,
				\\
				0 & \mbox{ elsewhere.}
				\end{cases}
\end{equation*}
The right-hand side in the previous equation defines a function which is zero almost everywhere. However, by \eqref{B1-B2 in L^p} we have
\begin{equation*}
	\bigg\|\frac{\chi_{B(th,1)}-\chi_{B(0,1)}}{t}\bigg\|_{L^1}
	\asymp
	|h|,
\end{equation*}
for every $t>0$, so $\varphi^{-1}$ is not Fr\'echet differentiable.
Thus $M$ is not embedded in $L^1$.

\subsection{Indicator functions on balls with variable centres, radii and intensities}\label{sect6.2}
Take $A>0$ and $0<\varrho_0<\varrho_1$ and let
\[
M=\{\lambda\chi_{B(a,r)}\,:\,a\in\R^d,\; |a|<A,\; \lambda,r\in (\varrho_0,\varrho_1)\}.
\]
 Consider the function $\varphi\colon M\to\R^d\times\R\times\R$ given by $\varphi(\lambda\chi_{B(a,r)})=(a,r,\lambda)$ for every $\lambda\chi_{B(a,r)}\in M$, so that $\varphi(M)=B(0,A)\times (\varrho_0,\varrho_1)^2$. 

\subsubsection*{$M$ is a differentiable manifold Lipschitz in $L^1=L^1(\R^d)$} 
We show that set $M$ together with the atlas $\{(M,\varphi)\}$ is an $(d+2)$-dimensional differentiable manifold in $L^1=L^1(\R^d)$. Since $M\subseteq L^1$ and $\varphi\circ\varphi^{-1}=Id$ is the unique transition map, it only remains to check that $\varphi$ is a homeomorphism. In fact, we claim that $\varphi$ is bi-Lipschitz, and in particular $M$ is Lipschitz in $L^1$.

Let $(a_1,r_1,\lambda_1),(a_2,r_2,\lambda_2)\in\varphi(M)$. Then
\begin{equation*}
\begin{split}
	\|\lambda_1\chi_{B(a_1,r_1)}-\lambda_2\chi_{B(a_2,r_2)}\|_{L^1}
	=
	~&
	\int_{\R^d}\big|\lambda_1\chi_{B(a_1,r_1)}(x)-\lambda_2\chi_{B(a_2,r_2)}(x)\big|\ dx
	\\
	=
	~&
	\lambda_1\,|B(a_1,r_1)\setminus B(a_2,r_2)|
	\\
	~&
	+\lambda_2\,|B(a_2,r_2)\setminus B(a_1,r_1)|
	\\
	~&
	+|\lambda_1-\lambda_2|\,|B(a_1,r_1)\cap B(a_2,r_2)|,
\end{split}
\end{equation*}
Recalling \eqref{estimate} we get
\begin{multline*}
	\lambda_1\,|B(a_1,r_1)\setminus B(a_2,r_2)|
	+\lambda_2\,|B(a_2,r_2)\setminus B(a_1,r_1)|
	\\
	\asymp
	|B(a_1,r_1)\triangle B(a_2,r_2)|
	\asymp
	|(a_1,r_1)-(a_2,r_2)|.
\end{multline*}
Furthermore, observe that
\begin{equation*}
	|B(a_1,r_1)\cap B(a_2,r_2)|
	\le
	\omega_d\max\{r_1,r_2\}^d
	\le \omega_d \varrho_1^d,
\end{equation*}
where $\omega_d$ stands for the Lebesgue measure of the unit ball in $\R^d$.
Combining these two estimates, we obtain that $\varphi^{-1}$ is Lipschitz continuous.

For the reverse inequality we distinguish three cases. First, if $|a_1-a_2|<|r_1-r_2|$ then either $B(a_1,r_1)\subseteq B(a_2,r_2)$ or $B(a_2,r_2)\subseteq B(a_1,r_1)$, and thus
\begin{equation*}
	|B(a_1,r_1)\cap B(a_2,r_2)|
	\ge
	\omega_d\min\{r_1,r_2\}^d
	\ge
	\omega_d \varrho_0^d.
\end{equation*}
On the other hand, if $|a_1-a_2|\ge|r_1-r_2|$ and $|a_1-a_2|<\varrho_0$, by using the argument at the beginning of subsection \ref{sub:B3} we have
\begin{equation*}
\begin{split}
	|B(a_1,r_1)\cap B(a_2,r_2)|
	\ge
	~&
	\omega_d\left(\frac{r_1+r_2-|a_1-a_2|}{2}\right)^d
	\\
	\ge
	~&
	\omega_d\left(\frac{2\varrho_0-|a_1-a_2|}{2}\right)^d
	>
	\omega_d\left(\frac{\varrho_0}{2}\right)^d.
\end{split}
\end{equation*}
Otherwise, if $|a_1-a_2|\ge|r_1-r_2|$ and $|a_1-a_2|\ge\varrho_0$, then we estimate $|B(a_1,r_1)\cap B(a_2,r_2)|\ge 0$ and thus there exists a constant $c>0$ such that
\begin{equation*}
\begin{split}
	\|\lambda_1\chi_{B(a_1,r_1)}-\lambda_2\chi_{B(a_2,r_2)}\|_{L^1}
	\ge
	~& c
	|(a_1,r_1)-(a_2,r_2)|
	\\
	\asymp
	~&
	|a_1-a_2|+|r_1-r_2|
	\\
	\ge
	~&
	\varrho_0+|r_1-r_2|
	\\
	\ge
	~&
	\varrho_0
	\\
	\ge
	~&
	\frac{\varrho_0}{\diam\varphi(M)}\,|(a_1,r_1,\lambda_1)-(a_2,r_2,\lambda_2)|.
\end{split}
\end{equation*}
Hence
\begin{equation}\label{eq:biLip}
\begin{split}
	\|\lambda_1\chi_{B(a_1,r_1)}-\lambda_2\chi_{B(a_2,r_2)}\|_{L^1}
	\asymp
	~&
	|(a_1,r_1)-(a_2,r_2)|
	+|\lambda_1-\lambda_2|
	\\
	\asymp
	~&
	|(a_1,r_1,\lambda_1)-(a_2,r_2,\lambda_2)|,
\end{split}
\end{equation}
that is, $\varphi$ is bi-Lipschitz in $M$. Thus $M$ is an $(d+2)$-dimensional differentiable manifold Lipschitz in $L^1$.

\subsection{Gaussians with different centres}
For $a\in\R^d$, let $G_a\colon \R^d\to\R$ be the function $G_a(z)=e^{-|z-a|^2}$. We define $M=\set{G_a}{a\in\R^d}$ and $\varphi\colon M\to\R^d$ as the function $\varphi(G_a)=a$. It is clear that $M\subseteq L^p=L^p(\R^d)$ for every $p\in[1,+\infty]$. In fact, we now show that $M$, together with the atlas $\{(M,\varphi)\}$, is a differentiable manifold embedded in $L^p$. For simplicity, we only treat the case $p\in [1,+\infty)$.

\subsubsection*{$\varphi^{-1}$ is Fr\'echet differentiable and $(\varphi^{-1})'(a)$ is injective for every $a\in\R^d$}
To see this we first show that $\varphi^{-1}\colon \R^d\to L^p$ is Fr\'echet differentiable, i.e.\ that there exists a linear map $(\varphi^{-1})'(a)$ at each $a\in\R^d$ satisfying
\begin{equation}\label{varphi-diff}
	\lim_{h\to 0}\frac{\|\varphi^{-1}(a+h)-\varphi^{-1}(a)-(\varphi^{-1})'(a)h\|_{L^p}}{|h|}
	=
	0,
\end{equation}
and then that $(\varphi^{-1})'(a)$ is  injective for each $a\in\R^d$.
To do this, we assume without loss of generality that $a=0$ and we observe that the gradient of $h\mapsto G_h(z)=e^{-|z-h|^2}=e^{-|z|^2}e^{-|h|^2+2\prodin{z}{h}}$ at $h=0$ is equal to $2e^{-|z|^2}z$. Thus
\begin{equation*}
	\lim_{h\to 0}\frac{\big|e^{-|z-h|^2}-e^{-|z|^2}-2e^{-|z|^2}\prodin{z}{h}\big|}{|h|}
	=
	0
\end{equation*}
for each $z\in\R^d$. In order to apply Lebesgue's dominated convergence theorem to obtain \eqref{varphi-diff} we need to show that the function 
\begin{equation*}
	z
	\longmapsto
	\frac{\big|e^{-|z-h|^2}-e^{-|z|^2}-2e^{-|z|^2}\prodin{z}{h}\big|}{|h|}
\end{equation*}
is bounded by a function in $L^p$ for every sufficiently small $|h|>0$. Indeed, by the triangle inequality,
\begin{equation*}
	\frac{\big|e^{-|z-h|^2}-e^{-|z|^2}-2e^{-|z|^2}\prodin{z}{h}\big|}{|h|}
	\le
	\frac{\big|e^{-|z-h|^2}-e^{-|z|^2}\big|}{|h|}+2|z|e^{-|z|^2}.
\end{equation*}
Notice that, for every $z\in\R^d$ the following holds
\begin{equation*}
\begin{split}
	\big|e^{-|z-h|^2}-e^{-|z|^2}\big|
	=
	~&
	e^{-|z|^2}\,\big|e^{|z|^2-|z-h|^2}-1\big|
	\\
	\le
	~&
	e^{-|z|^2}\,\big(e^{\left||z|^2-|z-h|^2\right|}-1\big)
	\\
	\le
	~&
	\left||z|^2-|z-h|^2\right|e^{\left||z|^2-|z-h|^2\right|-|z|^2}
	\\
	=
	~&
	\left||h|^2-2\prodin{z}{h}\right|e^{\left||h|^2-2\prodin{z}{h}\right|-|z|^2}
	\\
	\le
	~&
	|h|\big(|h|+2|z|\big)e^{|h|^2+2|z||h|-|z|^2},
\end{split}
\end{equation*}
where in the second inequality we have used that $e^t\leq1+te^t$ for every $t\geq0$. Thus, for every $0<|h|<1$ we have
\begin{equation*}
\begin{split}
	\frac{\big|e^{-|z-h|^2}-e^{-|z|^2}\big|}{|h|}
	\le
	~&
	\big(1+2|z|\big)e^{1+2|z|-|z|^2}
	=
	e^2\big(1+2|z|\big)e^{-(|z|-1)^2}.
\end{split}
\end{equation*}
Summarizing,
\begin{equation*}
	\frac{\big|e^{-|z-h|^2}-e^{-|z|^2}-2e^{-|z|^2}\prodin{z}{h}\big|}{|h|}
	\le
	e^2\big(1+2|z|\big)e^{-(|z|-1)^2}+2|z|e^{-|z|^2}
\end{equation*}
for every $|h|<1$, where the right-hand side is in $L^p$. Hence, by Lebesgue's dominated convergence theorem,
\begin{equation*}
	\lim_{h\to 0}\frac{1}{|h|^p}\int_{\R^d}\big|e^{-|z-h|^2}-e^{-|z|^2}-2e^{-|z|^2}\prodin{z}{h}\big|^p\ dz
	=
	0,
\end{equation*}
and \eqref{varphi-diff} follows with $(\varphi^{-1})'(a)=\{z\mapsto2e^{-|z-a|^2}\prodin{z-a}{\cdot}\}$, which is an injective linear map in $\L(\R^d,L^p)$.

\subsubsection*{$M$ is a differentiable manifold embedded in $L^p$}
It is easy to show that $(\varphi^{-1})'$ is continuous. We can now apply the inverse function theorem and obtain that $\varphi^{-1}$ is a diffeomorphism in a neighbourhood of every $a\in\R^d$. In particular, $\varphi^{-1}\colon \R^d\to M$ is a diffeomorphism and so $M$ is a differentiable manifold embedded in $L^p$. In particular, the tangent space $T_{G_a}M$ is contained in $L^p$,
\begin{equation*}
				T_{G_a}M
				=
				\set{z\mapsto 2e^{-|z-a|^2}\prodin{z-a}{h}}{h\in\R^d}
				\subseteq
				L^p(\R^d).
\end{equation*}

\subsubsection*{$M$ is Lipschitz in $L^p$}
By the mean value theorem for Gateaux differentiable functions between Banach spaces we obtain that for every $a,b\in\R^d$ there exists $c\in\R^d$ such that
\[
\|\varphi^{-1}(a)-\varphi^{-1}(b)\|_{L^p}\le \|(\varphi^{-1})'(c)\|_{\L(\R^d,L^p)}|a-b| =  \|(\varphi^{-1})'(0)\|_{\L(\R^d,L^p)}|a-b|,
\]
where we use the fact that, by construction, the norm of $(\varphi^{-1})'(c)$ is translation invariant.

\subsubsection*{$\varphi\colon M\to\R^d$ is uniformly Lipschitz}
We already know that $\varphi$ is Lipschitz continuous, since it is a diffeomorphism. However, adapting the argument used to show that $\varphi^{-1}$ is uniformly Lipschitz, we immediately derive that $\varphi$ is uniformly Lipschitz too.

\subsection{A classical inverse problem}
In this example we consider the classical inverse problem of differentiation. We show that, even if this inverse problem is notoriously ill-posed, Lipschitz stability is restored by restricting the unknown to a finite-dimensional manifold.

\subsubsection*{The manifold $M$}

Let $X=Y=L^1=L^1([0,1])$, fix $\eps\in(0,\frac12)$ and consider the set
\begin{equation*}
	M
	=
	\{\chi_{[a,b]}\,:\,a,b\in(0,1),\ b-a>\eps\}
	\subseteq
	L^1.
\end{equation*}
Let $\varphi\colon M\to\R^2$ be the function given by
\begin{equation*}
	\varphi(\chi_{[a,b]})
	=
	(a,b).
\end{equation*}

\subsubsection*{The map $\varphi$ is bi-Lipschitz, and $M$ is a differentiable manifold Lipschitz in $L^1$}
For $\chi_{[a_1,b_1]},\chi_{[a_2,b_2]}\in M$, observe that
\begin{equation*}
	\|\chi_{[a_1,b_1]}-\chi_{[a_2,b_2]}\|_{L^1}
	=
	\big|[a_1,b_1]\triangle[a_2,b_2]\big|.
\end{equation*}
If $[a_1,b_1]\cap[a_2,b_2]\neq\emptyset$, then
\begin{equation*}
	\big|[a_1,b_1]\triangle[a_2,b_2]\big|
	=
	|a_1-a_2|+|b_1-b_2|.
\end{equation*}
On the other hand, if $[a_1,b_1]\cap[a_2,b_2]=\emptyset$, then
\begin{equation*}
	\big|[a_1,b_1]\triangle[a_2,b_2]\big|
	=
	b_1-a_1+b_2-a_2
	\le
	|a_1-a_2|+|b_1-b_2|,
\end{equation*}
and for the other inequality, since $b_1-a_1,b_2-a_2\ge\eps$ and $a_1,a_2,b_1,b_2\in[0,1]$,
\begin{equation*}
	\big|[a_1,b_1]\triangle[a_2,b_2]\big|
	=
	b_1-a_1+b_2-a_2
	\ge
	2\eps
	\ge
	\eps(|a_1-a_2|+|b_1-b_2|).
\end{equation*}
Thus
\begin{equation*}
	\eps(|a_1-a_2|+|b_1-b_2|)
	\le
	\|\chi_{[a_1,b_1]}-\chi_{[a_2,b_2]}\|_{L^1}
	\le
	|a_1-a_2|+|b_1-b_2|
\end{equation*}
for $\chi_{[a_1,b_1]},\chi_{[a_2,b_2]}\in M$, so $\varphi$ is bi-Lipschitz and $M\subseteq L^1$ is a $2$-dimensional differentiable manifold $1$-H\"older (Lipschitz) in $L^1$. 
\subsubsection*{The operator $F$}
Next we define $F\colon L^1\to L^1$ as the function
\begin{equation*}
	u\longmapsto
	F(u)(t)
	:\,=
	\int_0^tu(s)\ ds, \quad t\in[0,1].
\end{equation*}

\subsubsection*{The differential of $F$}
A direct computation shows
\begin{equation*}
	F(\chi_{[a,b]})(t)
	=
	F(\chi_{[a,1]})(t)-F(\chi_{[b,1]})(t)
	=
	(t-a)\chi_{[a,1]}(t)-(t-b)\chi_{[b,1]}(t).
\end{equation*}
Since $F$ is linear, then the Fr\'echet derivative of $F$ at $u\in L^1$ coincides with $F$, that is, $F'(u)\equiv F$ for each $u\in L^1$. However, since $M$ is not embedded in $L^1$ (see \Cref{example-1}), we cannot define $dF_{\chi_{[a,b]}}$ as the restriction of $F'(\chi_{[a,b]})$ to $T_{\chi_{[a,b]}}M$ due to the fact that the tangent space is not contained in $L^1$.
In order to compute $dF_{\chi_{[a,b]}}$ we first need to check that $F\circ\varphi^{-1}$ is Fr\'echet differentiable (see \Cref{sub:diff}). That is, we need to show that there exists a linear map $A\colon\R^2\to L^1$ such that
\begin{equation}\label{diff-Ah}
	\lim_{\mu\to 0^+}\frac{1}{\mu}\left\|(F\circ\varphi^{-1})(a+\mu h_1,b+\mu h_2)-(F\circ\varphi^{-1})(a,b)-\mu Ah\right\|_{L^1}
	=
	0,
\end{equation}
for $h=(h_1,h_2)\in\R^2\setminus\{0\}$. 
We start from the ansatz
\begin{multline*}
	Ah(t)
	=
	\lim_{\mu\to 0^+}\frac{(F\circ\varphi^{-1})(a+\mu h_1,b+\mu h_2)(t)-(F\circ\varphi^{-1})(a,b)(t)}{\mu}
	\quad \text{a.e. }t\in[0,1].
\end{multline*}
Observe that the right-hand side above is equal to the derivative at $\mu=0$ of
\begin{equation*}
\begin{split}
\mu\longmapsto
	&(F\circ\varphi^{-1})(a+\mu h_1,b+\mu h_2)(t)
	\\
	&=
	F(\chi_{[a+\mu h_1,b+\mu h_2]})(t)
	\\
	&=
	(t-a-\mu h_1)\chi_{[a+\mu h_1,1]}(t)-(t-b-\mu h_2)\chi_{[b+\mu h_2,1]}(t).
\end{split}
\end{equation*}
Hence,
\begin{equation*}
	Ah(t)
	=
	\begin{cases}
	0 & \text{ if } t\in[0,a),
	\\
	-h_1 & \text{ if } t\in(a,b),
	\\
	h_2-h_1 & \text{ if } t\in(b,1].
	\end{cases}
\end{equation*}
In other words,
\begin{equation*}
	Ah
	=
	h_2\chi_{[b,1]}-h_1\chi_{[a,1]}
	\qquad \text{a.e. }t\in[0,1].
\end{equation*}
We check that, in fact, $A$ is the Fr\'echet differential of $F$:
\begin{equation*}
\begin{split}
	&\left\|(F\circ\varphi^{-1})(a+\mu h_1,b+\mu h_2)-(F\circ\varphi^{-1})(a,b)-\mu Ah\right\|_{L^1}
	\\
&\qquad	=
	\int_\R\big|(t-a-\mu h_1)(\chi_{[a+\mu h_1,1]}(t)-\chi_{[a,1]}(t))
	\\
	~&
	\hspace{40pt}-(t-b-\mu h_2)(\chi_{[b+\mu h_2,1]}(t)-\chi_{[b,1]}(t))\big|\ dt
	\\
	&\qquad
	\le
	\bigg|\int_a^{a+\mu h_1}|t-a-\mu h_1|\ dt\bigg|
	+
	\bigg|\int_b^{b+\mu h_2}|t-b-\mu h_2|\ dt\bigg|
	\\
	&\qquad
	=
	\int_0^{\mu|h_1|}s\ ds
	+
	\int_0^{\mu|h_2|}s\ ds
	=
	\frac{\mu^2|h|^2}{2},
\end{split}
\end{equation*}
so \eqref{diff-Ah} holds.
Thus $F\circ\varphi^{-1}$ is Fr\'echet differentiable and
\begin{equation*}
	dF_{\chi_{[a,b]}}(h_1,h_2)
	=
	h_2\chi_{[b,1]}-h_1\chi_{[a,1]},
	\quad (h_1,h_2)\in\R^2,
\end{equation*}
which  is injective in $\R^2$ (since $a<b<1$).

\subsubsection*{The operator $F$ is of class $C^1(M,L^1)$}
By \Cref{F in M diffble}, we need to show that $(F\circ\varphi^{-1})'\colon\varphi(M)\to\L(\R^2,L^1)$ is continuous, where
\begin{equation*}
	\varphi(M)
	=
	\{(a,b)\in\R^2\,:\,a,b\in(0,1),\ b-a>\eps\}
\end{equation*}
and
\begin{equation*}
	(F\circ\varphi^{-1})'(a,b)
	=
	dF_{\chi{[a,b]}}.
\end{equation*}
In fact, if $(a,b),(a',b')\in\varphi(M)$, then
\begin{equation*}
\begin{split}
	\left\|\big(dF_{\chi_{[a,b]}}-dF_{\chi{[a',b']}}\big)(h_1,h_2)\right\|_{L^1}
	=
	~&
	\left\|h_2(\chi_{[b,1]}-\chi_{[b',1]})-h_1(\chi_{[a,1]}-\chi_{[a',1]})\right\|_{L^1}
	\\
	\le
	~&
	|h_1|\,|a-a'|+|h_2|\,|b-b'|
	\\
	\le
	~&
	(|h_1|+|h_2|)(|a-a'|+|b-b'|)
\end{split}
\end{equation*}
for every $(h_1,h_2)\in\R^2$. Thus
\begin{equation*}
	\left\|dF_{\chi_{[a,b]}}-dF_{\chi{[a',b']}}\right\|_{\L(\R^2,L^1)}
	\le
	\sqrt{2}(|a-a'|+|b-b'|),
\end{equation*}
and the continuity of $dF$ in $\L(\R^2,L^1)$ follows.

\subsubsection*{Lipschitz stability with infinite-dimensional measurements}
Let $K\subseteq M$ be a compact set. For example, we can take
\[
	K
	=
	\{\chi_{[a,b]}\,:\,a,b\in[\eps,1-\eps],\ b-a\geq2\eps\}.
\]
By \Cref{MAINTHM1}, $F$ satisfies the following Lipschitz stability estimate for some constant $C>0$,
\begin{equation*}
	\|\chi_{[a_1,b_1]}-\chi_{[a_2,b_2]}\|_{L^1}
	\le
	C
	\|F(\chi_{[a_1,b_1]})-F(\chi_{[a_2,b_2]})\|_{L^1},
	\quad \chi_{[a_1,b_1]},\chi_{[a_2,b_2]}\in K.
\end{equation*}

\subsubsection*{Lipschitz stability with finitely many measurements}
Furthermore, for $N\in\N$ let us define the map $Q_N\colon L^1\to L^1$ as the convolution $Q_Nu=\F_N*u$, where the so-called Fej\'er kernel is given by
\begin{equation*}
	\F_N(t)
	=
	\sum_{k=-N}^N\bigg(1-\frac{|k|}{N+1}\bigg)e^{2\pi ikt}.
\end{equation*}
Constructed in this way, it turns out that $Q_N$ satisfies the conditions from \Cref{Q_N} with $\widetilde Y=Y=L^1$. Indeed, since $\F_N*u$ converges to $u$ in $L^1$ for every $u\in L^1$, then
\begin{equation*}
	\lim_{N\to\infty}
	\|u-Q_Nu\|_{L^1}
	=
	0
\end{equation*}
for every $u\in L^1$. On the other hand,
\begin{equation*}
	\|Q_Nu\|_{L^1}
	=
	\|\F_N*u\|_{L^1}
	\le
	\|\F_N\|_{L^1}\|u\|_{L^1}=\|u\|_{L^1},
\end{equation*}
so
\begin{equation*}
	\|Q_N\|_{\L(L^1,L^1)}
	\le
	1,
	\qquad
	N\in\N.
\end{equation*}
As a  consequence, by virtue of \Cref{QNF-stability}, for $N$ and $C$ large enough, $Q_NF$ satisfies the Lipschitz stability estimate \eqref{QNF-stability-estimate}
\begin{equation*}
			\|x-y\|_X \le C	\|Q_NF(x)-Q_NF(y)\|_Y
								,\qquad x,y\in K.
\end{equation*}
It is worth observing that measuring $Q_Nz$ for a certain $z\in L^1$ corresponds to measuring a low-frequency approximation of $z$.

\section{Proofs: infinite-dimensional measurements}\label{sec:proofs-1}

In this section we prove Theorems~\ref{THM infinite measurements without compactness} and \ref{MAINTHM1}.

\subsection{Lipschitz stability estimate in the large distance case}

One of the common elements in the theorems in \Cref{subsec:infinity-meas} is the injectivity assumption on the function $F$. This hypothesis, together with the continuity of $F$, is crucial to obtain the Lipschitz stability estimate in the \emph{large distance} case, that is, when the distance between $x$ and $y$ in $K\subseteq X$ compact is uniformly bounded away from zero.

\begin{lemma}\label{large distance lemma}
Let $X$ and $Y$ be Banach spaces, $K\subseteq X$ be a compact set, $F\colon K\to Y$ be a continuous and injective function and $\delta>0$. There exists a constant $C>0$ such that
\begin{equation}\label{Lipest>delta}
	\|x-y\|_X \le C\|F(x)-F(y)\|_Y
\end{equation}
for every $x,y\in K$ such that $\|x-y\|_X\ge\delta$. 
\end{lemma}

\begin{proof}
If $\set{(x,y)\in K\times K}{\|x-y\|_X\ge\delta}=\emptyset$, the result is trivial. Otherwise, observe that since the set $\set{(x,y)\in K\times K}{\|x-y\|_X\ge\delta}$ is compact and the function $(x,y)\longmapsto\|F(x)-F(y)\|_Y$ is continuous, we can define
\begin{equation*}
	C'
	:\,=
	\min\set{\|F(x)-F(y)\|_Y}{x,y\in K\ \text{ s.t. }\ \|x-y\|_X\ge\delta}.
\end{equation*}
Then the injectivity of $F$ yields that $C'>0$ and \eqref{Lipest>delta} follows with $C=\diam K/C'$.
\end{proof}

The previous lemma shows that the analysis can be restricted to the case where $x,y\in K$ are arbitrarily close, in which case we need to impose certain conditions for obtaining Lipschitz stability.

\subsection{Proof of \Cref{THM infinite measurements without compactness}}

In order to deal with the lack of convexity in the assumptions of \Cref{THM infinite measurements without compactness}, we show that it is possible to extend the compact set $K$ to a bigger compact subset of the manifold containing all line segments between points in $K$ of small length. This property, which can be understood as some sort of \emph{short distance convexity}, turns out to be enough for our purposes. Here and in the rest of the paper we use the following notation: for $S\subseteq X$ and $\delta>0$, we let
\[
B_X(S,\delta) = \bigcup_{x\in S} B_X(x,\delta)
\]
denote the $\delta$-neighbourhood of $S$.

\begin{lemma}\label{delta-K}
Let $X$ be a Banach space,  $A\subseteq X$ be an open set, $W\subseteq X$ be an $n$-dimensional subspace and $K\subseteq W\cap A$ be a compact subset. There exist $\delta_K\in(0,\diam K]$ and a compact set $\widehat K\subseteq W\cap A$ such that $(1-t)x+ty\in\widehat K$ for all $t\in[0,1]$ and every $x,y\in K$ satisfying $\|x-y\|_X\le\delta_K$.
\end{lemma}

\begin{proof}
If $K$ is a singleton, the result is immediate. We assume that $\diam K>0$.

We begin by observing that in the case in which $W\cap A=W$, we just simply define $\widehat K$ as the convex hull of $K$, so the result follows for any $\delta_K>0$.
On the other hand, if $W\setminus A\neq\emptyset$ we set
\[
\tilde\delta_K:\,=\frac{1}{2}\dist(K,W\setminus A)=\frac{1}{2}\inf\set{\|x-w\|_X}{x\in K,\ w\in W\setminus A},
\]
and $\delta_K=\min\{\tilde\delta_k,\diam K\}$.
Since $K$ is compact and $W\setminus A$ is closed, then $\delta_K>0$. Observe that for every $x\in K$, $w\in W\setminus A$ and $y\in B_X(x,\delta_K)\cap W$ we have that
\begin{equation*}
		2\delta_K
		\le
		\|x-w\|_X
		\le
		\|x-y\|_X+\|y-w\|_X
		<
		\delta_K+\|y-w\|_X,
\end{equation*}
so $\|y-w\|_X>\delta_K$. As a consequence, we have that
\begin{equation*}
	\widehat{K}
	:\,=
	\overline{B_X(K,\delta_K)}\cap W
	\subseteq
	W\cap A.
\end{equation*}
Furthermore, the line segment between $x$ and $y$ is strictly contained in $\widehat{K}$ for every $x,y\in K$ satisfying $\|x-y\|_X\le\delta_K$. Finally, since $K$ is contained in a finite-dimensional subspace $W$, to see that $\widehat K$ is compact it is enough to check that $\widehat K$ is bounded and closed. This follows from the fact that $\widehat K$ is the set of points which are at a distance of at most $\delta_K$ from a point in the compact set $K$.
\end{proof}

It is worth to mention that, despite the fact that it might not be convex, the compact set $\widehat K$ contains every closed line segment between close points in $K$, and the same argument in the proof of \cite[Theorem 2.1]{BOU} is still valid if instead of $K$ we consider its extension $\widehat K$. However, for the sake of completeness and the benefit of the reader, we have decided to include the proof of the Lipschitz stability estimate here.

\begin{proof}[Proof of \Cref{THM infinite measurements without compactness}]
By \Cref{large distance lemma}, we can assume that $x,y\in K$ are given such that $\|x-y\|_X<\delta$ for some fixed $\delta\in(0,\delta_K]$ to be determined later. Then, by \Cref{delta-K}, the closed line segment between $x$ and $y$ is contained in a compact set $\widehat K\subseteq W\cap A$ , i.e. $\gamma(t):\,=(1-t)x+ty\in\widehat K$ for every $t\in[0,1]$. Recalling the fundamental theorem of calculus we can write
\begin{equation*}
	F(y)-F(x)
	=
	\int_0^1(F\circ\gamma)'(t)\ dt
	=
	\int_0^1F'(\gamma(t))(y-x)\ dt,
\end{equation*}
where in the second equality we have used the fact that $F\in C^1(A,Y)$.
Therefore,
\begin{equation*}
	F'(x)(x-y)
	=
	F(x)-F(y)+\int_0^1\big[F'(x)-F'(\gamma(t))\big](x-y)\ dt,
\end{equation*}
and taking norms we get
\begin{equation*}
	\|F'(x)(x-y)\|_Y
	\le
	\|F(x)-F(y)\|_Y
	+
	\int_0^1\|F'(x)-F'(\gamma(t))\|_{\L(W,Y)}\,\|x-y\|_X\ dt.
\end{equation*}
A rearrangement and an estimation of the terms gives us the following inequality,
\begin{equation*}
	\frac{\|F(x)-F(y)\|_Y}{\|x-y\|_X}
	\ge
	\inf_{z\in\S_W}\big\{\|F'(x)z\|_Y\big\}
	-
	\sup_{t\in[0,1]}\|F'(x)-F'(\gamma(t))\|_{\L(W,Y)},
\end{equation*}
which holds for every $x,y\in K$ such that $\|x-y\|_X<\delta$,
where $\S_W=\S_X\cap W$ is the unit sphere of $W$. The injectivity of $F'(x)$ in $W$ together with the compactness of $K$ and $\S_W$ yields that
\begin{equation*}
	C'
	:\,=
	\frac{1}{2}\inf_{x\in K,\,z\in\S_W}\|F'(x)z\|_Y
	>
	0.
\end{equation*}
On the other hand, since $F\in C^1(A,Y)$, $\widehat{K}\subseteq W\cap A$ is compact and $\gamma(t)\in\widehat{K}$ for every $t\in[0,1]$, there exists a non-decreasing modulus of continuity $\omega_{F',\widehat{K}}$ such that
\begin{equation*}
	\|F'(x)-F'(\gamma(t))\|_{\L(W,Y)}
	\le
	\omega_{F',\widehat{K}}(\|x-\gamma(t)\|_X)
	\le
	\omega_{F',\widehat{K}}(\|x-y\|_X)
	\le
	\omega_{F',\widehat{K}}(\delta)
\end{equation*}
for every $t\in[0,1]$.
Then, choosing a small enough $\delta\in(0,\delta_K]$ such that $\omega_{F',\widehat{K}}(\delta)\le C'$ we obtain \eqref{Lip-stability-1} with $C=1/C'$ for every $x,y\in K$ such that $\|x-y\|_X<\delta$. 
\end{proof}

As an immediate consequence of \Cref{THM infinite measurements without compactness}, we obtain the following corollary in the particular case $X=\R^n$.

\begin{corollary}\label{COR-1}
Let $Y$ be a Banach space, $A\subseteq\R^n$ an open set and $K\subseteq A$ a compact subset. Consider $F\in C^1(A,Y)$ satisfying that
\begin{enumerate}
\item $F$ is injective;
\item $F'(x)\in\mathcal{L}(\R^n,Y)$ is injective for every $x\in A$.
\end{enumerate}
Then there exists $C>0$ such that
\begin{equation*}
				|x-y|\le C \|F(x)-F(y)\|_Y
				,\qquad x,y\in K.
\end{equation*}
\end{corollary}

We now pass to the proof of the main stability estimate in the case of an infinite number of measurements.

\subsection{Proof of \Cref{MAINTHM1}}

In the following technical lemma we show that the parameter $\delta$ can be chosen small enough so that if $x$ and $y$ are two points such that $\|x-y\|_X<\delta$, then $x$ and $y$ belong to a single compact set contained in a chart $U_i$ for some $i\in I$. As a result, we will be able to consider a single chart in the atlas of $M$.

\begin{lemma}\label{one chart lemma}
Let $X$ be a Banach space, $M\subseteq X$ be an $n$-dimensional differentiable manifold with an atlas $\{(U_i,\varphi_i)\}_{i\in I}$ and $K\subseteq M$ be a compact set. There exist $\delta_{K,M}>0$ and a finite collection of compact sets $K_1,\ldots,K_m\subseteq M$ such that $K=K_1\cup\cdots\cup K_m$ and for every $j\in\{1,\dots,m\}$ there exists $i\in I$ such that
\[
\overline{B_X(K_j,\delta_{K,M})}\cap M\subseteq U_i.
\]
\end{lemma}

\begin{remark}\label{rem:modulus}
This result implies a stronger continuity property of the charts. Recall that for every $j\in\{1,\dots,m\}$ there exists $i_j\in I$ such that
\[
\overline{B_X(K_j,\delta_{K,M})}\cap M\subseteq U_{i_j}.
\]
Since the set $\overline{B_X(K_j,\delta_{K,M})}\cap K$ is compact and $\varphi_{i_j}$ is continuous, there exists a modulus of continuity $\omega_j$ for $\varphi_{i_j}$ restricted to $\overline{B_X(K_j,\delta_{K,M})}\cap K$. Setting
\[
\omega(t)=\max_{j=1,\dots,m} \omega_j (t)
\]
yields a unique modulus of continuity $\omega$ which is valid for all the charts, namely
\[
|\varphi_{i_j}(x)-\varphi_{i_j}(x)|\le \omega(\|x-y\|_X),\qquad  x,y\in \overline{B_X(K_j,\delta_{K,M})}\cap K,
\]
for every $j=1,\dots,m$.
\end{remark}

\begin{proof}
Let us start with the case in which the manifold is associated to an atlas with just one chart, say $\{(M,\varphi)\}$. Then the inclusion
\[
\overline{B_X(K,\delta)}\cap M\subseteq M
\]
holds for any $\delta >0$. Therefore, in what follows we assume that the atlas $\{(U_i,\varphi_i)\}$ has at least two charts.

Observe next that $M=\bigcup_{i\in I}U_i$ and $K\subseteq M$, so $\{U_i\}_{i\in I}$ is an open cover of $K$. Since $K$ is compact, we can extract a finite subcover of $K$ denoted by $\{U_i\}_{i=1,\ldots,m}$.

Next fix any $i=1,\ldots,m$ and observe that the function
\begin{equation*}
	x
	\longmapsto
	\dist(x,M\setminus U_i)
	:\,=
	\inf\set{\|x-w\|_X}{w\in M\setminus U_i}
\end{equation*}
is continuous in $M$. Indeed, for $x,y\in M$ and $S=M\setminus U_i$ closed in $M$ we have that $\dist(x,S)\le\|x-w\|_X\le\|y-w\|_X+\|x-y\|_X$ for every $w\in S$. Taking the infimum we get $\dist(x,S)\le\dist(y,S)+\|x-y\|_X$, where the roles of $x$ and $y$ are interchangeable, so $|\dist(x,S)-\dist(y,S)|\le\|x-y\|_X$. On the other hand, since $U_i$ is open with respect to the topology of $M$ inherited from $X$, the function $x\mapsto\dist(x,S)$ is positive in $U_i$ (more precisely,  every $x\in U_i$ has a neighbourhood  contained in $U_i$, namely there is a sufficiently small $\eps>0$ such that $B_X(x,\eps)\cap M\subseteq U_i$, so $\dist(x,S)\ge\eps>0$). Therefore, since $K\subseteq\bigcup_{i=1}^mU_i\subseteq M$, the function $d\colon K\to(0,\infty)$ given by
\begin{equation*}
	d(x)
	:\,=
	\max_{i=1,\ldots,m}\{\dist(x,M\setminus U_i)\}
\end{equation*}
is continuous and positive in $K$. Hence, since $K\subseteq M$ is compact, then $\delta_{K,M}:\,=\frac{\min_K d}{2}>0$.

Finally, let us define the closed sets
\begin{equation*}
	K_i
	:\,=
	\{x\in K\cap U_i\,:\,\dist(x,M\setminus U_i)\ge 2\delta_{K,M} \},
	\qquad i=1,\ldots,m.
\end{equation*}
Then $\overline{B_X(x,\delta_{K,M})}\cap M\subseteq U_i$ for each $x\in K_i$, and since $K_i\subseteq K$, the set $K_i$ is compact. Moreover, for each $x\in K$, since $d(x)\ge 2\delta_{K,M}$ we have that $x$ is contained in some $K_i$, so $K=K_1\cup\cdots\cup K_m$. 
\end{proof}

We are now ready to prove \Cref{MAINTHM1}. 

\begin{proof}[Proof of \Cref{MAINTHM1}]
Let $\delta_{K,M}$ be the constant and $\{K_j\}_{j=1}^m$ be the compact sets from \Cref{one chart lemma}. If $x,y\in K$ with $\|x-y\|_X\ge\delta_{K,M}$, we obtain the Lipschitz stability estimate by recalling \Cref{large distance lemma}. Hence, we focus on the case in which $x,y\in K$ with $\|x-y\|<\delta_{K,M}$. Thus, as an immediate consequence of \Cref{one chart lemma}, there exist $i\in I$ and $j\in\{1,\dots,m\}$ such that $x\in K_j$ and
\begin{equation*}
	y
	\in
	\overline{B_X(K_j,\delta_{K,M})}\cap M
	\subseteq
	U_i.
\end{equation*}
Let $\widetilde A=\varphi_i(U_i)\subseteq\R^n$ and consider $\widetilde F=F\circ\varphi_i^{-1}\colon\widetilde A\to Y$. Since $F\in C^1(M,Y)$, then $\widetilde F\in C^1(\widetilde A,Y)$. Moreover, by continuity, $\widetilde K_j=\varphi_i(\overline{B_X(K_j,\delta_{K,M})\cap M})$ is a compact set in $\widetilde A$. By assumption, since $\varphi_i$ is a homeomorphism, we have that
\begin{enumerate}
\item $\widetilde F$ is injective;
\item $dF_x=\widetilde F'(\widetilde x)\in\L(\R^n,Y)$ is injective for every $\widetilde x=\varphi_i(x)\in\widetilde A$.
\end{enumerate}
Then the hypotheses in \Cref{COR-1} are satisfied, and so there exists $C_{K_j}>0$ such that
\begin{equation}\label{lip-est F tilde}
				\frac{\|\widetilde{F}(\widetilde x)-\widetilde{F}(\widetilde y)\|_Y}{|\widetilde x-\widetilde y|}
				\ge
				C_{K_j},
\end{equation}
for every  $\widetilde x=\varphi_i(x)$ and $\widetilde y=\varphi_i(y)$ in $\tilde K_j$. 
Next, we have
\begin{equation*}
	\frac{\|F(x)-F(y)\|_Y^\alpha}{\|x-y\|_X}
	=
	\bigg(\frac{\|\widetilde{F}(\widetilde x)-\widetilde{F}(\widetilde y)\|_Y}{|\widetilde x-\widetilde y|}\bigg)^\alpha\cdot\frac{|\varphi_i(x)-\varphi_i(y)|^\alpha}{\|x-y\|_X}
	\ge C_{K_j}^\alpha\ell,
\end{equation*}
where in the inequality we have used \eqref{lip-est F tilde} together with the regularity assumptions of the manifold  \eqref{holder parametrization}.

Finally, choosing $C=(\ell \min\{C^\alpha_{K_1},\ldots,C^\alpha_{K_m}\})^{-1}$ we obtain \eqref{Lip-stability} for $\|x-y\|<\delta_{K,M}$, which concludes the proof.
\end{proof}

\section{Proofs: finite-dimensional measurements} \label{sec:pfII}
In this section we prove Theorems~\ref{FM-1}, \ref{FM-2} and \ref{QNF-stability}. We begin with a lemma that guarantees stability for sufficiently distant points, in the case of a finite number of measurements.

\begin{lemma}\label{large distance lemma QN}
Let $X$ and $Y$ be Banach spaces, $K\subseteq X$ be a compact set, $Q_N\colon Y\to Y$ be bounded linear maps satisfying \Cref{Q_N} and  $F\colon K\to Y$ be a continuous function such that $F(x)-F(y) \in \widetilde Y$ for  every $x,y \in K$ and satisfying the Lipschitz stability estimate
\begin{equation*}
	\|x-y\|_X\le C \|F(x)-F(y)\|_Y
	,\qquad x,y\in K,
\end{equation*}
for some $C>0$. Given $\delta>0$, then
\begin{equation}\label{Lipest>delta QN}
	\|x-y\|_X\le \frac{2C\diam K}{\delta}\|Q_NF(x)-Q_NF(y)\|_Y
\end{equation}
for every $x,y\in K$ such that $\|x-y\|_X\ge\delta$ and every sufficiently large $N\in\N$ such that
\begin{equation}\label{QNF-delta}
	\sup_{\xi,\eta\in K}\|F(\xi)-F(\eta)-Q_N(F(\xi)-F(\eta))\|_Y\le\frac{\delta}{2C},
\end{equation}
where the left hand side of this inequality goes to $0$ as $N\to+\infty$.
\end{lemma}

\begin{proof}
We can assume that $0<\delta\le\diam K$ since otherwise the result is trivial. Then, if $x,y\in K$ satisfy $\|x-y\|_X\ge\delta$, by the Lipschitz stability of $F$ together with the triangle inequality we obtain
\begin{align*}
	\delta C^{-1}
	&\le
	\|F(x)-F(y)\|_Y\\
	&\le \|Q_NF(x)-Q_NF(y)\|_Y+ \sup_{\xi,\eta\in K}\|F(\xi)-F(\eta)-Q_N(F(\xi)-F(\eta))\|_Y.
\end{align*}
Hence, the result will follow for every $N\in\N$ such that
\eqref{QNF-delta}
holds.

It remains to show that
\begin{equation}\label{eq:to0}
	\lim_{N\to\infty}\,\sup_{\xi,\eta\in K}f_N(\xi,\eta)
	=
	0,
\end{equation}
where $f_N\colon K\times K\to[0,\infty)$ is the function defined by
\[
f_N(\xi,\eta)=\|F(\xi)-F(\eta)-Q_N(F(\xi)-F(\eta))\|_Y,\qquad \xi,\eta\in K.
\]
Since $F(\xi)-F(\eta) \in \widetilde Y$ for $\xi,\eta \in K$ by assumption, recalling \Cref{Q_N} we get that $f_N(\xi,\eta)\to 0$ as $N\to\infty$ for every $\xi,\eta\in K$. Let us check that $f_N$ is continuous in $K\times K$: for $\xi_1,\xi_2,\eta_1,\eta_2\in K$, we have
\begin{equation}\label{eq:fN}
\begin{split}
		|f_N(\xi_1,\eta_1)-f_N(\xi_2,\eta_2)|
		=
		~&
		\big|\|F(\xi_1)-F(\eta_1)-Q_N(F(\xi_1)-F(\eta_1))\|_Y \\
		&-\|F(\xi_2)-F(\eta_2)-Q_N(F(\xi_2)-F(\eta_2))\|_Y\big|
		\\
		\le
		~&
		\|F(\xi_1)-F(\xi_2)\|_Y+\|F(\eta_1)-F(\eta_2)\|_Y\\
		&+\|Q_N(F(\xi_1)-F(\xi_2))\|_Y+\|Q_N (F(\eta_1)-F(\eta_2))\|_Y
		\\
		\le
		~&
		(D+1)\left(\|F(\xi_1)-F(\xi_2)\|_Y+\|F(\eta_1)-F(\eta_2)\|_Y\right),
\end{split}
\end{equation}
where \Cref{Q_N} has been recalled in the second inequality, so the continuity of $f_N$ follows from the fact that $F$ is continuous.  Moreover, since $K\times K$ is compact, the maximum of $f_N$ is attained at some $(\xi_N,\eta_N)\in K\times K$, that is 
\[
\sup_{\xi,\eta\in K}f_N(\xi,\eta)=f_N(\xi_N,\eta_N).
\]
This produces a sequence of points $(\xi_N,\eta_N)_N$ in $K\times K$.
Let us consider a convergent subsequence $(\xi_{N_j},\eta_{N_j})\to (\widetilde \xi,\widetilde \eta)\in K\times K$, then
\begin{align*}
		\lim_{j\to\infty}f_{N_j}(\xi_{N_j},\eta_{N_j})
		&=
		\lim_{j\to\infty}|f_{N_j}(\xi_{N_j},\eta_{N_j})-f_{N_j}(\widetilde \xi,\widetilde \eta)|\\
		&\le(D+1)\lim_{j\to\infty}\left(\|F(\xi_{N_j})-F(\widetilde \xi)\|_Y+\|F(\eta_{N_j})-F(\widetilde \eta)\|_Y\right)\\
		&=0,
\end{align*}
where we used \eqref{eq:fN} in the last inequality. Since the same argument can be applied to any convergent subsequence $(\xi_{N_j},\eta_{N_j})_j$, by compactness this shows \eqref{eq:to0}, and the proof follows.
\end{proof}

\subsection{Proof of \Cref{FM-1}}

The proof of this result follows  the same argument of the proof of \cite[Theorem 2]{ALB-SAN}.

\begin{proof}[Proof of \Cref{FM-1}]
Let $\delta_K>0$ and $\widehat K$ be the constant and the compact neighbourhood of $K$  from \Cref{delta-K}. If $K$ is convex, simply set $\delta_K=\diam K$ and $\widehat K=K$. By \Cref{large distance lemma QN}, the Lipschitz stability of $Q_NF$ \eqref{LS-QNF} follows in the case $\|x-y\|_X\ge\delta_K$ with $c_K=\frac{\diam K}{\delta_K}\in [1,+\infty)$, so it only remains to show the estimate in the case $\|x-y\|_X<\delta_K$.

By the triangle inequality,
\begin{equation*}
		\|F(x)-F(y)\|_Y
		\le
		\|Q_NF(x)-Q_NF(y)\|_Y+\|(I_Y-Q_N)(F(x)-F(y))\|_Y,
\end{equation*}
and since $F$ is Lipschitz stable in $K$ by assumption,
\begin{equation*}
		\frac{\|Q_NF(x)-Q_NF(y)\|_Y}{\|x-y\|_X}
		\ge
		C^{-1}-\frac{\|(I_Y-Q_N)(F(x)-F(y))\|_Y}{\|x-y\|_X}.
\end{equation*}
In order to show the Lipschitz stability estimate for $Q_NF$, it is sufficient  to show that
\begin{equation}\label{QNF estimate}
		\frac{\|(I_Y-Q_N)(F(x)-F(y))\|_Y}{\|x-y\|_X}
		\le
		\frac{1}{2C},\qquad x,y\in K, \|x-y\|_X<\delta_K,
\end{equation}
holds for every large enough $N\in\N$. The strategy is to show that the left hand side in \eqref{QNF estimate} is uniformly bounded in $K$ by a constant depending on $N$ and vanishing when $N\to\infty$. Since $\|x-y\|_X<\delta_K$, \Cref{delta-K} yields that the closed line segment between $x$ and $y$ is contained in the compact set $\widehat K$. On the other hand, since $Q_N$ is linear, then $(Q_NF)'(x)=Q_NF'(x)$. Then, by the mean value theorem for Gateaux differentiable functions between Banach spaces, for every $x,y\in K$ with  $\|x-y\|_X<\delta_K$ there exists $\xi_0\in\widehat K$ such that
\begin{equation*}
\begin{split}
		\frac{\|(I_Y-Q_N)(F(x)-F(y))\|_Y}{\|x-y\|_X}
		\le
		~&
		\frac{\|(I_Y-Q_N)F'(\xi_0)(x-y)\|_Y}{\|x-y\|_X}
		\\
		\le
		~&
		\sup_{\xi\in\widehat K}\sup_{\zeta\in\S_W}\|(I_Y-Q_N)F'(\xi)\zeta\|_Y
		=\,:
		s_N,
\end{split}
\end{equation*}
where the second inequality comes from the fact that $x$ and $y$ belong to the vector space $W$.
Therefore, \eqref{QNF estimate} (and thus the desired Lipschitz stability estimate with constant $2C\le 2c_K C$) follows whenever $s_N\le\frac{1}{2C}$, which is ensured by the fact that $s_N\to0$ as $N\to\infty$ (see \cite[Theorem 2]{ALB-SAN}).
\end{proof}

As in the case with infinite-dimensional measurements (\Cref{COR-1}), for $X=\R^n$ we obtain the following corollary as an immediate consequence of \Cref{FM-1}.

\begin{corollary}\label{FM-COR}
Let $Y$ be a Banach space, $A\subseteq\R^n$ an open set and $K\subseteq A$ a compact subset. Consider a Fr\'echet differentiable map $F\in C^1(A,Y)$ such that
\begin{enumerate}
\item $F(x)-F(y) \in \widetilde Y$ for every  $x,y \in K$;
\item $\ran\{F'(x)\}\subseteq\widetilde Y$ for every $x\in A$;
\item the Lipschitz stability estimate
\begin{equation*}
				|x-y|\le C\|F(x)-F(y)\|_Y
				,\qquad x,y\in K,
\end{equation*}
is satisfied for some $C>0$. 
\end{enumerate}
Then $Q_NF$ satisfies the Lipschitz stability estimate
\begin{equation*}
				|x-y|\le  2c_K C\|Q_NF(x)-Q_NF(y)\|_Y
				,\qquad x,y\in K, 
\end{equation*}
for some $c_K>0$ depending only on $K$ and every  sufficiently large $N\in\N$. If $K$ is convex, we can choose $c_K=1$. In the general case, we have
\[
c_K=\frac{\diam K}{\delta_K},
\]
where $\delta_K$ is the constant given in Lemma~\ref{delta-K}.
\end{corollary}

We are now able to prove the main stability results in the case of a finite number of measurements and nonlinear priors.

\subsection{Proof of \Cref{FM-2}}

Let $\delta_{K,M}>0$ be the constant and $\{K_j\}_{j=1}^m$ be the compact sets from \Cref{one chart lemma}. If $x,y\in K$ satisfy $\|x-y\|_X\ge\delta_{K,M}$, the Lipschitz stability follows by  \Cref{large distance lemma QN}. Therefore, for the rest of the proof we assume that $\|x-y\|_X<\delta_{K,M}$, so $x$ and $y$ are covered by the same chart $(U_i,\varphi_i)$. In addition, let $K_j\subseteq K\cap U_i$ be the compact set from \Cref{one chart lemma} such that $x,y\in K_j$.

Let $\widetilde A=\varphi_i(U_i)\subseteq\R^n$ and consider $\widetilde F=F\circ\varphi_i^{-1}\colon\widetilde A\subseteq\R^n\to Y$. Since $F\in C^1(M,Y)$, then $\widetilde F\in C^1(\widetilde A,Y)$. Moreover, by continuity, $\widetilde K_j=\varphi_i(\overline{B_X(K_j,\delta_{K,M})}\cap K)$ is a compact set in $\widetilde A$. Then $\ran(\widetilde F'(\widetilde x))=\ran(dF_x)\subseteq\widetilde Y$ for every $x\in U_i$, where $\widetilde x=\varphi_i(x)$, by assumption (2). Observe that assumption (3a) immediately implies (3b). Thus
\begin{equation}\label{F-tilde-estimate}
	\frac{\|\widetilde{F}(\widetilde x)-\widetilde{F}(\widetilde y)\|_Y}{|\widetilde x-\widetilde y|}
	\ge
	\frac{\ell}{C},
\end{equation}
for every $\widetilde x=\varphi_i(x)$ and $\widetilde y=\varphi_i(y)$ in $\tilde K_j$. By \Cref{FM-COR} there exists $N_j\in\N$ such that
\begin{equation*}
	\frac{\|Q_N\widetilde F(\widetilde x)-Q_N\widetilde F(\widetilde y)\|_Y}{|\widetilde x-\widetilde y|}
	\ge
	\frac{\ell}{2c_{\widetilde K_j} C},\qquad \tilde x,\tilde y\in \tilde K_j,
\end{equation*}
for every $N\ge N_j$. This together with the fact that $\varphi_i^{-1}$ is $\alpha$-H\"older continuous yields
\begin{equation*}
\begin{split}
	\frac{\|Q_NF(x)-Q_NF(y)\|_Y^\alpha}{\|x-y\|_X}
	=
	~&
	\bigg(\frac{\|Q_N\widetilde F(\widetilde x)-Q_N\widetilde F(\widetilde y)\|_Y}{|\widetilde x-\widetilde y|}\bigg)^\alpha\cdot\frac{|\varphi_i(x)-\varphi_i(y)|^\alpha}{\|x-y\|_X}
	\\
	\ge
	~&
	\Big(\frac{\ell}{2c_{\widetilde K_j} C}\Big)^\alpha\ell \\ 
	\ge
	~&
	\frac{\ell^{\alpha+1}}{(2C)^\alpha}\,\max\{c_{\widetilde K_1},\ldots,c_{\widetilde K_m}\}^{-\alpha}
	,
\end{split}
\end{equation*}
for every $N\ge\max(N_1,\dots,N_m)$, and the proof is finished. \qed

\subsection{Proof of \Cref{QNF-stability}}
The proof of the Lipschitz stability estimate for $Q_NF$ \eqref{QNF-stability-estimate} follows by combining \Cref{MAINTHM1} and \Cref{FM-2} with $\alpha=1$ and under assumption (3b), which was already obtained in the proof of \Cref{MAINTHM1} (estimate \eqref{lip-est F tilde}). \qed

\subsection{Proof of \Cref{QNF-stability-multiple}}
By applying \Cref{QNF-stability} to $F|_{M_p}$ for $p=1,\dots,P$ we obtain that there exist $C_p>0$ and $N_p\in\N$ such that for every $N\ge N_p$ we have
\begin{equation}\label{QNF-stability-estimate-p}
			\|x-y\|_X \le C_p	\|Q_NF(x)-Q_NF(y)\|_Y
								,\qquad x,y\in K_p.
\end{equation}

Since the sets $K_p$ are compact and pairwise disjoint, their pairwise distance is positive, namely $d(K_p,K_q)>0$ for $p\neq q$. Thus there exists $\delta>0$ such that
\begin{equation}\label{eq:distance}
\delta\le d(K_p,K_q),\qquad p,q=1,\dots,P,\;p\neq q.
\end{equation}
Let us apply \Cref{large distance lemma} to $F|_K$: there exists $C'>0$ such that
\begin{equation*}
	\|x-y\|_X \le C'\|F(x)-F(y)\|_Y,\qquad x,y\in K,\;\|x-y\|_X\ge\delta.
\end{equation*}
As a consequence, by \eqref{QNF-stability-estimate-p} and \eqref{eq:distance} we obtain
\begin{equation*}
	\|x-y\|_X \le C''\|F(x)-F(y)\|_Y,\qquad x,y\in K,
\end{equation*}
where $C''=\max(C',DC_1,\dots,DC_P)$.
Thus, by  \Cref{large distance lemma QN} applied to $F|_K$, there exists $\tilde N\in \N$ such that for every $N\ge \tilde N$ we have
\begin{equation}\label{eq:almostthere}
	\|x-y\|_X\le \tilde C\|Q_NF(x)-Q_NF(y)\|_Y,\qquad x,y\in K,\;\|x-y\|_X\ge\delta,
\end{equation}
where $\tilde C=\frac{2C''\diam K}{\delta}$.

Finally, combining  \eqref{QNF-stability-estimate-p}, \eqref{eq:distance} and \eqref{eq:almostthere} we obtain \eqref{QNF-stability-estimate2}, namely 
\begin{equation*}
			\|x-y\|_X \le C	\|Q_NF(x)-Q_NF(y)\|_Y
								,\qquad x,y\in K,
\end{equation*}
for $N\ge \max(\tilde N, N_1,\dots,N_P)$, where
$
C=\max(\tilde C,C_1,\dots,C_P).
$
\qed

\section{The Calder\'on problem with a triangular inclusion}\label{sec:cald}

We first introduce the manifold we consider in this example. Next, we will apply our main results to prove Theorem~\ref{thm:cald}.

\subsection{The manifold: indicator functions on simplexes}\label{sub:simplex}

For the sake of generality, we proceed in arbitrary dimension $d\ge 2$, even if the two-dimensional case would be sufficient for the current purposes.

\subsubsection*{The $d$-simplexes  in $\R^d$}
Let $\triangle^d$ denote the family of $d$-simplexes in $\R^d$. That is
\begin{equation*}
	\triangle^d
	:\,=
	\{\conv\{v_0,v_1,\ldots,v_d\}\,:\,v_0,v_1,\ldots,v_d\in\R^d,\ \det(v_1-v_0,\ldots,v_d-v_0)\neq 0\},
\end{equation*}
where $\conv S$ stands for the convex hull of a set $S$, i.e.\ the smallest set containing all convex combinations of elements in $S$. The condition $\det(v_1-v_0,\ldots,v_d-v_0)\neq 0$ in the definition of $\triangle^d$ ensures that the simplexes are not degenerate.
In addition, we introduce a constant $\mu>0$ controlling the size of the simplexes in the following way:  we  assume that
\begin{equation}\label{mu-control}
	|v_i-v_j|<\mu
	\qquad\text{ for } i,j=0,1,\ldots,d.
\end{equation}

It turns out that each triangle in $\triangle^d$ can be identified with a $d\times(d+1)$ real matrix containing the coordinates of its vertices as column vectors,
\begin{align*}
	&T\in\triangle^d\longrightarrow\; v^T:=(v_0,v_1,\ldots,v_d)\in\R^{d\times(d+1)},\quad T=\conv\{v_0,v_1,\ldots,v_d\},\\
	&v=(v_0,v_1,\ldots,v_d)\in\R^{d\times(d+1)} \longrightarrow\; T_v=\conv\{v_0,v_1,\ldots,v_d\}\in\triangle^d.
\end{align*}
However, the matrix $v^T$ is not unique, since any permutation of its columns would represent the same triangle.
To avoid this inconvenient, we assume without loss of generality that the vertices of $v^T$ are labeled according to an order in $\R^d$ (such as, for example, the so-called lexicographical order). 
We define a norm $\|\cdot\|_{\triangle^d}$ in $\R^{d\times(d+1)}$ by
\begin{equation*}
	\|(v_0,v_1,\ldots,v_d)\|_{\triangle^d}
	=
	\max\{|v_0|,|v_1|,\ldots,|v_d|\},
\end{equation*}
where $|\cdot|$ stands for the usual Euclidean norm of a vector in $\R^d$. This norm is equivalent to the usual Euclidean norm when the elements in $\triangle^d$ are treated as vectors in $\R^{d(d+1)}$.

\subsubsection*{The manifold and the atlas}
For each triangle $T\in\triangle^d$, we let $\chi_T$ be the indicator function on $T$, which is a function in $L^1=L^1(\R^d)$. We define
\begin{equation*}
	\widetilde M
	:\,=
	\{\chi_{T_v}\,:\,v=(v_0,v_1,\ldots,v_d)\in \R^{d\times(d+1)} \text{ satisfies \eqref{mu-control}}\}.
\end{equation*}
We shall show that $\widetilde M\subseteq L^1$ is a $d(d+1)$-dimensional manifold. To see this, we first need to construct an atlas for $M$. Let $T\in\triangle^d$ be any fixed triangle. Let $R_T=\frac{1}{4}\min_{i\neq j}|v^{T}_i-v^{T}_j|$ and define the following set of functions in $\widetilde M$: 
\begin{equation*}
	U_T
	=
	\left\{\chi_{T_v}\in \widetilde M\,:\,\|v^T-v\|_{\triangle^d}<R_T\right\}.
\end{equation*}
In addition, let $\varphi_T\colon U_T\to\R^{d\times(d+1)}\approx\R^{d(d+1)}$ be the function
\begin{equation*}
	\varphi_T(\chi_{T_v})
	:\,=
	v,	
	\qquad
	\text{such that}
	\qquad \|v^T-v\|_{\triangle^d}<R_T.
\end{equation*}
Note that the matrix $v$ is determined in a unique way. Indeed, since each vertex of $T_v$ is at a distance of at most $R_T$ from one of the vertices of $T$, and the balls of radius $R_T$ centred at the vertices of $T$ do not intersect, a permutation of the column vectors in $v$ would result in a new matrix $v'$ such that
\begin{equation}\label{eq:permutation_RT}
\| v^T-v'\|_{\triangle^d}\ge 3R_T\ge R_T.
\end{equation}

It is worth remarking that, even though the function $\varphi_T$ was constructed assuming that the vertices in $v^T$ had been labeled according to a pre-established order in $\R^d$, the matrices $v$ satisfying $\|v^T-v\|_{\triangle^d}<R_T$ might not have its vertices ordered in the same way. However, this is not a problem, since $\varphi_T^{-1}(v)=\chi_{ T_v}$ does not depend on the order in which the vertices of $v$ are considered.

By construction, $\varphi_T$ is a bijective function between $U_T$ and
$\varphi_T(U_T)=\{v\in\R^{d\times(d+1)}\,:\,v\text{ satisfies \eqref{mu-control} and }\|v^T-v\|_{\triangle^d}<R_T\}$. Note that $\varphi_T(U_T)$ is open.

\subsubsection*{The maps $\varphi_T^{-1}\colon\varphi_T(U_T)\to \widetilde M$ are Lipschitz}

Let $v,v'\in \varphi_T(U_T)$.
Since
\begin{equation*}
	\|\varphi_T^{-1}(v')-\varphi_T^{-1}(v)\|_{L^1}
	=
	\|\chi_{T_{v'}}-\chi_{T_v}\|_{L^1}
	=
	|T_{v'}\triangle T_v|,
\end{equation*}
 our aim is to show that the inequality
\begin{equation}\label{Lipschitz-triangles}
	|T_{v'}\triangle T_v|
	\le
	C\|v'-v\|_{\triangle^d},\qquad v,v'\in \varphi_T(U_T),
\end{equation}
holds for some constant $C>0$. For simplicity, let us write $\delta=\|v'-v\|_{\triangle^d}$ and observe that the symmetric difference between $T_{v'}$ and $T_v$ is contained in the $\delta$-neighbourhood of $\partial T_v$, that is, $T_{v'}\triangle T_v\subseteq\{\xi\in\R^d\,:\,\dist(\xi,\partial T_v)\le\delta\}$. In turn, by \eqref{mu-control}, this set is included in the union of $d+1$ rectangular prisms of measure $(\mu+2\delta)^{d-1}2\delta$ (see \Cref{fig:triangle}).
Hence, we can estimate
\begin{equation*}
	|T_{v'}\triangle T_v|
	\le
	(d+1)(\mu+2\delta)^{d-1}2\delta,
\end{equation*}
and since $\delta\le \|v'-v^T\|_{\triangle^d}+\|v^T-v\|_{\triangle^d}<2R_T=\frac12\min_{i\neq j}|v_i-v_j|\le\mu$, then
\begin{equation*}
	|T_{v'}\triangle T_v|
	\le
	3^d(d+1)\mu^{d-1}\delta,
\end{equation*}
so \eqref{Lipschitz-triangles} follows with $C=3^d(d+1)\mu^{d-1}$.

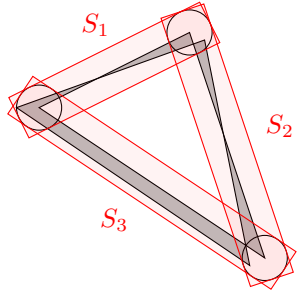
\begin{figure}
\begin{tikzpicture}[scale=1]
\fill[even odd rule,gray!50!white] (0,0)--(2,1)--(3,-2)--cycle (-0.3,0)--(2.2,0.9)--(2.8,-2.1);
\draw [black](0,0)--(2,1)--(3,-2)--cycle;
\draw [black](-0.3,0)--(2.2,0.9)--(2.8,-2.1)--cycle;
\draw [black](0,0) circle (0.3);
\draw [black](2,1) circle (0.3);
\draw [black](3,-2) circle (0.3);
\filldraw [red,opacity=0.05,rotate=26.5650512](-0.3,-0.3)--(-0.3,0.3)--(2.5360679775,0.3)--(2.5360679775,-0.3)--cycle;
\draw [red,rotate=26.5650512](-0.3,-0.3)--(-0.3,0.3)--(2.5360679775,0.3)--(2.5360679775,-0.3)--cycle;
\filldraw [red,opacity=0.05,rotate=-33.6900675](-0.3,-0.3)--(-0.3,0.3)--(3.90555127546,0.3)--(3.90555127546,-0.3)--cycle;
\draw [red,rotate=-33.6900675](-0.3,-0.3)--(-0.3,0.3)--(3.90555127546,0.3)--(3.90555127546,-0.3)--cycle;
\filldraw [red,opacity=0.05,rotate around={-71.5650512:(2,1)}](1.7,0.7)--(1.7,1.3)--(5.46227766017,1.3)--(5.46227766017,0.7)--cycle;
\draw [red,rotate around={-71.5650512:(2,1)}](1.7,0.7)--(1.7,1.3)--(5.46227766017,1.3)--(5.46227766017,0.7)--cycle;
\draw [red] (0.75,1.1) node {$S_1$};
\draw [red] (3.2,-0.25) node {$S_2$};
\draw [red] (1,-1.5) node {$S_3$};
\end{tikzpicture}
\caption{For $d=2$ and $v',v$ such that$\|v'-v\|_{\triangle^2}=\delta<R_T$, we have that $T_{v'}\triangle T_v\subseteq S_1\cup S_2\cup S_3$, where the $S_i$'s are rectangles of area bounded by $(\mu+2\delta)2\delta$.}
\label{fig:triangle}
\end{figure}

\subsubsection*{The maps $\varphi_T\colon U_T\to\R^{d\times(d+1)}$ are continuous}

For each $\chi_{T_0}\in U_T$ we take any sequence $\{\chi_{T_k}\}_k$ in $U_T$ converging to $\chi_{T_0}$. Then
\begin{equation*}
	|T_k\triangle T_0|
	=
	\|\chi_{T_k}-\chi_{T_0}\|_{L^1}
	\xrightarrow[k\to\infty]{}0.
\end{equation*}
This means that $T_k$ converges as a set to $T_0$,  and thus each vertex of $T_k$ converges to the corresponding vertex of $T_0$, so we have that
\begin{equation*}
	\|v^k-v^0\|_{\triangle^d}\xrightarrow[k\to\infty]{}0,
\end{equation*}
for some matrices $v^k$ and $v^0$ such that $T_{v^k}=T_k$ and $T_{v^0}=T_0$.
Then the continuity of $\varphi_T$ follows by construction.

\subsubsection*{The sets $U_T$ are open in $\widetilde M$}
Let $\chi_{T_0}\in U_T$. We show that $\chi_{T_0}$ is an interior point of $U_T$. Assume by contradiction that for every $k\ge 1$ there exists $\chi_{T_k}\in B_{L^1}(\chi_{T_0},\frac{1}{k})\cap M$ such that $\chi_{T_k}\notin U_T$. Then $\|v^k-v^T\|_{\triangle^d}\ge R_T$ for some $v^k\in\R^{d\times(d+1)}$ such that $T_{v^k}=T_k$. By the triangle inequality we then get that
\begin{equation*}
	0
	<
	R_T-\|v^0-v^T\|_{\triangle^d}
	\le
	\|v^k-v^0\|_{\triangle^d},
	\qquad k\ge 1,
\end{equation*}
for some matrix $v^0$ such that $T_{v_0}=T_0$. Arguing as above, since $\|\chi_{T_k}-\chi_{T_0}\|_{L^1}\to 0$, we obtain that $\|v^k-v^0\|_{\triangle^d}\to 0$, a contradiction.

\subsubsection*{The transition maps are continuously differentiable}

By construction, the transition maps are simply permutations of the vertices of the triangles, and are therefore smooth.

\subsubsection*{$\widetilde M$ is a differentiable manifold Lipschitz in $L^1$}

We have shown that $\{(U_T,\varphi_T)\,:\,T\in\triangle^d\}$ is an atlas for $\widetilde M$, which is a $d(d+1)$-dimensional differentiable manifold Lipschitz in $L^1$.

\subsection{Lipschitz stability with finite measurements}

We refer to Section~\ref{sub:cald} for the basic notation of Calder\'on's problem. Let us remind here the considered setup.
\begin{itemize}
\item $\Omega \subset \R^2$ is a bounded Lipschitz domain.
\item $X = L^1(\Omega)$.
\item $Y = \L(H^{\frac 1 2}(\partial \Omega), H^{-\frac 1 2}(\partial \Omega))$.
\item $M = \{ \sigma_T = 1 + (k-1)\chi_{T} : T \in \triangle^2, T \subseteq \Omega, \dist(T,\partial \Omega) > d_0/2\}$ for some fixed $d_0 > 0$ and $k >0$, $k \neq 1$.
\item $K = \{\sigma_T \in M : \dist(T,\partial \Omega) \ge d_0, |T| \ge d_1 \}$, for some $d_1 >0$.
\item $P^1_N\colon H^{\frac 1 2}(\partial \Omega) \to H^{\frac 1 2}(\partial \Omega)$ and $P^2_N\colon H^{-\frac 1 2}(\partial \Omega) \to H^{-\frac 1 2}(\partial \Omega)$ are bounded linear maps for $N \in \N$, such that $(P^j_N)^* = P^j_N$, $j=1,2,$ and $P^1_N \to I_{H^{\frac 1 2}(\partial \Omega)}$ and $P^2_N \to I_{H^{-\frac 1 2}(\partial \Omega)}$ strongly as $N \to +\infty$. In particular, since $H^s(\partial \Omega)$ is a Hilbert space for $s \in \R$, the $P^j_N$ can be chosen as orthogonal projections onto the space spanned by the first N elements of any orthonormal bases.
\item $Q_N y = P^2_N y P^1_N$ for $y \in Y$, as in Example \ref{ex:gal}.
\end{itemize}

We recall the statement of our main result for Calder\'on's problem, Theorem~\ref{thm:cald}, the reader's convenience.

\begin{theorem7}
Under the above assumptions, there exists $C>0$ depending only on $\Omega$, $d_0$, $d_1$ and $k$ such that 
\begin{equation}\label{est:Cald}
\|\sigma_1-\sigma_2\|_{L^1(\Omega)} \le C \|Q_N(\Lambda_{\sigma_1})-Q_N(\Lambda_{\sigma_2})\|_{H^{\frac 1 2}(\partial \Omega) \to H^{-\frac 1 2}(\partial \Omega)},\qquad \sigma_1, \sigma_2 \in K,
\end{equation}
for every sufficiently large $N \in \N$.
\end{theorem7}

We recall a result that will be used several times in the proof. First we define the set of triangles associated to our conductivities:
\[
\A = \{  T : T \in \triangle^2, T \subseteq \Omega, \dist(T,\partial \Omega) \ge d_0, |T| \ge d_1\}.
\]
By an abuse of notation, we also define $\partial \A = \{\partial T : T \in \A\}$.
We also recall that given two bounded non empty subsets $A$ and $B$ of $\overline \Omega$, the Hausdorff distance $d_H(A,B)$ is defined as:
\[
d_H(A,B) = \max \left\{\sup _{x \in A} \inf _{y \in B} \operatorname{dist}(x, y), \sup _{y \in B} \inf _{x \in A} \operatorname{dist}(y, x)\right\}.
\]

\begin{proposition} \label{prop:33}
There exists $C>0$ depending only on $\Omega$, $d_0$, $d_1$ and $k$ such that, if $T_1,T_2 \in \A$,
then the vertices of $T_1$ and $T_2$ can be ordered in such a way that
\[
\|v_1-v_2\|_{\triangle^2} \le C d_H(\partial T_1,\partial T_2),
\]
where $v_1,v_2$ are the matrices of the vertices associated to $T_1, T_2$.
\end{proposition}
\begin{proof}
By \cite[Proposition 3.3]{beretta2019}, there exist $\delta_0 >0$ and $C_0>0$ depending only on $\Omega$, $d_0$, $d_1$ and $k$ such that, if $T_1,T_2 \in \A$  satisfy
\[
d_H(\partial T_1,\partial T_2) \le \delta_0,
\]
then the vertices of $T_1$ and $T_2$ can be ordered in such a way that
\[
\|v_1-v_2\|_{\triangle^2} \le C_0 d_H(\partial T_1,\partial T_2),
\]
where $v_1,v_2$ are the matrices of the vertices associated to $T_1, T_2$.

If $d_H(\partial T_1,\partial T_2) \le \delta_0$, the estimate is proven. Otherwise, suppose $d_H(\partial T_1,\partial T_2) > \delta_0$. Then
\[
\|v_1-v_2\|_{\triangle^2} \le \|v_1\|_{\triangle^2}+ \|v_2\|_{\triangle^2}\le C(\Omega) \frac{\delta_0}{\delta_0}\le \frac{C(\Omega)}{\delta_0}d_H(\partial T_1,\partial T_2).
\]
It is sufficient to take $C=\max(C_0,\frac{C(\Omega)}{\delta_0})$.
\end{proof}

We are now ready to prove \Cref{thm:cald}.
\begin{proof}[Proof of \Cref{thm:cald}]
We need to verify that the assumptions of Theorem~\ref{FM-2} are satisfied.

The results in Section~\ref{sub:simplex} can be easily adapted to show that 
\[
M = \{ \sigma_T = 1 + (k-1)\chi_{T} : T \in \triangle^2, T \subseteq \Omega, \dist(T,\partial \Omega) > d_0/2\}
\]
is a six-dimensional differentiable manifold Lipschitz in $L^1(\Omega)$. The main difference is that we define the atlas as follows: for $T \in \triangle^2$ such that  $T \subseteq \Omega$ and  $\dist(T,\partial \Omega) > d_0/2$ let 
\begin{equation*}
U_T = \left\{\sigma_{T_v}\in M\,:\,\|v^T-v\|_{\triangle^2}<R_T\right\},
\end{equation*}
where $R_T=\frac{1}{4}\min_{i\neq j}|v^{T}_i-v^{T}_j|$. The charts $\varphi_T\colon U_T\to\R^{2\times 3}\approx\R^{6}$ are defined by
\begin{equation}\label{def:atlcal}
	\varphi_T(\sigma_{T_v})
	=
	v,	
	\qquad
	\text{such that}
	\qquad \|v^T-v\|_{\triangle^2}<R_T.
\end{equation}

We now prove that $K= \{\sigma_T \in M : \dist(T,\partial \Omega) \ge d_0, |T| \ge d_1 \}$ is compact as a subset of $L^1(\Omega)$. First we show that $\partial \A$ is compact in the topology induced by the Hausdorff distance. It is known that the set of non-empty, closed and bounded subsets of $\overline \Omega$ is a compact set in the Hausdorff distance topology, since $\overline \Omega$ is compact \cite{henrikson1999completeness}. The set $\partial\A$ is closed in the Hausdorff distance topology because, given a converging sequence, thanks to Proposition~\ref{prop:33} we can order the vertices in such a way that they converge to the vertices of a triangle in $\partial\A$, which is therefore the limit of the sequence. The closedness of $\partial\A$ immediately yields its compactness. Now, given a sequence $(\sigma_{T_j})_j \subseteq K$, we consider the associated sequences $\{\partial T_j\}_j\subseteq \partial\A$ of triangles. Then we can extract a converging subsequence $\{\partial T_{j_m}\}_m$ converging to $\partial T^*\in \partial\A$ according to the Hausdorff distance. Hence, for $m$ sufficiently large, we can again apply Proposition~\ref{prop:33} and order the vertices of $T_{j_m}$ such that they converge to the vertices of $T^*$ in the $\triangle^2$ norm. Finally, thanks to inequality \eqref{Lipschitz-triangles} we have that $\sigma_{T_{j_m}} \to \sigma_{T^*}$ in $L^1$ as $m \to +\infty$.

The regularity properties of the map $\Lambda_\sigma$ over the manifold $M$ have been studied in \cite[Corollary 4.5]{beretta2017differentiability} and \cite{beretta2019}. In particular, \cite[Lemma 3.5]{beretta2019} (see also \cite[Lemma 4.4]{beretta2017differentiability}) shows the continuity of the partial derivative of the Dirichlet-to-Neumann map with respect to the position of the vertices of a polygon. From standard results of analysis \cite[Theorem 1.9]{ambrosetti1995}, this yields that the DN map (composed with the charts) is in fact differentiable with continuous derivative, which in our setting is equivalent to $\sigma \mapsto \Lambda_\sigma \in C^1(M,Y)$.

Let $\tilde Y = \{ y \in Y: y \text{ is a compact operator}\}$. It is well known \cite{mandache2001} that $\Lambda_{\sigma_2}-\Lambda_{\sigma_1} \in \tilde Y$ for every $\sigma_1,\sigma_2 \in M$, because $\sigma_2 (x) = \sigma_1 (x) = 1$ for $x \in \Omega$ with $\dist(x, \partial \Omega) < d_0/2$, by assumption. Moreover $\tilde Y$ is closed, therefore $\ran (d\Lambda_\sigma) \subseteq \tilde Y$ for every $\sigma \in M$, as shown in Remarks~\ref{rem:frechet} and \ref{ref:deriv_compact}.

Now, from \cite[Theorem 2.1]{beretta2019} we have that for every $T_1, T_2 \in \A$ 
\begin{equation}\label{est:estH}
d_H(\partial T_1, \partial T_2) \le c \|\Lambda_{\sigma_1}-\Lambda_{\sigma_2}\|_{H^{\frac 1 2}(\partial \Omega) \to H^{-\frac 1 2}(\partial \Omega)},
\end{equation}
where $\sigma_j = 1 + (k-1)\chi_{T_j}$, $j=1,2$, with the constant $c$ depending only on $\Omega$, $d_0$, $d_1$ and $k$. For $\sigma_{T_1}, \sigma_{T_2} \in U_T$, for some $T \in \A$, we can combine Proposition~\ref{prop:33} with estimates \eqref{est:estH} and, after ordering the vertices of $T_1$ and $T_2$ accordingly, we get:
\begin{equation}\label{eq:almost_stability}
 \|v_1-v_2\|_{\triangle^2} \le c \|\Lambda_{\sigma_1}-\Lambda_{\sigma_2}\|_{H^{\frac 1 2}(\partial \Omega) \to H^{-\frac 1 2}(\partial \Omega)},
\end{equation}
for some constant $c>0$ depending only on $M$ and $K$.

We now claim that
\begin{equation}\label{eq:claim_order}
    \|\varphi_T(\sigma_1)-\varphi_T(\sigma_2) \|_{\triangle^2}\le \|v_1-v_2\|_{\triangle^2}.
\end{equation}
For each $i=1,2$, $\varphi_T(\sigma_i)$ and $v_i$ differ up to a permutation $\gamma_i$. By an abuse of notation, we write $v_i = \gamma_i(\varphi_T(\sigma_i))$. Since $\|v\|_{\triangle^2}=\|\gamma(v)\|_{\triangle^2}$ for all $v$ and $\gamma$, we have
\[
\|v_1-v_2\|_{\triangle^2} 
= \|\gamma_1(\varphi_T(\sigma_1))-\gamma_2(\varphi_T(\sigma_2))\|_{\triangle^2} = \|\gamma_2^{-1} (\gamma_1(\varphi_T(\sigma_1)))-\varphi_T(\sigma_2)\|_{\triangle^2}.
\]
If $\gamma_1=\gamma_2$, the claim follows. Otherwise, $\gamma_2^{-1} \circ\gamma_1$ is a nontrivial permutation, so that by \eqref{eq:permutation_RT} and \eqref{def:atlcal} we have
\[
\begin{split}
\|v_1-v_2\|_{\triangle^2} 
&= \|\gamma_2^{-1} (\gamma_1(\varphi_T(\sigma_1)))-v^T+v^T-\varphi_T(\sigma_2)\|_{\triangle^2}\\
&\ge \|\gamma_2^{-1} (\gamma_1(\varphi_T(\sigma_1)))-v^T\|_{\triangle^2}-\|v^T-\varphi_T(\sigma_2)\|_{\triangle^2}\\
&\ge 3R_T-R_T\\
&=2R_T\\
&\ge \|\varphi_T(\sigma_1)-v^T\|_{\triangle^2}+\|v^T-\varphi_T(\sigma_2)\|_{\triangle^2}\\
&\ge \|\varphi_T(\sigma_1)-\varphi_T(\sigma_2) \|_{\triangle^2},
\end{split}
\]
and \eqref{eq:claim_order} follows.

Combining \eqref{eq:almost_stability} and \eqref{eq:claim_order} yields condition (\ref{3b}b) of Theorem~\ref{FM-2}.
We can now apply Theorem~\ref{FM-2} in order to obtain the desired estimate \eqref{est:Cald}.
\end{proof}

\section{The Gel'fand-Calder\'on problem with spherical inclusions}\label{sec:gelf}

This section is devoted to the proof of \Cref{thm:gelf}.

\subsection{The manifold}

First, we introduce the manifold of potentials we use. Even if in this paper we will eventually deal only with the three-dimensional case, we provide the description in arbitrary dimension because it may turn out useful in other settings.

\subsubsection{The space of parameters}

Let $d \ge 2$, $A>0$ and $0<\varrho_0<\varrho_1<+\infty$ and define the following parameter space
\begin{equation*}
	\P_1
	=
	\{v^\top=(a^\top,r,\lambda)\in\R^{d+2}\,:\,a\in\R^d,\; |a|<A,\; \lambda,r\in\R,\;|\lambda|,r\in (\varrho_0,\varrho_1)\}.
\end{equation*}
Let $\Omega\subseteq\R^d$ be an open set such that $B(0,A+\varrho_1)\subseteq\Omega$, so that $B(a,r)\subseteq\Omega$ for every $(a^\top,r,\lambda)\in\P_1$.

For simplicity, in this section we denote by $C>0$ any constant depending only on $d$, $A$, $\varrho_0$ and $\varrho_1$, while additional dependences are denoted as subindexes.

Let $N\in\N$. We regard each element $V\in \P_1^N$ as a $(d+2)\times N$ matrix whose columns are vectors in $\P_1$. That is
\begin{equation*}
	V
	=
	(v_1,\ldots,v_N)
	=
	\begin{pmatrix}
	a_1 & a_2 & \cdots & a_N \\ r_1 & r_2 & \cdots & r_N \\ \lambda_1 & \lambda_2 & \cdots & \lambda_N
	\end{pmatrix}
	\in
	\P_1^N.
\end{equation*}
Given $V=(v_1,\ldots,v_N)\in\P_1^N$ we introduce the notation
\begin{equation*}
	\mathrm{vec}(V)^\top
	:\,=
	(v_1^\top,\ldots,v_N^\top)
	\in
	\R^{N(d+2)}.
\end{equation*}
In addition, we restrict the class of matrices in $\P_1^N$ to the matrices $V=(v_1,\ldots,v_N)\in \P_1^N$ such that
\begin{equation}\label{disjoint}
	r_k+r_\ell
	<
	|a_k-a_\ell|
	\ \text{ for every } \ k\neq\ell.
\end{equation}
We denote the set of matrices in $\P_1^N$ satisfying such condition by $\P_N$ . Constructed in this way, every matrix $V\in\P_N$ determines a  collection of pairwise disjoint balls $\{B(a_k,r_k)\,:\,k=1,\ldots,N\}$ in $\R^d$. Conversely, given a finite family of disjoint balls $B(a,r)$ with $|a|<A$ and $r\in(\varrho_0,\varrho_1)$, each of them assigned with an intensity $\lambda\in\R$ with $|\lambda|\in(\varrho_0,\varrho_1)$, there exists a unique matrix  $V\in\P_N$ (up to a permutation of its columns) representing the given set of balls and intensities.

In what follows we consider the matrix norm $\Norm{\cdot}$ in $\R^{N\times(d+2)}$ defined as
\begin{equation*}
	\Norm{V}
	=
	\Norm{(v_1,v_2,\ldots,v_N)}
	=
	\max\{|v_1|,|v_2|,\ldots,|v_N|\},
\end{equation*}
where $|\cdot|$ stands for the usual Euclidean norm of a vector in $\R^{d+2}$. Since $\R^{N\times(d+2)}$ is finite-dimensional, the norm $\Norm{\cdot}$ is equivalent to the usual Euclidean norm when the elements in $\R^{N\times(d+2)}$ are regarded as vectors in $\R^{N(d+2)}$. More precisely,
\begin{equation*}
	\frac{1}{\sqrt{N}}|\mathrm{vec}(V)|
	\le
	\Norm{V}
	\le
	|\mathrm{vec}(V)|.
\end{equation*}
On the other hand, since $v_k^\top=(a_k^\top,r_k,\lambda_k)\in\R^{d+2}$ for each $k=1,\ldots,N$,
\begin{equation}\label{a-r-lambda<V}
	|a_k|+|r_k|+|\lambda_k|
	\le
	\sqrt{3}\,|v_k|
	\le
	\sqrt{3}\,\Norm{V}.
\end{equation}

With an abuse of notation, if $\sigma\in\mathrm{Perm}(N)$ is a permutation of order $N$ then we denote the corresponding permutation of the column vectors of $V\in\P_N$ as
\begin{equation*}
	\sigma(V)
	=
	(v_{\sigma(1)},\ldots,v_{\sigma(N)})
\end{equation*}
and
\begin{equation*}
	\sigma(\mathrm{vec}(V))
	=
	\mathrm{vec}(\sigma(V))
	=
	(v_{\sigma(1)}^\top,\ldots,v_{\sigma(N)}^\top)^\top.
\end{equation*}
As a consequence, $\Norm{\sigma(V)}=\Norm{V}$ for every $V\in\P_N$ and every permutation $\sigma$.

Given any $V=(v_1,v_2,\ldots,v_N)\in\P_N$ we set
\begin{equation}\label{gamma}
	\gamma_V
	:\,=
	\max_{k\neq\ell}\bigg\{\frac{r_k+r_\ell}{|a_k-a_\ell|}\bigg\}
	\in
	(0,1).
\end{equation}
Notice that, since $r_k,r_l>\varrho_0$, we have
\begin{equation}\label{eq:gammabound}
	|a_k-a_\ell|
	\ge
	\frac{r_k+r_\ell}{\gamma_V}
	>
	\frac{2\varrho_0}{\gamma_V},\qquad k\neq\ell.
\end{equation}
In addition, we define
\begin{equation}\label{eta}
	\eta_V
	:\,=
	\min\bigg\{1\, ,\, \frac{1-\gamma_V}{\sqrt{3}\gamma_V}\bigg\}\in (0,1].
\end{equation}

\begin{lemma}\label{lemma-disjoint}
Let $V=(v_1,\ldots,v_N)\in\P_N$. If $V^1=(v^1_1,\ldots,v^1_N)$ and $V^2=(v^2_1,\ldots,v^2_N)$ are matrices in $\P_1^N$ satisfying
\begin{equation*}
	\Norm{V^i-V}
	<
	\eta_V\varrho_0,
	\qquad i=1,2,
\end{equation*}
with $\eta_V$ given by \eqref{eta}, then $V^1,V^2\in\P_N$ and
\begin{equation*}
	B(a^1_k,r^1_k)\cap B(a^2_\ell,r^2_\ell)
	=
	\emptyset,
	\qquad k\neq\ell,
\end{equation*}
where $(v^i_k)^\top=((a^i_k)^\top,r^i_k,\lambda^i_k)$ for $i=1,2$ and  $k=1,\ldots,N$.
\end{lemma}

\begin{proof}
 Our aim is to show that the inequality
\begin{equation}\label{disjoint 1-2}
	r^1_k+r^2_\ell
	<
	|a^1_k-a^2_\ell|
\end{equation}
holds for $k\neq\ell$. We observe that the fact that $V^1$ and $V^2$ belong to $\P_N$ is a direct consequence of this inequality. Indeed, if we assume for a moment that $V^2$ is equal to $V^1$, then $V^1$ satisfies \eqref{disjoint} and thus $V^1$ belongs to $\P_N$ by definition. To show \eqref{disjoint 1-2}, observe that adding and subtracting $r_k+r_\ell$ we can estimate
\begin{equation*}
	r^1_k+r^2_\ell-|a^1_k-a^2_\ell|
	\le
	r_k+r_\ell-|a^1_k-a^2_\ell|+|r^1_k-r_k|+|r^2_\ell-r_\ell|.
\end{equation*}
Since $V\in\P_N$, using \eqref{gamma} we obtain
\begin{equation*}
\begin{split}
	r_k+r_\ell-|a^1_k-a^2_\ell|
	\le
	~&
	\gamma_V|a_k-a_\ell|-|a^1_k-a^2_\ell|
	\\
	=
	~&
	-(1-\gamma_V)|a_k-a_\ell|+|a_k-a_\ell|-|a^1_k-a^2_\ell|
	\\
	\le
	~&
	-(1-\gamma_V)|a_k-a_\ell|+|a^1_k-a_k|+|a^2_\ell-a_\ell|,
\end{split}
\end{equation*}
and replacing in the previous inequality we get
\begin{equation*}
\begin{split}
	r^1_k+r^2_\ell&-|a^1_k-a^2_\ell|
	\\
	\le
	~&
	-(1-\gamma_V)|a_k-a_\ell|
	+|a^1_k-a_k|+|r^1_k-r_k|
	+|a^2_\ell-a_\ell|+|r^2_\ell-r_\ell|
	\\
	\le
	~&
	-(1-\gamma_V)|a_k-a_\ell|
	+\sqrt{3}\,\Norm{V^1-V}+\sqrt{3}\,\Norm{V^2-V},
\end{split}
\end{equation*}
where in the second inequality we have used \eqref{a-r-lambda<V}.
Then \eqref{disjoint 1-2} follows from the fact that
\begin{equation*}
	\Norm{V^i-V}
	<
	\eta_V\varrho_0
	\le
	\frac{1-\gamma_V}{\sqrt{3}\gamma_V}\varrho_0
	<
	\frac{1-\gamma_V}{2\sqrt{3}}|a_k-a_\ell|
	,\qquad i=1,2,
	\end{equation*}
which is a consequence of \eqref{eq:gammabound} and \eqref{eta}.
\end{proof}

\begin{lemma}\label{Lemma-uniqueness}
Let $V,\widetilde V\in\P_N$ such that $\Norm{\widetilde V-V}<\eta_V\varrho_0$. Then
\begin{equation*}
	\Norm{\sigma(\widetilde V)-V}>\eta_V\varrho_0
\end{equation*}
for every non trivial permutation $\sigma(\widetilde V)$ of the column vectors of $\widetilde V$.
\end{lemma}

\begin{proof}
Let us consider any permutation $\sigma\in\mathrm{Perm}(N)\setminus\{\mathrm{Id}\}$, so that there is an index $i=1,\ldots,N$ such that $\sigma(i)\neq i$. Using the notation $\sigma(V)=(v_{\sigma(1)},\ldots,v_{\sigma(N)})$, we have
\begin{equation*}
\begin{split}
	2\varrho_0
	& <
	r_{\sigma(i)}+r_i \\ &
	<
	|a_{\sigma(i)}-a_i| \\ &
	\le
	|v_{\sigma(i)}-v_i|
	\\ &
	\le
	|\widetilde v_{\sigma(i)}-v_{\sigma(i)}|+|\widetilde v_{\sigma(i)}-v_i|
	\\ &
	\le
	\Norm{\widetilde V-V}+|\widetilde v_{\sigma(i)}-v_i|
	\\ &
	<
	\eta_V\varrho_0+|\widetilde v_{\sigma(i)}-v_i|.
\end{split}
\end{equation*}
Since $\eta_V\le 1$, we have $\eta_V\le 2-\eta_V$, so that
\[
\eta_V\varrho_0 < |\widetilde v_{\sigma(i)}-v_i| \le \Norm{\sigma(\widetilde V)-V}. \qedhere
\]
\end{proof}

\subsubsection{The simple functions}

Each matrix $V\in \P_1^N$ yields a linear combination of indicator functions $q_V\in L^1(\Omega)$ defined as
\begin{equation} \label{def:qV}
	q_V
	=
	\sum_{k=1}^N\lambda_k\chi_{B(a_k,r_k)}.
\end{equation}
Moreover, for $V\in\P_N$, since by condition \eqref{disjoint} the balls $\{B(a_k,r_k)\}_{k=1,\ldots,N}$ are pairwise disjoint, we have
\begin{equation*}
	\|q_V\|_{L^p(\Omega)}
	=
	\bigg(\omega_d\sum_{k=1}^N|\lambda_k|^pr_k^d\bigg)^{1/p},
\end{equation*}
where $\omega_d=|B(0,1)|$.

In the following result we give $L^p$ continuity estimates for $q_V$ with $V\in\P_1^N$.

\begin{lemma}\label{q1-q2 in Lp}
Let $N\in\N$ and $1\le p<\infty$. There exists a constant $C>0$ (depending only on $A$, $\varrho_0$, $\varrho_1$, $N$ and $p$) such that
\begin{equation}\label{q1-q2}
	\|q_{V^1}-q_{V^2}\|_{L^p(\Omega)}
	\le
	C\Norm{V^1-V^2}^{1/p}
\end{equation}
for every $V^1=(v^1_1,\ldots,v^1_N)$ and $V^2=(v^2_1,\ldots,v^2_N)$ in $\P_1^N$.
\end{lemma}

\begin{proof}
Using the triangle inequality we get
\begin{equation*}
\begin{split}
	\|q_{V^1}-q_{V^2}&\|_{L^p(\Omega)}
	\le
	\sum_{k=1}^N\|\lambda^1_k\chi_{B(a^1_k,r^1_k)}-\lambda^2_k\chi_{B(a^2_k,r^2_k)}\|_{L^p(\Omega)}
	\\
	\le
	~&
	\sum_{k=1}^N
	\Big(|\lambda^2_k|\,\|\chi_{B(a^1_k,r^1_k)}-\chi_{B(a^2_k,r^2_k)}\|_{L^p(\Omega)}
	+
	|\lambda^1_k-\lambda^2_k|\,\|\chi_{B(a^1_k,r^1_k)}\|_{L^p(\Omega)}\Big).
\end{split}
\end{equation*}
We estimate the two terms in the sum separately. For the first term we recall that $|\lambda^2_k|<\varrho_1$ and \eqref{estimate} to obtain
\begin{equation*}
\begin{split}
	|\lambda^2_k|\|\chi_{B(a^1_k,r^1_k)}-\chi_{B(a^2_k,r^2_k)}\|_{L^p(\Omega)}
	=
	~&
	|\lambda^2_k||B(a^1_k,r^1_k)\triangle B(a^2_k,r^2_k)|^{1/p}
	\\
	\le
	~&
	C|(a^1_k,r^1_k)-(a^2_k,r^2_k)|^{1/p}
	\\
	\le
	~&
	C|(a^1_k,r^1_k,\lambda^1_k)-(a^2_k,r^2_k,\lambda^2_k)|^{1/p}.
\end{split}
\end{equation*}
For the other term,
\begin{equation*}
	\|\chi_{B(a^1_k,r^1_k)}\|_{L^p(\Omega)}
	=
	|B(a^1_k,r^1_k)|^{1/p}
	\le
	C,
\end{equation*}
where we have used that $r^1_k<\varrho_1$. On the other hand, since $(a^1_k,r^1_k,|\lambda^1_k|),(a^2_k,r^2_k,|\lambda^2_k|)\in B(0,A)\times(\varrho_0,\varrho_1)^2$,
\begin{equation*}
\begin{split}
	|\lambda^1_k-\lambda^2_k|
	\le
	~&
	|(a^1_k,r^1_k,\lambda^1_k)-(a^2_k,r^2_k,\lambda^2_k)|
	\\
	\le
	~&
	(2A+4\varrho_1)^{1-\frac{1}{p}}|(a^1_k,r^1_k,\lambda^1_k)-(a^2_k,r^2_k,\lambda^2_k)|^{1/p}.
\end{split}
\end{equation*}
Therefore 
\begin{equation*}
	|\lambda^1_k-\lambda^2_k|\,\|\chi_{B(a^1_k,r^1_k)}\|_{L^p(\Omega)}
	\le
	C|(a^1_k,r^1_k,\lambda^1_k)-(a^2_k,r^2_k,\lambda^2_k)|^{1/p}
\end{equation*}
and 
\begin{equation*}
\begin{split}
	\|q_{V^1}-q_{V^2}\|_{L^p(\Omega)}
	\le
	~&
	C\sum_{k=1}^N|(a^1_k,r^1_k,\lambda^1_k)-(a^2_k,r^2_k,\lambda^2_k)|^{1/p}
	\\
	\le
	~&
	CN\Norm{V^1-V^2}^{1/p},
\end{split}
\end{equation*}
where we have used the definition of $\Norm{\cdot}$ in the second inequality.
\end{proof}

\subsubsection{The manifold}

We denote by $M_N$ the collection of all the functions $q_V$, as defined in \eqref{def:qV}, parametrized by $V\in\P_N$, i.e.
\begin{equation*}
	M_N
	=
	\{q_V\,:\,V\in\P_N\}
	\subset
	L^1(\Omega).
\end{equation*}

Every function $q\in M_N$ determines a set $\V(q)\subset\P_N$ consisting of all the matrices $V\in\P_N$ such that $q_V=q$. Observe that by construction each matrix in $\V(q)$ can be obtained as a permutation of the column vectors of any other matrix of the set. More precisely, for any $V\in\P_N$ such that $q_V=q$,
\begin{equation*}
	\V(q)
	=
	\{\sigma(V)\,:\,\sigma\in\mathrm{Perm}(N)\}.
\end{equation*}
Equivalently, the $\V(q)$'s are the classes of equivalence in the quotient space $\P_N/\hspace{-3pt}\sim$, where $V\sim W$ if and only if $V$ and $W$ have the same column vectors arranged in different orders.

Next, we define an atlas for $M_N$ as follows: for each $V\in\P_N$ let $U_V\subset M_N$ be the set defined as
\begin{equation*}
	U_V
	:\,=
	\big\{q_{\tilde V}\in M_N\,:\, \tilde  V\in \P_N ,\; \Norm{\tilde V-V}<\eta_V\varrho_0 \big\}.
\end{equation*}
Observe that, by \Cref{Lemma-uniqueness},  the matrix $\tilde V$ is unique, so in particular the map 
\begin{equation*}
	\varphi_V\colon U_V\to\R^{N(d+2)},\qquad \varphi_V(q_{\tilde V})
	=
	\mathrm{vec}(\tilde V),
\end{equation*}
is well defined.

\begin{lemma}\label{varphi bi-lip}
For every $V\in \P_N$, the map $\varphi_V\colon U_V\subseteq L^1(\Omega)\to\R^{N(d+2)}$ is bi-Lipschitz.
\end{lemma}

\begin{proof}
Take  $q_1,q_2\in U_V$. There exist $V^1=(v^1_1,v^1_2,\ldots,v^1_N)\in\V(q_1)$ and $V^2=(v^2_1,v^2_2,\ldots,v^2_N)\in\V(q_2)$ such that
\begin{equation*}
	\Norm{V^i-V}
	<
	\eta_V\varrho_0,
	\qquad i=1,2,
\end{equation*}
so that $\varphi_V(q_i)
	=
	\mathrm{vec}(V^i)$ for $i=1,2$.
Then
\begin{equation*}
\begin{split}
	\|q_1-q_2\|_{L^1(\Omega)}
	=
	~&
	\int_\Omega|q_1-q_2|\ dx
	\\
	=
	~&
	\int_\Omega\bigg|\sum_{k=1}^N\lambda^1_k\chi_{B(a^1_k,r^1_k)}-\sum_{\ell=1}^N\lambda^2_\ell\chi_{B(a^2_\ell,r^2_\ell)}\bigg|\ dx
	\\
	=
	~&
	\sum_{k=1}^N
	\int_\Omega\big|\lambda^1_k\chi_{B(a^1_k,r^1_k)}-\lambda^2_k\chi_{B(a^2_k,r^2_k)}\big|\ dx.
\end{split}
\end{equation*}
Notice that in the third equality we have recalled that, by \Cref{lemma-disjoint}, $B(a^1_k,r^1_k)\cap B(a^2_\ell,r^2_\ell)\neq\emptyset$ if and only if $k=\ell$, and that $\{B(a^i_k,r^i_k)\}_k$ are pairwise disjoint for $i=1,2$.

By the bi-Lipschitz estimate \eqref{eq:biLip} obtained in \Cref{sect6.2},
\begin{equation*}
	\big\|\lambda^1_k\chi_{B(a^1_k,r^1_k)}-\lambda^2_k\chi_{B(a^2_k,r^2_k)}\big\|_{L^1(\Omega)}
	\asymp
	|(a^1_k,r^1_k,\lambda^1_k)-(a^2_k,r^2_k,\lambda^2_k)|
	=
	|v^1_k-v^2_k|
\end{equation*}
for $k=1,\ldots,N$.
Then
\begin{equation*}
	\|q_1-q_2\|_{L^1(\Omega)}
	\asymp
	\sum_{k=1}^N
	|v^1_k-v^2_k|
	\asymp
	\max_k |v^1_k-v^2_k|
	=
	\Norm{V^1-V^2},
	\end{equation*}
	and so
	\begin{equation*}
	\|q_1-q_2\|_{L^1(\Omega)}
	\asymp
	|\mathrm{vec}(V^1)-\mathrm{vec}(V^2)|
	=
	|\varphi_V(q_1)-\varphi_V(q_2)|,
\end{equation*}
and the proof is concluded.
\end{proof}

We have that $U_V\subseteq M_N$ for each $V\in\P_N$ and 
\begin{equation*}
	\bigcup_{V\in\P_N}U_V
	=
	M_N.
\end{equation*}
It turns out that $\A_N=\{(U_V,\varphi_V)\,:\,V\in\P_N\}$ forms an atlas for $M_N$.

\begin{lemma}\label{lem:manifgelf}
Consider $M_N$ together with the atlas $\A_N$. Then $M_N$ is an $N(d+2)$-dimensional differentiable manifold in $L^1(\Omega)$. Moreover, $M_N$ is Lipschitz in $L^1(\Omega)$.
\end{lemma}

\begin{proof}
By the definition of differentiable manifold (see \Cref{def:manifold}), we need to show that for each $V,W\in\P_N$ the following hold:
\begin{enumerate}
\item $U_V$ is open with respect to the topology of $M_N$ inherited from $L^1(\Omega)$;
\item $\varphi_V(U_V)$ is open with respect to the topology of $\R^{N(d+2)}$;
\item $\varphi_V\colon U_V\to\R^{N(d+2)}$ is a homeomorphism onto its image for each $V\in\P_N$;
\item the transition maps
\begin{equation*}
	\varphi_{W}\circ\varphi_V^{-1}\colon\varphi_V(U_V\cap U_W)\to\varphi_{W}(U_V\cap U_{W})
\end{equation*}
are continuously differentiable.
\end{enumerate}

Proof of (1): 
Let $q_0\in U_V$. We show that $q_0$ is an interior point of $U_V$. Assume by contradiction that for every $j\ge 1$ there exists $q_j\in M_N$ such that
\begin{equation*}
	\|q_j-q_0\|_{L^1(\Omega)}
	<
	\frac{1}{j}
	\qquad \text{ and } \qquad
	q_j\notin U_V.
\end{equation*}
If we write $q_j=\sum_{k=1}^N\lambda^j_k\chi_{B(a^j_k,r^j_k)}$ and $q_0=\sum_{k=1}^N\lambda_k\chi_{B(a_k,r_k)}$ then
\begin{equation*}
	\int_\Omega\bigg|\sum_{k=1}^N\left(\lambda^j_k\chi_{B(a^j_k,r^j_k)}-\lambda_k\chi_{B(a_k,r_k)}\right)\bigg|\ dx
	<
	\frac{1}{j},\qquad j\ge 1,
\end{equation*}
and as a consequence (relabeling the terms in the sums if necessary) we deduce that $a^j_k\to a_k$,  $r^j_k\to r_k$ and $\lambda^j_k\to\lambda_k$. That is, the corresponding matrices $V_{q_j}$ converge component by component to $V_{q_0}$, so $\Norm{V_{q_j}-V_{q_0}}\to 0$ as $j\to\infty$.

Since $q_j\notin U_V$, then $\Norm{V_{q_j}-V}\ge\eta_V\varrho_0$, and by the triangle inequality,
\begin{equation*}
	0
	<
	\eta_V\varrho_0-\Norm{V_{q_0}-V}
	\le
	\Norm{V_{q_j}-V_{q_0}},\qquad j\in\N.
\end{equation*}
Since the right-hand side converges to zero as $j\to\infty$, we reach a contradiction.
\\

Proof of (2): It is immediate to see that
\begin{equation*}
\begin{split}
	\varphi_V(U_V)
	&=
	\{\mathrm{vec}(\widetilde V)\,:\,\widetilde V\in\P_N\ \text{and}\ \Norm{\widetilde V-V}<\eta_V\varrho_0\}\\
	&=
	\{\mathrm{vec}(\widetilde V)\,:\,\widetilde V\in\P_1^N\text{, $\widetilde V$ satisfies \eqref{disjoint} and}\ \Norm{\widetilde V-V}<\eta_V\varrho_0\}
	\end{split}
\end{equation*}
is open in $\R^{N(d+2)}$.
\\

Proof of (3): The fact that $\varphi_V$ is a homeomorphism is a direct consequence of the fact that $\varphi_V$ is bijective and bi-Lipschitz (\Cref{varphi bi-lip}).
\\

Proof of (4): Let $q\in U_V\cap U_W$. There exist $V_q,W_q\in\V(q)$ such that $\mathrm{vec}(V_q)\in\varphi_V(U_V\cap U_{W})$ and $\mathrm{vec}(W_q)\in\varphi_{W}(U_V\cap U_{W})$. Then
\begin{equation*}
	\mathrm{vec}(W_q)
	=
	\varphi_{W}\circ\varphi_V^{-1}(\mathrm{vec}(V_q)).
\end{equation*}
Since both $V_q$ and $W_q$ are in $\V(q)$, there exists a permutation $\sigma$ such that $W_q=\sigma(V_q)$, so
\begin{equation*}
	\varphi_{W}\circ\varphi_V^{-1}(\mathrm{vec}(V_q))
	=
	\mathrm{vec}(\sigma(V_q))
	=
	\sigma(\mathrm{vec}(V_q)),
\end{equation*}
where with an abuse of notation $\sigma(\mathrm{vec}(V)):\,=\mathrm{vec}(\sigma(V))$.
Moreover, the permutation $\sigma$ is independent of $q$. To see this observe that by \Cref{varphi bi-lip} both $\varphi_V$  and $\varphi_W$ are bi-Lipschitz, so the function $\varphi_{W}\circ\varphi_V^{-1}$ is continuous, and since the number of possible permutations in  $\mathrm{Perm}(N)$ is finite we have that
\begin{equation*}
	\varphi_{W}\circ\varphi_V^{-1}
	\equiv
	\sigma,
\end{equation*}
which is continuously differentiable.
\end{proof}

\begin{lemma}\label{lem:manifdisj}
$M_{N_1}\cap M_{N_2}=\emptyset$ for every $N_1\neq N_2$.
\end{lemma}

\begin{proof}
This immediately follows by construction of $M_N$, since for $V\in \P_N$ we have
\begin{equation*}
	q_V
	=
	\sum_{k=1}^N\lambda_k\chi_{B(a_k,r_k)},
\end{equation*}
where $\lambda_k\neq 0$ for every $k$. In other words, functions in $M_{N_1}$ and in $M_{N_2}$ are linear combinations of indicator functions of a different number of balls.
\end{proof}

\subsubsection{The tangent space $T_qM_N$}
Arguing as in \Cref{example-1}, it is possible to show that the manifold $M_N$ is not embedded in $L^1(\Omega)$. If $F\colon M_N\to Y$ is a differentiable function,  the differential of $F$ at $q_V$ is the linear map $dF_{q_V}\colon T_{q_V}M_N\to Y$ given by
\begin{equation*}
	dF_{q_V}
	=
	(F\circ\varphi_V^{-1})'(\varphi_V(q_V))
	=
	(F\circ\varphi_V^{-1})'(\mathrm{vec}(V)).
\end{equation*}
Thus, given $D\in\R^{(d+2)\times N}$ with an abuse of notation we denote
\begin{equation*}
	dF_{q_V}(D)
	:=
	dF_{q_V}(\mathrm{vec}(D))
	=
	(F\circ\varphi_V^{-1})'(\mathrm{vec}(V))\mathrm{vec}(D).
\end{equation*}
Note that, without loss of generality, we have chosen the chart $(U_V,\varphi_V)$ to compute the differential (see \Cref{sub:diff}).

\subsection{The inverse problem for the Schr\"odinger equation with indicator functions on balls}\label{sub:schr}
For the rest of this section, we assume $d=3$. However, we prefer to leave $d$ in all the mathematical expressions, in order to hightlight that, unless otherwise stated,  the derivation is valid in any dimension.

\subsubsection{The PDE model}
Consider the differentiable manifold constructed in the previous section,
\begin{equation*}
	M_N
	=
	\{q_V\,:\,V\in\P_N\}
	\subset
	L^1(\Omega),
\end{equation*}
with the atlas $\A_N=\{(U_V,\varphi_V)\,:\,V\in\P_N\}$ where
\begin{equation*}
	U_V
	:\,=
	\big\{q_{\tilde V}\in M_N\,:\, \tilde  V\in \P_N ,\; \Norm{\tilde V-V}<\eta_V\varrho_0 \big\}.
\end{equation*}
and
\begin{equation*}
	\varphi_V\colon U_V\to\R^{N(d+2)},\qquad \varphi_V(q_{\tilde V})
	=
	\mathrm{vec}(\tilde V).
\end{equation*}
We  consider $M_N$ as a subset of $L^1(\Omega)$, where $\Omega\subseteq\R^d$ is a bounded Lipschitz domain containing the support of all the elements in $M_N$, that is $\Omega\supseteq B(0,A+\varrho_1)$.
As we showed in \Cref{varphi bi-lip}, $M_N$ is Lipschitz in $L^1(\Omega)$.

Let $\beta\in L^\infty(\Omega)$ be a known background potential. In this section we consider the Dirichlet problem for the Schr\"odinger equation
\begin{equation}\label{eq:dir}
	\left\{\begin{array}{rl}
	-\Delta u+(\beta+q)u=0 & \text{ in } \Omega,
	\\
	u=\phi & \text{ on } \partial\Omega,
	\end{array}\right.
\end{equation}
with $\phi\in H^{1/2}(\partial\Omega)$ and $q\in M_N$.
In order to ensure the well-posedness of the problem, we need to restrict the functions $q$ in $M_N$ to those for which $0$ is not a Dirichlet eigenvalue of $(-\Delta+\beta+q)$ in $\Omega$. For that reason, instead of working with $M_N$ we consider
\begin{equation*}
	\widetilde M_N
	=
	\{q\in M_N\,:\,0 \text{ is not a Dirichlet eigenvalue of } (-\Delta+\beta+q) \text{ in } \Omega\}.
\end{equation*}
We shall see as a consequence of the following lemmas that $\widetilde M_N$ is open in $M_N$, and so it is a manifold itself, Lipschitz in $L^1(\Omega)$.

\begin{lemma}\label{Lmu invertible}
Take $V\in \P_N$. For $D\in\R^{(d+2)\times N}$ such that $\Norm{D}\leq1$, let $L_D\colon H^1_0(\Omega)\to H^{-1}(\Omega)$ be the linear operator defined as
\begin{equation*}
	L_D
	=
	-\Delta+\beta+q_{V+D}+s,
\end{equation*}
where $s=\|\beta\|_{L^\infty(\Omega)}+\varrho_1+2$. Then there exists $c>0$ such that for every $D\in\R^{(d+2)\times N}$ such that $\Norm{D}\leq1$ we have that $L_D$ is invertible and 
\begin{equation}\label{Lmu<C}
	\|L_D\|_{\L(H_0^1(\Omega),H^{-1}(\Omega))}
	\le
	c
	\qquad \text{ and }\qquad
	\|L_D^{-1}\|_{\L(H^{-1}(\Omega),H_0^1(\Omega))}
	\le
	c.
\end{equation}
\end{lemma}

\begin{proof}

Observe that
\begin{equation*}
\begin{split}
	\big|\prodin{L_D u}{v}_{H^{-1}(\Omega)\times H^1_0(\Omega)}\big|
	=
	~&
	\bigg|\int_\Omega \nabla u\cdot\nabla v+(\beta+q_{V+D}+s)uv\, dx\bigg|
	\\
	\le
	~&
	\int_\Omega|\nabla u|\,|\nabla v|\ dx
	+
	\|\beta+q_{V+D}+s\|_{L^\infty(\Omega)}
	\int_\Omega |uv|\ dx
	\\
	\le
	~&
	3s\|u\|_{H^1(\Omega)}\|v\|_{H^1(\Omega)}
\end{split}
\end{equation*}
for every $u,v\in H^1_0(\Omega)$. On the other hand,  since
\begin{equation*}
	|q_{V+D}|
	\le
	\max\{|\lambda_k+t_k|\,:\,k=1,\ldots,N\}
	\le \max_k|\lambda_k|+\Norm{D}
	\le \rho_1+1
\end{equation*}
in $\Omega$, then
\begin{equation*}
\begin{split}
	\beta+q_{V+D}+s
	&=
	\beta+\|\beta\|_{L^\infty(\Omega)}
	+
	q_{V+D}+\rho_1+2\\
	&\ge
	q_{V+D}+\rho_1+2\\
	&\ge \rho_1+2-(\rho_1+1)\\
	&\ge 1.
	\end{split}
\end{equation*}
Then
\begin{equation*}
\begin{split}
	\prodin{L_D u}{u}
	=
	~&
	\int_\Omega\big(|\nabla u|^2+(\beta+q_{V+ D}+s)u^2\big)\ dx
	\\
	\ge
	~&
	\int_\Omega|\nabla u|^2\ dx
	+
	\int_\Omega u^2\ dx
	\\
	=
	~&
	\|u\|_{H^1(\Omega)}^2.
\end{split}
\end{equation*}
Thus the operator $L_D$ is bounded and coercive, and by the Lax-Milgram theorem $L_D$ is invertible and there exists a constant $c>0$ such that the bounds in \eqref{Lmu<C} hold.
\end{proof}

\begin{definition}
Let $X$ be a Banach space. A sequence $\{T_j\}_j$ of bounded linear operators $T_j\colon X\to X$ is \emph{collectively compact} if the set $\{T_j(x)\,:\,\|x\|_X\le 1,\ j\in\N\}$ is precompact, that is, if its closure is compact.
\end{definition}

Even if the following proof requires $d\ge 3$, by using a different Sobolev embedding it is easy to see that the following result holds true even when $d=2$.

\begin{lemma}\label{lemma-op-convergence}
Take $V\in \P_N$. Let $D_j\in\R^{(d+2)\times N}$  be such that $\Norm{D_j}\leq1$ and $D_j\to 0$. Then the sequence of operators $\{sL_{D_j}^{-1}I\}_j$ is collectively compact and converges pointwise to the operator $sL_0^{-1}I\colon H^1_0(\Omega)\to H^1_0(\Omega)$, where $I$ is the compact embedding  $H^1_0(\Omega)\to H^{-1}(\Omega)$.
\end{lemma}

\begin{proof}

\textit{Step 1: a useful bound.} For $q\in L^{\frac{d}{2}}(\Omega)$  and $u\in H^1(\Omega)$ we 
have 
\[
\begin{split}
\|qu\|_{H^{-1}(\Omega)} &=\sup_{\|v\|_{H^1_0(\Omega)}=1} |\langle qu,v\rangle_{H^{-1}(\Omega)\times H^1_0(\Omega)}| \\
&\le \sup_{\|v\|_{H^1_0(\Omega)}=1} \int_\Omega |quv|\,dx \\
&= \sup_{\|v\|_{H^1_0(\Omega)}=1} \|quv\|_{L^1(\Omega)}.
\end{split}
\]
Thus, by the generalized H\"older inequality we can estimate 
\[
\|qu\|_{H^{-1}(\Omega)} \le  \sup_{\|v\|_{H^1_0(\Omega)}=1}
\|q\|_{L^{\frac{d}{2}}(\Omega)}
\|u\|_{L^{\frac{2d}{d-2}}(\Omega)}
\|v\|_{L^{\frac{2d}{d-2}}(\Omega)}.
\]
Next,  the Sobolev embedding $H^1(\Omega)\subseteq L^{\frac{2d}{d-2}}(\Omega)$ yields
\begin{equation}\label{eq:multiplier}
\begin{split}
\|qu\|_{H^{-1}(\Omega)}& \le  C(\Omega) \sup_{\|v\|_{H^1_0(\Omega)}=1}
\|q\|_{L^{\frac{d}{2}}(\Omega)}
\|u\|_{H^1(\Omega)}
\|v\|_{H^1_0(\Omega)}\\
&=  C(\Omega)
\|q\|_{L^{\frac{d}{2}}(\Omega)}
\|u\|_{H^1(\Omega)}.
\end{split}
\end{equation}

\textit{Step 2: $L_{D_j}^{-1}\to L_{0}^{-1}$ strongly.} Take $g\in H^{-1}(\Omega)$ and set $u_j=L_{D_j}^{-1}g$ and $u_0=L_{0}^{-1}g$. We need to show that $u_j\to u_0$ in $H^1_0(\Omega)$, namely,
\[
\|u_j-u_0\|_{H^1_0(\Omega)}\to 0.
\]
Since $L_{D_j} u_j=g=L_0u_0$ we have $L_0(u_0-u_j)=(L_{D_j}-L_0)u_j$, so that
\[
u_0-u_j = L_0^{-1} (L_{D_j}-L_0) L_{D_j}^{-1}g.
\]
Thanks to \eqref{Lmu<C} we obtain
\[
\|u_0-u_j\|_{H^1_0(\Omega)}  \le c^2 \|g\|_{H^{-1}(\Omega)}\|L_{D_j}-L_0\|_{\L(H_0^1(\Omega),H^{-1}(\Omega))}.
\]
It remains to bound the operator norm of $L_{D_j}-L_0$. For $u\in H^1_0(\Omega)$ we have $(L_{D_j}-L_0)u=(q_{V+D_j}-q_V)u$, and so by \eqref{eq:multiplier}
\[
\|L_{D_j}-L_0\|_{\L(H_0^1(\Omega),H^{-1}(\Omega))}\le C(\Omega)
\|q_{V+D_j}-q_V\|_{L^{\frac{d}{2}}(\Omega)}
\le C\Norm{D_j}^{\frac{2}{d}},
\]
where the second inequality comes from \Cref{q1-q2 in Lp}. Thus  the right hand side of this expression goes to $0$ as $j\to+\infty$, and the proof follows.

\textit{Step 3: $\{sL_{D_j}^{-1}I\}_j$ is collectively compact.} We follow the proof of \cite[Lemma~3]{2000-vogelius-etal}. We need to show that the set
\[
A=\{L_{D_j}^{-1}Iv:v\in H^1_0(\Omega),\; \|v\|_{H^1_0(\Omega)}\le 1,\;j\in\N\}
\]
is precompact in $H^1_0(\Omega)$. In other words, every  sequence in $A$ must have a subsequence convergent in $H^1_0(\Omega)$. Observe that
\[
A=\{L_{D_j}^{-1}f:f\in B,\;j\in\N\},
\]
where $B=I(\overline{B_{H^1_0(\Omega)}(0,1)})\subseteq H^{-1}(\Omega)$ is compact. Let $\{L_{D_{j_k}}^{-1}f_k\}_k$ be a sequence in $A$. Since $B$ is compact, up to a subsequence we have $f_k\to f$ in $H^{-1}(\Omega)$.
If $(j_k)_k$ is bounded, it admits a definitely constant subsequence, and so the result immediately follows from the convergence $f_k\to f$. Otherwise, take a subsequence so that $j_k\to+\infty$. By \eqref{Lmu<C} we have
\[
\begin{split}
\|L_{D_{j_k}}^{-1}f_k- L_{0}^{-1}f\|_{H^1_0(\Omega)}
&\le 
\|L_{D_{j_k}}^{-1}f_k- L_{D_{j_k}}^{-1}f\|_{H^1_0(\Omega)}
+\|L_{D_{j_k}}^{-1}f- L_{0}^{-1}f\|_{H^1_0(\Omega)}\\
&\le 
c\|f_k- f\|_{H^{-1}(\Omega)}
+\|L_{D_{j_k}}^{-1}f- L_{0}^{-1}f\|_{H^1_0(\Omega)},
\end{split}
\]
where the right hand side goes to $0$ by using Step 2. 
In other words,
\[
L_{D_{j_k}}^{-1}f_k\to L_{0}^{-1}f\in H^1_0(\Omega).
\]
This concludes the proof.
\end{proof}

\begin{lemma}[{\cite[Theorem 4.3]{anselone-1971}}]\label{lemma-col.compact-ops}
Let $X$ be a Banach space and $\{T_j\colon X\to X\}_j$ be a collectively compact family of bounded linear operators that converge pointwise to a linear operator $T\colon X\to X$. Then $I_X-T$ is an isomorphism if and only if there exists $j_0\in\N$ such that for every $j\ge j_0$    the operators $I_X-T_j$ are isomorphisms and $(I_X-T_j)^{-1}$ are uniformly bounded.
\end{lemma}

Next we show that the manifold $\widetilde M_N$ is open in $M_N$, which is a consequence of the previous results.
\begin{lemma}\label{lem:MNopen}
$\widetilde M_N$ is open in $M_N$. In particular, $\widetilde M_N$ is a manifold, and is Lipschitz in $L^1(\Omega)$.
\end{lemma}
\begin{proof}
Let $q_V\in\widetilde M_N$. Since $0$ is not a Dirichlet eigenvalue of $L_0-sI=-\Delta+\beta+q_V$ by the definition of $\widetilde M_N$, we have that the operator $L_0-sI$ is an isomorphism. As a consequence, since $L_0$ is invertible by \Cref{Lmu invertible}, then
\begin{equation*}
	I_{H^1_0(\Omega)}-sL_0^{-1}I
	=
	L_0^{-1}(-\Delta+\beta+q_V)
\end{equation*}
is an isomorphism.

It is now sufficient to show that there exist $\bar\mu\in (0,1]$ and $\bar c>0$ such that for every  $D\in\R^{(d+2)\times N}$ with $\Norm{D}\le\bar\mu$, the operator 
\begin{align}\label{eq:iso}
&	I_{H^1_0(\Omega)}-sL_D^{-1}I
	=
	L_D^{-1}(-\Delta+\beta+q_{V+ D})
	\quad \text{is an isomorphism,}\\
	\label{eq:iso2}
&	\text{and }\|(I_{H^1_0(\Omega)}-sL_{D}^{-1}I)^{-1}\|_{\L(H_0^1(\Omega),H_0^1(\Omega))}\le \bar c.
\end{align}
In particular, this implies that $q_{V+ D}\in\widetilde M_N$ (the uniform bound \eqref{eq:iso2} will be useful later). Suppose, by contradiction, that for every $j\ge 1$ there exists  $D_j$ with $\Norm{D_j}\le \frac{1}{j}$,  such that, either
\begin{equation*}
	I_{H^1_0(\Omega)}-sL_{D_j}^{-1}I
	=
	L_{D_j}^{-1}(-\Delta+\beta+q_{V+ D_j})
\end{equation*}
is not an isomorphism, or $\|(I_{H^1_0(\Omega)}-sL_{D_j}^{-1}I)^{-1}\|_{\L(H_0^1(\Omega),H_0^1(\Omega))}>j$.  By \Cref{lemma-op-convergence}, the family of operators $\{sL_{D_j}^{-1}I\}_j$ is collectively compact and converges pointwise to the operator $sL_0^{-1}I$. Thus, by \Cref{lemma-col.compact-ops} we have that $I_{H^1_0(\Omega)}-sL_{0}^{-1}I$ is not an isomorphism, a contradiction. 

We have shown that $\widetilde M_N$ is open in $M_N$. The rest is an immediate consequence of \Cref{lem:manifgelf}.
\end{proof}

\begin{lemma}\label{umu H1}
Take $V\in\P_N$ such that $q_V\in\widetilde M_N$.
There exist $\bar\mu,c>0$ such that for every $\mu\in [0,\bar\mu]$, $D\in\R^{(d+2)\times N}$ such that $\Norm{D}\leq1$ and $\phi\in H^{1/2}(\partial\Omega)$ the problem
\begin{equation}\label{DPmu}
	\left\{\begin{array}{rl}
	-\Delta u_\mu+(\beta+q_{V+\mu D}) u_\mu=0 & \text{ in } \Omega,
	\\
	u_\mu=\phi & \text{ on } \partial\Omega,
	\end{array}\right.
\end{equation}
has a unique weak solution $u_\mu\in H^1(\Omega)$ satisfying
\[
	\|u_\mu\|_{H^1(\Omega)}
	\le
	c\|\phi\|_{H^{1/2}(\partial\Omega)},
\]
and
\begin{equation*}
	\|u_\mu-u_0\|_{H^1(\Omega)}
	\le
	c\|u_0\|_{H^1(\Omega)}\mu^{2/d}\,\Norm{D}^{2/d}.
\end{equation*}
\end{lemma}

\begin{proof}
Let $\bar\mu\in (0,1]$ be as in the proof of \Cref{lem:MNopen}. For $\mu\in [0,\bar\mu]$, $0$ is not a Dirichlet eigenvalue of $-\Delta+\beta+q_{V+\mu D}$, and so \eqref{DPmu} has a unique solution $u_\mu\in H^1(\Omega)$. It remains to prove the bound on  $\|w_\mu\|_{H^1(\Omega)}$, where $w_\mu=u_\mu-u_0\in H^1_0(\Omega)$.

A direct calculation shows that $w_\mu$ satisfies
\begin{equation*}
	L_{\mu D} w_\mu
	=
	sw_\mu-(q_{V+\mu D}-q_V)u_0.
\end{equation*}
By \Cref{Lmu invertible} we obtain
\begin{equation*}
	 (I_{H^1_0(\Omega)}-
	sL_{\mu D}^{-1})w_\mu=-L_{\mu D}^{-1}\bigl((q_{V+\mu D}-q_V)u_0\bigr),
\end{equation*}
and, in view of \eqref{eq:iso2} and \Cref{Lmu invertible},
\[
\begin{split}
\|w_\mu\|_{H^1_0(\Omega)}&\le \|(I_{H^1_0(\Omega)}-
	sL_{\mu D}^{-1})^{-1}\|_{\L(H_0^1(\Omega),H_0^1(\Omega))} \|L_{\mu D}^{-1}\bigl((q_{V+\mu D}-q_V)u_0\bigr)\|_{H^1_0(\Omega)} \\
	&\le C
	\|(q_{V+\mu D}-q_V)u_0\|_{H^{-1}(\Omega)}.
\end{split}
\]
Therefore, by \eqref{eq:multiplier} we obtain
\begin{equation*}
\begin{split}
	\|w_\mu\|_{H^1(\Omega)}
	\le
	~&
	C
\|q_{V+\mu D}-q_V\|_{L^{\frac{d}{2}}(\Omega)}
\|u_0\|_{H^1(\Omega)}
	\\
	\le
	~&
	C\|u_0\|_{H^1(\Omega)}\mu^{2/d}\,\Norm{D}^{2/d},
\end{split}
\end{equation*}
where in the last inequality we have used \eqref{q1-q2}.
\end{proof}

We define $\Lambda\colon\widetilde M_N\to\L_*=\L(H^{1/2}(\partial\Omega),H^{-1/2}(\partial\Omega))$ as the Dirichlet-to-Neumann map
\begin{equation*}
	\Lambda_q(\phi)
	:\,=
	\frac{\partial u_q^\phi}{\partial\nu}\Big|_{\partial\Omega},
\end{equation*}
where $u^\phi_q\in H^1(\Omega)$ is the unique solution of \eqref{eq:dir} and the Neumann derivative is (weakly) defined as 
\begin{equation*}
	\prodin{\Lambda_q(\phi)}{\psi}_{H^{-1/2}(\partial\Omega)\times H^{1/2}(\partial\Omega)}
	=
	\int_\Omega(\nabla u^\phi_q\cdot\nabla\psi+qu^\phi_q\psi)\ dx,
	\qquad  \psi\in H^1(\Omega).
\end{equation*}

\subsubsection{The map $\Lambda$ is of class $C^1$}

The following lemma holds only in dimension $d=3$. We believe that extensions to higher dimension is possible, but with different proof techniques.

\begin{lemma}\label{lem:DNdif}
The Dirichlet-to-Neumann map $\Lambda$ is differentiable in $\widetilde M_N$. Moreover, for every $q=q_V\in \widetilde M_N$,  $d\Lambda_{q}\colon\R^{(d+2)\times N}\to\L_*$ is given by
\begin{multline}\label{dF}
	\prodin{d\Lambda_{q}(D)(\phi)}{\psi}_{H^{-1/2}(\partial\Omega)\times H^{1/2}(\partial\Omega)}
	\\
	=
	\sum_{k=1}^N\bigg\{
	t_k\int_{B(a_k,r_k)}u^\phi_q u^\psi_q\ dx
	+
	\lambda_k\int_{\partial B(a_k,r_k)}\Big(h_k\cdot\frac{x-a_k}{r_k}+\rho_k\Big)u^\phi_q u^\psi_q \ d\sigma(x)
	\bigg\},
\end{multline}
for every $\phi,\psi\in H^{1/2}(\partial\Omega)$, where $u^\phi_q$ and $u^\psi_q$ are the corresponding solutions to \eqref{eq:dir} and
\begin{equation}\label{eq:VD}
	V
	=
	\begin{pmatrix}
	a_1 & a_2 & \cdots & a_N \\ r_1 & r_2 & \cdots & r_N \\ \lambda_1 & \lambda_2 & \cdots & \lambda_N
	\end{pmatrix},\qquad 
	D
	=
	\begin{pmatrix}
	h_1 & h_2 & \cdots & h_N \\
	\rho_1 & \rho_2 & \cdots & \rho_N \\
	t_1 & t_2 & \cdots & t_N
	\end{pmatrix}
	\in
	\R^{(d+2)\times N}.
\end{equation}
\end{lemma}

\begin{proof}
\textit{Step I: setting up the problem.}
Take $q=q_V \in \widetilde M_N$, for some $V\in\P_N$. In order to show that $\Lambda\colon\widetilde M_N\to\mathcal L_* $ is differentiable at $q$, we need to prove that
\[
\Lambda\circ \varphi_V^{-1}\colon \varphi_V(U_V)\to \mathcal L_*
\]
is Fr\'echet differentiable. Moreover, we need to show that its differential is given by  \eqref{dF}. More precisely, we need to show that
\begin{equation*}
    \lim_{D\to 0} \frac{\|\Lambda_{q_{V+D}}-\Lambda_{q_V} - d\Lambda_q(D)\|_{\L_*}}{\Norm{D}}=0.
\end{equation*}
Equivalently, we will prove that
\begin{equation*}
    \lim_{\mu\to 0^+} \sup_D \frac{\|\Lambda_{q_{V+\mu D}}-\Lambda_{q_V} - \mu \,d\Lambda_q(D)\|_{\L_*}}{\mu}=0,
\end{equation*}
where the supremum is taken over all $D\in \R^{(d+2)\times N}$ such that $\Norm{D}=1$.

By definition of operator norm we have
\begin{equation*}
    \|\Lambda_{q_{V+\mu D}}-\Lambda_{q_V} - \mu \,d\Lambda_q(D)\|_{\L_*} =\sup_{\phi,\psi}|\langle (\Lambda_{q_{V+\mu D}}-\Lambda_{q_V} - \mu \,d\Lambda_q(D))\phi,\psi\rangle|,
\end{equation*}
where the supremum is taken over all $\phi,\psi\in H^{\frac12}(\partial\Omega) $ such that $\|\phi\|_{H^{\frac12}(\partial\Omega)}=\|\psi\|_{H^{\frac12}(\partial\Omega)}=1$. By the triangle inequality we have
\begin{multline*}
    \|\Lambda_{q_{V+\mu D}}-\Lambda_{q_V} - \mu \,d\Lambda_q(D)\|_{\L_*} \le \sup_{\phi,\psi}
    |\langle (\Lambda_{q_{V+\mu D}}-\Lambda_{q_V})\phi,\psi\rangle -\int_\Omega (q_{V+\mu D}-q_V) u^\phi_{q} u^\psi_{q}\,dx|\\
    +\sup_{\phi,\psi}
    |\langle \int_\Omega (q_{V+\mu D}-q_V) u^\phi_{q} u^\psi_{q}\,dx-\mu\langle d\Lambda_q(D)\phi,\psi\rangle|
\end{multline*}
We deal with these two terms separately.

\medskip
\textit{Step II: We show that
\[
\lim_{\mu\to0^+} \sup_{\phi,\psi,D} \frac{
    |\langle (\Lambda_{q_{V+\mu D}}-\Lambda_{q_V})\phi,\psi\rangle -\int_\Omega (q_{V+\mu D}-q_V) u^\phi_{q} u^\psi_{q}\,dx|}{\mu}=0.
\]
}
Recalling Alessandrini's identity \cite{alessandrini1988} we have
\begin{equation*}
	\prodin{(\Lambda_{q_{V+\mu D}}-\Lambda_{q_V})\phi}{\psi}
	=
	\int_\Omega(q_{V+\mu D}-q_V)\,u_{q_{V+\mu D}}^\phi u^\psi_{q}\ dx,
\end{equation*}
so that by \eqref{eq:multiplier} we obtain
\begin{equation*}
    \begin{split}
        |\langle (\Lambda_{q_{V+\mu D}}-\Lambda_{q_V})\phi,\psi\rangle &-\int_\Omega (q_{V+\mu D}-q_V) u^\phi_{q} u^\psi_{q}\,dx| \\
        &= 
        |\int_\Omega(q_{V+\mu D}-q_V)(u_{q_{V+\mu D}}^\phi-u_{q_{V}}^\phi)u^\psi_{q}\ dx|\\ 
        &\le \|(q_{V+\mu D}-q_V)u^\psi_{q}\|_{H^{-1}(\Omega)}\|u_{q_{V+\mu D}}^\phi-u_{q_{V}}^\phi\|_{H^1(\Omega)}  \\
        &\le C(\Omega) \|u^\psi_{q}\|_{H^1(\Omega)}\|q_{V+\mu D}-q_V\|_{L^{d/2}(\Omega)}\|u_{q_{V+\mu D}}^\phi-u_{q_{V}}^\phi\|_{H^1(\Omega)}.
    \end{split}
\end{equation*}
We estimate all these factors as follows:
\begin{itemize}
    \item $\|u^\psi_{q}\|_{H^1(\Omega)}\le C \|\psi\|_{H^{\frac12}(\partial\Omega)}=C$ because $q\in \widetilde M_N$, and so problem \eqref{eq:dir} is well posed;
    \item $\|q_{V+\mu D}-q_V\|_{L^{d/2}(\Omega)}\le C\mu^{2/d}\Norm{D}^{2/d}=C\mu^{2/d}$ thanks to \Cref{q1-q2 in Lp};
    \item $\|u_{q_{V+\mu D}}^\phi-u_{q_{V}}^\phi\|_{H^1(\Omega)}\le C\|u^\phi_{q}\|_{H^1(\Omega)}\mu^{\frac2d}\Norm{D}^{\frac2d}\le C\|\phi\|_{H^{\frac12}(\partial\Omega)}\mu^{\frac2d}=C\mu^{\frac2d}$ thanks to \Cref{umu H1}.
\end{itemize}
Altogether, we have
\begin{equation*}
   \sup_{\phi,\psi} \frac{|\langle (\Lambda_{q_{V+\mu D}}-\Lambda_{q_V})\phi,\psi\rangle -\int_\Omega (q_{V+\mu D}-q_V) u^\phi_{q} u^\psi_{q}\,dx|}{\mu}\le C\mu^{\frac4d-1}=C\mu^{\frac13}\to 0,
\end{equation*}
uniformly in $D$ with $\Norm{D}=1$, as desired.

\medskip
\textit{Step III: We show that
\[
\lim_{\mu\to0^+} \sup_{\phi,\psi,D}\frac{
    |\langle \int_\Omega (q_{V+\mu D}-q_V) u^\phi_{q} u^\psi_{q}\,dx-\mu\langle d\Lambda_q(D)\phi,\psi\rangle|}{\mu}=0.
\]
}
Writing
\begin{equation*}
	q_{V+\mu D}-q_V
	=
	\sum_{k=1}^N\big((\lambda_k+\mu t_k)\chi_{B(a_k+\mu h_k,r_k+\mu\rho_k)}-\lambda_k\chi_{B(a_k,r_k)}\big),
\end{equation*}
we immediately derive
\begin{equation*}
    \begin{split}
        &\int_\Omega (q_{V+\mu D}-q_V) u^\phi_{q} u^\psi_{q}\,dx\\
        &= \sum_{k=1}^N\bigg((\lambda_k+\mu t_k)\int_{B(a_k+\mu h_k,r_k+\mu\rho_k)}u^\phi_{q} u^\psi_{q}\ dx-
	\lambda_k\int_{B(a_k,r_k)}u^\phi_{q} u^\psi_{q}\ dx\bigg)\\
	&=\sum_{k=1}^N\lambda_k\bigg(\int_{B(a_k+\mu h_k,r_k+\mu\rho_k)}u^\phi_{q} u^\psi_{q}\ dx-\int_{B(a_k,r_k)}u^\phi_{q} u^\psi_{q}\ dx\bigg)
	\\
	&\qquad +\mu \sum_{k=1}^N t_k\int_{B(a_k+\mu h_k,r_k+\mu\rho_k)}u^\phi_{q} u^\psi_{q}\ dx.
    \end{split}
\end{equation*}
Therefore, recalling \eqref{dF}, it is sufficient to show that for every $k=1,\dots,N$
\begin{equation}\label{eq:IIIa}
    \lim_{\mu\to 0^+}\sup_{\phi,\psi,D}\left|\int_{B(a_k+\mu h_k,r_k+\mu\rho_k)} u^\phi_{q} u^\psi_{q}\ dx-\int_{B(a_k,r_k)}u^\phi_{q} u^\psi_{q}\ dx\right|=0
\end{equation}
and
\begin{multline}\label{eq:IIIb}
    \lim_{\mu\to 0^+}\sup_{\phi,\psi,D}\mu^{-1}\left|\int_{B(a_k+\mu h_k,r_k+\mu\rho_k)}u^\phi_{q} u^\psi_{q}\ dx-\int_{B(a_k,r_k)}u^\phi_{q} u^\psi_{q}\ dx\right.\\
    -\left.\mu \int_{\partial B(a_k,r_k)}\Big(h_k\cdot\frac{x-a_k}{r_k}+\rho_k\Big)u^\phi_{q} u^\psi_{q} \ d\sigma(x) \right|=0.
\end{multline}
It is immediate to see that \eqref{eq:IIIb} implies \eqref{eq:IIIa}, and so it remains to prove \eqref{eq:IIIb}.

Suitable changes of variables yield
\begin{equation}\label{first}
\begin{split}
&\int_{B(a_k+\mu h_k,r_k+\mu\rho_k)}u^\phi_{q} u^\psi_{q}\ dx-\int_{B(a_k,r_k)}u^\phi_{q} u^\psi_{q}\, dx
	\\
	=
	~&
\int_{B(0,1)}\big[(r_k+\mu\rho_k)^d(u^\phi_{q} u^\psi_{q})(a_k+\mu h_k+(r_k+\mu\rho_k)x)
	- r_k^d(u^\phi_{q} u^\psi_{q})(a_k+r_kx)\big]\, dx
	\\
	=
	~&
	\bigl((r_k+\mu\rho_k)^d-r_k^d\bigr)\int_{B(0,1)}(u^\phi_{q} u^\psi_{q})(a_k+\mu h_k+(r_k+\mu\rho_k)x)\, dx
	\\
	~&
	+r_k^d\int_{B(0,1)}\big[(u^\phi_{q} u^\psi_{q})(a_k+\mu h_k+(r_k+\mu\rho_k)x)-(u^\phi_{q} u^\psi_{q})(a_k+r_kx)\big]\, dx.
\end{split}
\end{equation}
Recalling that $r_k<\varrho_1$ and $|\rho_k|\leq\Norm{D}=1$, it is easy to see that 
\begin{equation*}
    (r_k+\mu\rho_k)^d-r_k^d-dr_k^{d-1}\mu\rho_k
    =
    O(\mu^2)
\end{equation*}
holds uniformly for every $\Norm{D}=1$. Then, since $$\left|\int_{B(0,1)}(u^\phi_{q} u^\psi_{q})(a_k+\mu h_k+(r_k+\mu\rho_k)x)\, dx \right|\le C\|u^\psi_{q}\|_{L^2(\Omega)}\|u^\phi_{q}\|_{L^2(\Omega)}=C,$$ we have
\begin{multline}\label{second}
    \bigl((r_k+\mu\rho_k)^d-r_k^d\bigr)\int_{B(0,1)}(u^\phi_{q} u^\psi_{q})(a_k+\mu h_k+(r_k+\mu\rho_k)x)\, dx\\
    =dr_k^{d-1}\mu\rho_k \int_{B(0,1)}(u^\phi_{q} u^\psi_{q})(a_k+\mu h_k+(r_k+\mu\rho_k)x)\, dx+O(\mu^2),
\end{multline}
where all the constants hidden in the $O$ symbols (here and below) are independent of $\phi$, $\psi$ and $D$.

We consider now a subdomain $\Omega'\Subset\Omega$ such that $B(a_k+\mu h_k,r_k+\mu \rho_k)\subseteq\Omega'$ for every $k$ and for every $\mu\in [0,\bar\mu]$, where $\bar\mu>0$ is given by \Cref{umu H1}. Thanks to classical regularity results for elliptic PDEs,  we have that $u^\phi_{q}, u^\psi_{q}\in C^{1,\alpha}(\overline{\Omega'})$ for some $\alpha\in (0,1)$ and
\begin{equation}\label{eq:C1}
    \|u^\phi_{q}\|_{C^{1,\alpha}(\overline{\Omega'})}\le C\|\phi\|_{H^{\frac12}(\partial\Omega)}=C,\qquad  \|u^\psi_{q}\|_{C^{1,\alpha}(\overline{\Omega'})}\le C\|\psi\|_{H^{\frac12}(\partial\Omega)}=C.
\end{equation}
In particular, $\|u^\phi_{q}u^\psi_{q}\|_{C^{1,\alpha}(\overline{\Omega'})}\le C$. Thus, applying the identity
\begin{equation*}
	|u(y)-u(x)-(y-x)\cdot\nabla u(x)|
	\le
	\|u\|_{C^{1,\alpha}(\overline{\Omega'})}|y-x|^{1+\alpha},\qquad x,y\in \overline{\Omega'} ,
\end{equation*}
to $u=u^\phi_{q}u^\psi_{q}$ we have
\begin{equation*}
\begin{split}
	\big|(u^\phi_{q}u^\psi_{q})(a_k+\mu h_k+(r_k+\mu\rho_k)x)&-(u^\phi_{q}u^\psi_{q})(a_k+r_kx)\bigr.\\
&\qquad \bigl.	-\mu(h_k+\rho_kx)\cdot\nabla(u^\phi_{q}u^\psi_{q})(a_k+r_kx)\big|
	\\ &
	\le
C(|h_k|+|\rho_k|)^{1+\alpha}\mu^{1+\alpha}
\\ &
	=O(\mu^{1+\alpha}),
\end{split}
\end{equation*}
where in the last inequality we have used that $\Norm{D}=1$. Therefore
\begin{multline}\label{third}
    \int_{B(0,1)}\Big[(u^\phi_{q} u^\psi_{q})(a_k+\mu h_k+(r_k+\mu\rho_k)x)-(u^\phi_{q} u^\psi_{q})(a_k+r_kx)\Big]\, dx\\
    =\mu\int_{B(0,1)} (h_k+\rho_kx)\cdot\nabla(u^\phi_{q}u^\psi_{q})(a_k+r_kx)\, dx+O(\mu^{1+\alpha})=O(\mu).
\end{multline}

Finally, using \eqref{first}, \eqref{second} and \eqref{third} the expression in \eqref{eq:IIIb} may be rewritten (after a suitable change of variables) as
\begin{multline}\label{almostthere2}
    dr_k^{-1}\rho_k \int_{B(a_k,r_k)}(u^\phi_{q} u^\psi_{q})(x)\, dx+O(\mu)\\
    + \int_{B(a_k,r_k)} \Big(h_k+\rho_k\frac{x-a_k}{r_k}\Big)\cdot\nabla(u^\phi_{q}u^\psi_{q})(x)\, dx+O(\mu^{\alpha})
    \\ -\int_{\partial B(a_k,r_k)}\Big(h_k\cdot\frac{x-a_k}{r_k}+\rho_k\Big)u^\phi_{q} u^\psi_{q} \ d\sigma(x).
\end{multline}
Furthermore, noting that
\begin{equation*}
	d\frac{\rho_k}{r_k} u^\phi_{q} u^\psi_{q}+\Big(h_k+\rho_k\frac{x-a_k}{r_k}\Big)\cdot\nabla(u^\phi_{q} u^\psi_{q})
	\\
	=
	\diver\Big(\Big(h_k+\rho_k\frac{x-a_k}{r_k}\Big)u^\phi_{q} u^\psi_{q}\Big),
\end{equation*}
thanks to the divergence theorem we obtain that the three integral terms in \eqref{almostthere2} cancel, and so \eqref{eq:IIIb} is proven. This concludes the proof.
\end{proof}

Next, we show that $\Lambda$ is of class $C^1$.

\begin{lemma}\label{lem:DNC1}
The Dirichlet-to-Neumann map is continuously differentiable, that is $\Lambda\in C^1(\widetilde M_N,\L_*)$. Furthermore, for every $q_{V^1}\in\widetilde M_N$ there exists $C>0$ such that
\begin{equation*}
	\|d\Lambda_{q_{V^1}}-d\Lambda_{q_{V^2}}\|_{\R^{(d+2)\times N}\to \L_*}
	\\
	\le
	C\,\Norm{V^1-V^2}^{\frac 2 d}
\end{equation*}
for every $V^2\in\P_N$ such that $q_{V^2}\in\widetilde M_N$.
\end{lemma}

\begin{proof}
For the sake of simplicity, we introduce a functional $\F\colon\P_N\times\R^{(d+2)\times N}\times W^{1,1}(\Omega)\to\R$ given by
\begin{equation*}
	\F(V,D,u)
	=
	\sum_{k=1}^N\bigg\{
	t_k\int_{B(a_k,r_k)}u\ dx
	+
	\lambda_k\int_{\partial B(a_k,r_k)}\Big(h_k\cdot\frac{x-a_k}{r_k}+\rho_k\Big)u\ d\sigma(x)
	\bigg\},
\end{equation*}
where $V$ and $D$ are given by \eqref{eq:VD}.
Observe that, defined in this way, we have $\prodin{d\Lambda_{q_V}(D)(\phi)}{\psi}=\F(V,D,u_{q_V}^\phi  u_{q_V}^\psi)$ by \eqref{dF}. Using the divergence theorem we can rewrite $\F(V,D,u)$ as
\begin{equation*}
	\F(V,D,u)
	=
	\sum_{k=1}^N\int_{B(a_k,r_k)}\Big[\Big(t_k+\lambda_kd\frac{\rho_k}{r_k}\Big)\,u
	+
	\lambda_k\Big(h_k+\rho_k\frac{x-a_k}{r_k}\Big)\cdot\nabla u\Big]\ dx
\end{equation*}
and by \eqref{a-r-lambda<V} we obtain the following estimate,
\begin{equation}\label{FVDu-estimate-2}
\begin{split}
	\big|\F(V,D,u)\big|
	\le
	~&
	C\sum_{k=1}^N\big(|h_k|+|\rho_k|+|t_k|\big)
	\int_{B(a_k,r_k)}\big(|u|+|\nabla u|\big)\ dx
	\\
	\le
	~&
	C\Norm{D}\,\|u\|_{W^{1,1}(\Omega)}.
\end{split}
\end{equation}

Take $V^1\in\P_N$ such that $q_{V^1}\in\widetilde M_N$. We will show that $d\Lambda$ is continuous in $q_{V^1}$. Let $V^2\in\P_N$ such that $\Norm{V^1-V^2}\le \bar\mu$, where $\bar\mu>0$ is given in \Cref{umu H1}, so that $q_{V^2}\in\widetilde M_N$. Given $\phi,\psi\in H^{1/2}(\partial\Omega)$ we consider the $H^1(\Omega)$-solutions $u_j=u_{q_{V^j}}^\phi$ and $v_j=v_{q_{V^j}}^\psi$ with $j=1,2$. Recalling the formula for the differential of the Dirichlet-to-Neumann map \eqref{dF} and the definition of $\F$,
\begin{equation*}
\begin{split}
	\big|\prodin{(d\Lambda_{q_{V^1}}-d\Lambda_{q_{V^2}})(D)(\phi)}{\psi}\big|
	=
	~&
	\big|\F(V^1,D,u_1v_1)-\F(V^2,D,u_2v_2)\big|
	\\
	\le
	~&
	\big|\F(V^1,D,u_1v_1)-\F(V^2,D,u_1v_1)\big|
	\\
	~&
	+
	\big|\F(V^2,D,u_1v_1-u_2v_2)\big|
\end{split}
\end{equation*}
By \Cref{umu H1} we have the following estimates, 
\begin{eqnarray*}
	&&
	\|u_1-u_2\|_{H^1(\Omega)}
	\le
	C \|\phi\|_{H^\frac12(\partial\Omega)}\Norm{V^1-V^2}^{\frac 2 d},
	\\
	&&
	\|v_1-v_2\|_{H^1(\Omega)}
	\le
	C \|\psi\|_{H^\frac12(\partial\Omega)}\Norm{V^1-V^2}^{\frac 2 d},
\end{eqnarray*}
for some $C>0$ independent of $V^2$, $\phi$ and $\psi$ (in the following, we will use the same letter $C$ to denote different positive constants independent of $V^2$, $\phi$ and $\psi$). Hence, by using \eqref{FVDu-estimate-2} we can get the estimate
\begin{equation*}
\begin{split}
	&\big|\F(V^2,D,u_1v_1-u_2v_2)\big|
	\\
	&\le
	C\|u_1v_1-u_2v_2\|_{W^{1,1}(\Omega)}\Norm{D}
	\\
	&\le
	C\big(\|u_1\|_{H^1(\Omega)}\|v_1-v_2\|_{H^1(\Omega)}+\|v_2\|_{H^1(\Omega)}\|u_1-u_2\|_{H^1(\Omega)}\big)\Norm{D}
	\\
	&\le
	C\|\phi\|_{H^\frac12(\partial\Omega)}\|\psi\|_{H^\frac12(\partial\Omega)}\Norm{D}\,\Norm{V^1-V^2}^{\frac 2 d}.
\end{split}
\end{equation*}

In what follows we focus on the difference $\F(V^1,D,u_1v_1)-\F(V^2,D,u_1v_1)$. Performing changes of variables we obtain
\begin{multline*}
	\F(V^1,D,u_1v_1)-\F(V^2,D,u_1v_1)
	=
	\sum_{k=1}^N
	t_k\bigg(
	\int_{B(a^1_k,r^1_k)}u_1v_1\ dx
	-
	\int_{B(a^2_k,r^2_k)}u_1v_1\ dx
	\bigg)
	\\
	+
	\sum_{k=1}^N\int_{\S^{d-1}}(h_k\cdot x+\rho_k)\big(\lambda^1_k(r^1_k)^{d-1}(u_1v_1)(a^1_k+r^1_kx)
	\\
	-\lambda^2_k(r^2_k)^{d-1}(u_1v_1)(a^2_k+r^2_kx)\big)\ d\sigma(x),
\end{multline*}
so taking absolute values and recalling \eqref{a-r-lambda<V} we get
\begin{equation*}
\begin{split}
	\big|\F(V^1,D,u_1,v_1)-\F(V^2,D,u_1,v_1)\big|
	\le
	~&
	\sum_{k=1}^N\big(|t_k|[\textbf{I}_k]+(|h_k|+|\rho_k|)[\textbf{II}_k]\big)
	\\
	\le
	~&
	\sum_{k=1}^N\big(|h_k|+|\rho_k|+|t_k|\big)\big([\textbf{I}_k]+[\textbf{II}_k]\big)
	\\
	\le
	~&
	C\Norm{D}\sum_{k=1}^N\big([\textbf{I}_k]+[\textbf{II}_k]\big),
\end{split}
\end{equation*}
where
\begin{eqnarray*}
	\displaystyle
	[\textbf{I}_k]
	&=&
	\int_{B(a^1_k,r^1_k)\triangle B(a^2_k,r^2_k)}|u_1v_1|\ dx,
	\\
	\displaystyle
	[\textbf{II}_k]
	&=&
	\int_{\S^{d-1}}\big|\lambda^1_k(r^1_k)^{d-1}(u_1v_1)(a^1_k+r^1_kx)-\lambda^2_k(r^2_k)^{d-1}(u_1v_1)(a^2_k+r^2_kx)\big|\ d\sigma(x).
\end{eqnarray*}

We claim that
\begin{equation}\label{claim2}
	[\textbf{I}_k],\ [\textbf{II}_k]
	\le
	C\|\phi\|_{H^\frac12(\partial\Omega)}\|\psi\|_{H^\frac12(\partial\Omega)}\Norm{V^1-V^2}
\end{equation}
for each $k=1,\ldots,N$. This implies that
\begin{multline*}
	\big|\prodin{(d\Lambda_{q_{V^1}}-d\Lambda_{q_{V^2}})(D)(\phi)}{\psi}\big|
	\\
	\le
	C\|\phi\|_{H^\frac12(\partial\Omega)}\|\psi\|_{H^\frac12(\partial\Omega)}\Norm{D}\,\Norm{V^1-V^2}^{\frac 2 d}
	\\
	+
	C\|\phi\|_{H^\frac12(\partial\Omega)}\|\psi\|_{H^\frac12(\partial\Omega)}\Norm{D}\,\Norm{V^1-V^2}.
\end{multline*}
Then
\begin{equation*}
	\|(d\Lambda_{q_{V^1}}-d\Lambda_{q_{V^2}})(D)\|_*
	\le
	C\Norm{D}\,\Norm{V^1-V^2}^{\frac 2 d}
\end{equation*}
for every $D\in\R^{(d+2)\times N}$. Therefore, $d\Lambda$ is continuous in $q_{V^1}$, as desired.

Next we prove \eqref{claim2}. For the first term, recalling \eqref{estimate} and \eqref{eq:C1}, we have
\begin{equation*}
\begin{split}
	[\textbf{I}_k] =
	~&
	\int_{B(a^1_k,r^1_k)\triangle B(a^2_k,r^2_k)}|u_1v_1|\ dx
	\\
	\le
	~&
	\|u_1v_1\|_{L^\infty(\Omega')}|B(a^1_k,r^1_k)\triangle B(a^2_k,r^2_k)|
	\\
	\le
	~&
	C\|u_1\|_{L^\infty(\Omega')}\|v_1\|_{L^\infty(\Omega')}|(a^1_k,r^1_k)-(a^2_k,r^2_k)|
	\\
	\le
	~&
	C\|\phi\|_{H^\frac12(\partial\Omega)}\|\psi\|_{H^\frac12(\partial\Omega)}|(a^1_k,r^1_k,\lambda^1_k)-(a^2_k,r^2_k,\lambda^2_k)|
	\\
	\le
	~&
	C\|\phi\|_{H^\frac12(\partial\Omega)}\|\psi\|_{H^\frac12(\partial\Omega)}\Norm{V^1-V^2}
\end{split}
\end{equation*}
for each $k=1,\ldots,N$.
Let us focus now on the other term. Adding and subtracting terms and using that $|\lambda^2_k|(r^2_k)^{d-1}<\varrho_1^d$,
\begin{multline*}
	\big|\lambda^1_k(r^1_k)^{d-1}(u_1v_1)(a^1_k+r^1_kx)-\lambda^2_k(r^2_k)^{d-1}(u_1v_1)(a^2_k+r^2_kx)\big|
	\\
	\le
	\|u_1v_1\|_{L^\infty(\Omega)}|\lambda^1_k(r^1_k)^{d-1}-\lambda^2_k(r^2_k)^{d-1}|
	\\
	+
	\varrho_1^d\big|(u_1v_1)(a^1_k+r^1_kx)-(u_1v_1)(a^2_k+r^2_kx)\big|.
\end{multline*}
First, since $r^1_k,r^2_k,|\lambda^1_k|,|\lambda^2_k|<\varrho_1$,
\begin{equation*}
\begin{split}
	|\lambda^1_k(r^1_k)^{d-1}-\lambda^2_k(r^2_k)^{d-1}|
	\le
	~&
	\varrho_1^{d-1}|\lambda^1_k-\lambda^2_k|
	+
	\varrho_1|(r^1_k)^{d-1}-(r^2_k)^{d-1}|
	\\
	\le
	~&
	\varrho_1^{d-1}\big(|\lambda^1_k-\lambda^2_k|+(d-1)|r^1_k-r^2_k|\big)
	\\
	\le
	~&
	C|(a^1_k,r^1_k,\lambda^1_k)-(a^2_k,r^2_k,\lambda^2_k)|
	\\
	\le
	~&
	C\Norm{V^1-V^2}.
\end{split}
\end{equation*}
For the other term, since $u_1,v_1\in C^1(\Omega')$, using again \eqref{eq:C1},
\begin{equation*}
\begin{split}
	|(u_1v_1)(a^1_k+r^1_kx)&-(u_1v_1)(a^2_k+r^2_kx)|
	\\
	&\le
	\|\nabla(u_1v_1)\|_{L^\infty(\Omega')}|(a^1_k+r^1_kx)-(a^2_k+r^2_kx)|
	\\
	&\le
	\|\phi\|_{H^\frac12(\partial\Omega)}\|\psi\|_{H^\frac12(\partial\Omega)}(|a^1_k-a^2_k|+|r^1_k-r^2_k|)
	\\
	&\le
	C\|\phi\|_{H^\frac12(\partial\Omega)}\|\psi\|_{H^\frac12(\partial\Omega)}(|(a^1_k,r^1_k,\lambda^1_k)-(a^2_k,r^2_k,\lambda^2_k)|)
	\\
	&\le
	C\|\phi\|_{H^\frac12(\partial\Omega)}\|\psi\|_{H^\frac12(\partial\Omega)}\Norm{V^1-V^2}
\end{split}
\end{equation*}
for every $x\in\S^{d-1}$.
Hence,
\begin{equation*}
\begin{split}
	[\textbf{II}_k]
	\le
	~&
	C\|\phi\|_{H^\frac12(\partial\Omega)}\|\psi\|_{H^\frac12(\partial\Omega)}\Norm{V^1-V^2}.
\end{split}
\end{equation*}
and \eqref{claim2} follows.\qedhere

\end{proof}

\subsubsection{Complex geometrical optics (CGO) solutions}\label{subsub:CGO}

Recall that we consider the case $d=3$, even though the following construction works for any $d\ge 3$. We recall some basic properties of a special family of solutions of the Schr\"odinger equation first introduced in \cite{Faddeev1965} for quantum inverse scattering and in \cite{Sylvester1987} in inverse boundary value problems. 

Let $\zeta\in\mathbb{C}^d$, such that $\zeta\cdot\zeta=0$. From \cite[Lemma 5.5]{paivarinta2004}, for $|\zeta| \ge C(\|\beta + q_V\|_{L^{\infty}(\Omega)})$ sufficiently large, there exists a solution $u$ of the equation
\begin{equation*}
	-\Delta u+(\beta+q_V)u=0 \quad \text{ in } \Omega,
\end{equation*}
of the form
\begin{equation}\label{eq:zetaR}
	u(x)
	=
	e^{i\zeta\cdot x}(1+R(x)),
\end{equation}
where $R$ satisfies

\begin{equation}\label{L2norm R}
	\|R\|_{L^2(\Omega)}
	\le
	\frac{C}{|\zeta|}
	\qquad\text{ and }\qquad
	\|\nabla R\|_{L^2(\Omega)}
	\le
	C,
\end{equation}
for some constant $C\ge 1$.

Therefore, if $B=B(a,r)$ is a ball with $(a,r,\lambda)\in\P_1$, then
\begin{equation}
	\|R\|_{L^2(B)}
	\le
	\frac{C}{|\zeta|},
\end{equation}
and using the trace inequality (see \cite{ding1996}), interpolation inequality in $H^\theta(\Omega)$ ($\frac{1}{2}<\theta<1$) and \eqref{L2norm R} we can estimate $\|R\|_{L^2(\partial B(a,r))}$, as follows: 
\begin{equation}
\begin{split}
	\|R\|_{L^2(\partial B)}
	\le
	\|R\|_{H^{\theta-1/2}(\partial B)}
	\le
	~&
	C\|R\|_{H^\theta(\Omega)}
	\\
	\le
	~&
	C\|R\|_{H^0(\Omega)}^{1-\theta}\|R\|_{H^1(\Omega)}^\theta
	\\
	\le
	~&
	C\|R\|_{L^2(\Omega)}^{1-\theta}\left(\|R\|_{L^2(\Omega)}+\|\nabla R\|_{L^2(\Omega)}\right)^\theta
	\\
	\le
	~&
	C\Big(\frac{1}{|\zeta|}+1\Big)\frac{1}{|\zeta|^{1-\theta}}
	\\
	\le
	~&
	\frac{C}{|\zeta|^{1-\theta}}.
\end{split}
\end{equation}

Now let $\xi\in\R^d\setminus\{0\}$ be an arbitrary vector. We want to choose $\zeta_1,\zeta_2 \in \C^d$, such that $\zeta_j \cdot \zeta_j = 0$ for $j=1,2$ and $\zeta_1 + \zeta_2 = \xi$. Given any pair of unitary orthogonal vectors $\eta_1$ and $\eta_2$ in the orthogonal subspace $\{\xi\}^\bot\subset\R^d$ (that is, $\xi\cdot\eta_1=\xi\cdot\eta_2=\eta_1\cdot\eta_2=0$ and $|\eta_1|=|\eta_2|=1$), we define
\begin{equation}\label{eq:zeta}
	\zeta_1=\zeta_1(s)=\frac{1}{2}\xi+a\eta_1+b\eta_2
	\qquad \text{ and } \qquad
	\zeta_2=\zeta_2(s)=\frac{1}{2}\xi-a\eta_1-b\eta_2
\end{equation}
where $a$ and $b$ are complex numbers chosen so that $\zeta_1,\zeta_2\in\C^d$ satisfy $\zeta_j\cdot\zeta_j=0$ and $|\zeta_j(s)|=(\zeta_j\cdot\overline{\zeta_j})^{1/2}=s$ for $j=1,2$ where $s\ge 1$ is a free parameter. For example:
\begin{equation*}
	a
	=
	\frac{is}{\sqrt{2}}
	\qquad\text{ and }\qquad
	b
	=
	\frac{1}{\sqrt{2}}\sqrt{s^2-\frac{|\xi|^2}{2}}.
\end{equation*}
Then, by the above discussion there exists a constant $C = C(\|q+\beta\|_{L^\infty(\Omega)})\ge 1$ such that if $s=|\zeta_j(s)|\ge C$ then
\begin{equation}\label{L2norm R-s}
	\|R_j\|_{L^2(\Omega)}
	\le
	\frac{C}{s}
	\qquad\text{ and }\qquad
	\|R_j\|_{L^2(\partial B)}
	\le
	\frac{C}{s^{1-\theta}},
\end{equation}
for $j=1,2$, where $R_j$ corresponds to the choice $\zeta=\zeta_j$ in \eqref{eq:zetaR}. We will now use CGO solutions to prove injectivity of the Fr\'echet derivative of the DN map.

\subsubsection{Injectivity of the Fr\'echet derivative}

\begin{lemma}\label{lem:DN'inj}
Take $q_V\in\widetilde M_N$ and $D\in \R^{(d+2)\times N}$. If
\begin{equation*}
	\prodin{d\Lambda_{q_V}(D)(\phi)}{\psi}
	=
	0,
	\qquad \phi,\psi\in H^{1/2}(\partial\Omega),
\end{equation*}
then $D=0$. As a consequence, $d\Lambda_{q_V}$ is injective.
\end{lemma}

\begin{proof}
Using the notation \eqref{eq:VD}, by \eqref{dF} we have that 
\begin{equation}\label{eq:step1}
	\sum_{k=1}^N\bigg\{
	t_k\int_{B(a_k,r_k)}uv\ dx
	+
	\lambda_k\int_{\partial B(a_k,r_k)}\Big(h_k\cdot\frac{x-a_k}{r_k}+\rho_k\Big)uv\ d\sigma(x)
	\bigg\}=0
\end{equation}
for every  $u,v\in H^1(\Omega)$  solutions of
\begin{equation}\label{eq:global}
	-\Delta u+(\beta+q_V)u=0\quad
	\text{ in } \Omega,
	\qquad
	-\Delta v+(\beta+q_V)v=0\quad
	\text{ in } \Omega,
\end{equation}
where
\begin{equation*}	
q_{V}
	=
	\sum_{k=1}^N \lambda_k\chi_{B(a_k,r_k)}
.
\end{equation*}
We need to show that $h_k=0$ and $t_k=\rho_k=0$ for every $k=1,\dots,N$. Let us fix any $k_0\in \{1,\dots,N\}$. We claim that $h_{k_0}=0$ and $t_{k_0}=\rho_{k_0}=0$.

Let $\varepsilon>0$ be such that $B(a_k,r_k+\varepsilon)\subseteq \Omega$ for every $k$ and
\[
B(a_k,r_k+\varepsilon)\cap B(a_j,r_j+\varepsilon)=\emptyset,\qquad j\neq k.
\]
Set
\[
\Omega_1=\bigcup_{k=1}^N B(a_k,r_k),\qquad \Omega_2=\bigcup_{k=1}^N B(a_k,r_k+\varepsilon).
\]

We now consider functions of the form
\begin{equation}\label{eq:utilde}
\tilde u(x)=
\begin{cases}
e^{i\zeta_1\cdot x}(1+R_1(x)) & \text{if $x\in B(a_{k_0},r_{k_0}+\varepsilon)$,}\\
0 & \text{otherwise,}
\end{cases}
\end{equation}
and
\begin{equation}\label{eq:vtilde}
\tilde v(x)=
\begin{cases}
e^{i\zeta_2\cdot x}(1+R_2(x)) & \text{if $x\in B(a_{k_0},r_{k_0}+\varepsilon)$,}\\
0 & \text{otherwise,}
\end{cases}
\end{equation}
where $\zeta_1$ and $\zeta_2$ are certain vectors in $\mathbb{C}^d$ as in \Cref{subsub:CGO} such that $\zeta_1\cdot\zeta_1=\zeta_2\cdot\zeta_2=0$ and $R_1$ and $R_2$ are chosen so that $u(x)=e^{i\zeta_1\cdot x}(1+R_1(x))$ and $v(x)=e^{i\zeta_2\cdot x}(1+R_2(x))$ satisfy \eqref{eq:global}, and in particular 
\begin{equation*}
	-\Delta \tilde u+(\beta+q_V)\tilde u=0\quad
	\text{ in } \Omega_2,
	\qquad
	-\Delta \tilde v+(\beta+q_V)\tilde v=0\quad
	\text{ in } \Omega_2.
\end{equation*}

We now approximate these local solutions $\tilde u$ and $\tilde v$ by global solutions of \eqref{eq:global} by using the Runge approximation property \cite{LAX-1956,MALGRANGE-1955-56}. More precisely, thanks to the estimates given in \cite[Lemma~4.8]{BAL-UHLMANN-2013} (see also \cite[Corollary~7.9]{alberti-capdeboscq-2018}), there exist $u_n,v_n\in H^1(\Omega)$ such that
\begin{equation}\label{eq:globaln}
	-\Delta u_n+(\beta+q_V)u_n=0\quad
	\text{ in } \Omega,
	\qquad
	-\Delta v_n+(\beta+q_V)v_n=0\quad
	\text{ in } \Omega,
\end{equation}
and
\begin{equation}\label{eq:limitn}
\|u_n-\tilde u\|_{C^0(\overline{\Omega_1})}\to 0,\qquad \|v_n-\tilde v\|_{C^0(\overline{\Omega_1})}\to 0.
\end{equation}

By \eqref{eq:step1} and \eqref{eq:globaln} we have 
\begin{equation*}
	\sum_{k=1}^N\bigg\{
	t_k\int_{B(a_k,r_k)}u_nv_n\ dx
	+
	\lambda_k\int_{\partial B(a_k,r_k)}\Big(h_k\cdot\frac{x-a_k}{r_k}+\rho_k\Big)u_nv_n\ d\sigma(x)
	\bigg\}=0.
\end{equation*}
Taking the limit as $n\to+\infty$, by using \eqref{eq:limitn} we obtain
\begin{equation*}
	\sum_{k=1}^N\bigg\{
	t_k\int_{B(a_k,r_k)}\tilde u \tilde v\ dx
	+
	\lambda_k\int_{\partial B(a_k,r_k)}\Big(h_k\cdot\frac{x-a_k}{r_k}+\rho_k\Big)\tilde u \tilde v\ d\sigma(x)
	\bigg\}=0.
\end{equation*}
Next, by  \eqref{eq:utilde} and \eqref{eq:vtilde} we derive
\begin{equation*}
	\mathcal{G}(uv):=t_{k_0}\int_{B(a_{k_0},r_{k_0})}u  v\ dx
	+
	\lambda_{k_0}\int_{\partial B(a_{k_0},r_{k_0})}\Big(h_{k_0}\cdot\frac{x-a_{k_0}}{r_{k_0}}+\rho_{k_0}\Big) u  v\ d\sigma(x)
=0,
\end{equation*}
where $u(x)=e^{i\zeta_1\cdot x}(1+R_1(x))$ and $v(x)=e^{i\zeta_2\cdot x}(1+R_2(x))$ for each $x\in\Omega$.

In what follows we denote $\Psi=R_1+R_2+R_1R_2$, so that $(uv)(x)=(1+\Psi(x))e^{i\xi\cdot x}$. 
By the linearity of $\mathcal{G}$, after a rearrangement of the terms we get $\mathcal{G}(e^{i\xi\cdot(\cdot)})=-\mathcal{G}(\Psi e^{i\xi\cdot(\cdot)})$. Thus, taking taking absolute values we obtain
\begin{equation}\label{ineq:1}
	\big|\mathcal{G}(e^{i\xi\cdot(\cdot)})\big|
	=
	\big|\mathcal{G}(\Psi e^{i\xi\cdot(\cdot)})\big|
	\le
	 c
	\left(\int_{B}|\Psi|\ dx
	+
	\int_{\partial B}|\Psi|\ d\sigma(x)
	\right),
\end{equation}
where $B=B(a_{k_0},r_{k_0})$ and $c=|t_{k_0}|+|\lambda_{k_0}|(|h_{k_0}|+|\rho_{k_0}|)$.
Next we show that the left-hand side of this inequality is equal to zero by choosing a large enough value of the parameter $s$ in \eqref{eq:zeta}. To do that, let us start by  estimating the first integral in the right-hand side using \eqref{L2norm R-s}:
\begin{equation*}
\begin{split}
	\int_{B}|\Psi|\ dx
	\le
	~&
	\int_{B}|R_1|+|R_2|+|R_1R_2|\ dx
	\\
	\le
	~&
	C(\|R_1\|_{L^2(B)}+\|R_2\|_{L^2(B)})
	+\|R_1\|_{L^2(B)}\|R_2\|_{L^2(B)}
	\\
	\le
	~&
	C\Big(\frac{1}{|\zeta_1(s)|}+\frac{1}{|\zeta_2(s)|}+\frac{1}{|\zeta_1(s)|\,|\zeta_2(s)|}\Big)
	\\
	=
	~&
	C\Big(1+\frac{1}{s}\Big)\frac{1}{s}
	\\
	\le
	~&
	\frac{C}{s},
\end{split}
\end{equation*}
while for the other integral,
\begin{equation*}
\begin{split}
	\int_{\partial B}|\Psi|\ d\sigma(x)
	\le
	~&
	\int_{\partial B}|R_1|+|R_2|+|R_1R_2|\ d\sigma(x)
	\\
	\le
	~&
	C(\|R_1\|_{L^2(\partial B)}+\|R_2\|_{L^2(\partial B)})
	+\|R_1\|_{L^2(\partial B)}\|R_2\|_{L^2(\partial B)}
	\\
	\le
	~&
	C\Big(\frac{1}{|\zeta_1(s)|^{1-\theta}}
	+
	\frac{1}{|\zeta_2(s)|^{1-\theta}}
	+
	\frac{1}{|\zeta_1(s)|^{1-\theta}|\zeta_2(s)|^{1-\theta}}\Big)
	\\
	=
	~&
	C\Big(1+\frac{1}{s^{1-\theta}}\Big)\frac{1}{s^{1-\theta}}
	\\
	\le
	~&
	\frac{C}{s^{1-\theta}}.
\end{split}
\end{equation*}
Therefore, replacing in \eqref{ineq:1} we get that
\begin{equation*}
	\big|\mathcal{G}(e^{i\xi\cdot(\cdot)})\big|
	\le
	C\Big(\frac{1}{s}+\frac{1}{s^{1-\theta}}\Big)
	\le
	\frac{C}{s^{1-\theta}}.
\end{equation*}
which holds for every large enough $s\ge C\ge 1$. As a consequence, since $\theta<1$,
\begin{equation*}
	\mathcal{G}(e^{i\xi\cdot(\cdot)})
	=
	0,
\end{equation*}
which is equivalent to
\begin{equation*}
	t_{k_0}\int_{B(0,1)}r_{k_0}e^{ir_{k_0}\xi\cdot x}\ dx
	+
	\lambda_{k_0}\int_{\S^{d-1}}(h_{k_0}\cdot x+\rho_{k_0})e^{ir_{k_0}\xi\cdot x}\ d\sigma(x)
	=
	0,
\end{equation*}
where $\S^{d-1}=\partial B(0,1)$ denotes the $(d-1)$-dimensional unit sphere of $\R^d$.

The proof will be an immediate consequence of the following lemma.
\end{proof}


\begin{lemma}
Let $(a,r,\lambda)\in\P_1$ and $(h,t,r)\in\R^{d+2}$. If
\begin{equation}\label{eq:1}
	t\int_{B(0,1)}re^{ir\xi\cdot x}\ dx
	+
	\lambda\int_{\S^{d-1}}(h\cdot x+\rho)e^{ir\xi\cdot x}\ d\sigma(x)
	=
	0,
\end{equation}
holds for every $\xi\in\R^d$, then $(h,t,r)=0$.
\end{lemma}

\begin{proof}
Let us focus on the first integral in \eqref{eq:1}. Using that
\begin{equation*}
	re^{ir\xi\cdot x}
	=
	-\frac{i}{|\xi|^2}\,\xi\cdot\nabla(e^{ir\xi\cdot x}),
\end{equation*}
together with the divergence theorem we get
\begin{equation*}
	\int_{B(0,1)}re^{ir\xi\cdot x}\ dx
	=
	-\frac{i}{|\xi|^2}\int_{\S^{d-1}}\xi\cdot x\, e^{ir\xi\cdot x}\ d\sigma(x),
\end{equation*}
and replacing in \eqref{eq:1},
\begin{equation}\label{eq:2}
\begin{split}
	0
	=
	~&
	\int_{\S^{d-1}}\bigg(\lambda(h\cdot x+\rho)-\frac{it}{|\xi|^2}\,\xi\cdot x\bigg)e^{ir\xi\cdot x}\ d\sigma(x)
	\\
	=
	~&
	\lambda\rho\int_{\S^{d-1}}e^{ir\xi\cdot x}\ d\sigma(x)
	+
	\bigg(\lambda h-\frac{it}{|\xi|^2}\,\xi\bigg)\cdot\int_{\S^{d-1}}xe^{ir\xi\cdot x}\ d\sigma(x).
\end{split}
\end{equation}
We study each integral separately:
\begin{equation*}
\begin{split}
	\int_{\S^{d-1}}e^{ir\xi\cdot x}\ d\sigma(x)
	=
	\int_{\S^{d-1}}e^{ir|\xi|x_1} d\sigma(x)
	=
	\int_{\S^{d-1}}\cos(r|\xi|x_1)\ d\sigma(x),
\end{split}
\end{equation*}
and
\begin{equation*}
\begin{split}
	\int_{\S^{d-1}}xe^{ir\xi\cdot x}\ d\sigma(x)
	=
	~&
	\frac{\xi}{|\xi|}\int_{\S^{d-1}}x_1e^{ir|\xi|x_1}\ d\sigma(x)
	\\
	=
	~&
	i\,\frac{\xi}{|\xi|}\int_{\S^{d-1}}x_1\sin(r|\xi|x_1)\ d\sigma(x),
\end{split}
\end{equation*}
where we have used the symmetry of the unit sphere $\S^{d-1}$ to discard the odd terms.
If we define $I_1,I_2:\R\to\R$ by
\begin{equation*}
\begin{split}
	&
	I_1(\tau)
	:\,=
	\int_{\S^{d-1}}\cos(r\tau x_1)\ d\sigma(x),
	\\
	&
	I_2(\tau)
	:\,=
	\int_{\S^{d-1}}x_1\sin(r\tau x_1)\ d\sigma(x),
\end{split}
\end{equation*}
then \eqref{eq:2} reads as follows,
\begin{equation*}
	\lambda\rho I_1(|\xi|)+t\frac{I_2(|\xi|)}{|\xi|}
	+
	i\lambda h\cdot\xi\, \frac{I_2(|\xi|)}{|\xi|}
	=
	0.
\end{equation*}
Equivalently, 
\begin{equation*}
	\begin{cases}
		\lambda\rho|\xi| I_1(|\xi|)+tI_2(|\xi|)=0,
		\\
		\lambda h\cdot\xi\, I_2(|\xi|)=0.
	\end{cases}
\end{equation*}
Hence if we choose $\frac{\pi}{4r}\le|\xi|\le\frac{\pi}{2r}$, both $I_1(|\xi|)$ and $I_2(|\xi|)$ are positive real numbers, and from the second identity we immediately deduce that $h=0$.

Now we show that this also implies that $\rho=t=0$. We start by assuming that $t\neq 0$, since otherwise $\rho=0$. Note that $I_1' = -r I_2$. Inserting this into the first identity we get
\begin{equation*}
	\lambda\rho|\xi| I_1(|\xi|)-\frac{t}{r}I_1'(|\xi|)
	=
	0.
\end{equation*}
Solving the differential equation with initial value $I_1(0)=\sigma(\S^{d-1})=\frac{2\pi^{d/2}}{\Gamma(d/2)}$ we obtain that
\begin{equation*}
	I_1(|\xi|)
	=
	\frac{2\pi^{d/2}}{\Gamma(d/2)}\exp\Big\{\frac{\lambda\rho r}{2t}|\xi|^2\Big\}.
\end{equation*}
Then the contradiction follows by recalling \cite[Appendix B.4]{grafakos-2014}
\[
I_1(|\xi|) = \frac{(2 \pi)^{d/2}}{(r|\xi|)^{\frac{d-2}{2}}} J_{\frac{d-2}{2}}(r|\xi|),
\]
where $J_{\frac{d-2}{2}}(z)$ denotes the Bessel function of the first kind of order $\frac{d-2}{2}$. Indeed this implies that
\begin{equation*}
	J_{\frac{d-2}{2}}(z)
	=
	c_1z^{\frac{d-2}{2}}e^{c_2z^2}
\end{equation*}
for some constants $c_1,c_2\neq 0$, which yields the desired contradiction. Thus $(h,\rho,t)=0$.
\end{proof}

\subsubsection{Proof of \Cref{thm:gelf}.} We are now ready to prove \Cref{thm:gelf}.

\begin{proof}[Proof of \Cref{thm:gelf}]
We need to verify the assumptions of Theorem~\ref{QNF-stability-multiple}.

 \[
 \varrho_0=\varrho_{\rm min}/2,\qquad \varrho_1=\varrho_{\rm max}+\delta/3,\qquad A=a_{\rm max}+\delta/3,
 \]
where $ \delta = d(\overline{B(0,a_{\rm max}+\varrho_{\rm max})},\partial\Omega)>0$, so that $B(0,A+\varrho_1)\subseteq\Omega$, $\varrho_0<\varrho_{\rm min}$ and $\varrho_1>\varrho_{\rm max}$. Let 
\[M =
	\bigcup_{p=1}^{N_B} \widetilde M_p.
\]
 From \Cref{lem:MNopen}  we have that $\widetilde M_p$, $p=1,\dots,N_B$, are $p(d+2)$-dimensional Lipschitz manifolds in $L^1(\Omega)$. Furthermore, by \Cref{lem:manifdisj}, they are pairwise disjoint.

The set $K$ from the statement can be decomposed as $K = \cup_{p=1}^{N_B} K_p$ with $K_p \subseteq\widetilde M_p$, where each $K_p$ is compact thanks to the condition
\[
\|(-\Delta +\beta+ q)^{-1}\|_{\L(H^1_0(\Omega),H^{-1}(\Omega))}\le D,
\]
and the bounds on the parameters $a_k,\lambda_k, r_k$, $k=1,\dots,p$.

The Dirichlet-to-Neumann map $q \mapsto \Lambda_q$ is shown to be of class $C^1$ in Lemmas~\ref{lem:DNdif} and \ref{lem:DNC1}, while the Frech\'et derivative is injective thanks to Lemma~\ref{lem:DN'inj}. It is a classical result that the the Dirichlet-to-Neumann map uniquely determines a $L^\infty(\Omega)$ potential in dimension $d \ge 3$ (see, for instance \cite{hahner1996}), so this yields immediately the injectivity of the map $q \mapsto \Lambda_q$.

Let $\tilde Y = \{ y \in Y: y \text{ is a compact operator}\}$. Then it is well known \cite{mandache2001} that $\Lambda_{q_2}-\Lambda_{q_1} \in \tilde Y$ for every $q_1,q_2 \in M$, because $q_2 (x)= q_1 (x)$ for $x$ close to $\partial \Omega$ by construction. Moreover $\tilde Y$ is closed, therefore $\ran (d\Lambda_q) \subseteq \tilde Y$ for every $q \in M$, as shown in Remarks~\ref{rem:frechet} and \ref{ref:deriv_compact}.

The assumptions of Theorem~\ref{QNF-stability-multiple} are now verified, and this yields the desired Lipschitz stability estimate of \Cref{thm:gelf}.
\end{proof}

\section{Conclusions}\label{sec:conclusion}
In this work, we showed that the ill-posedness of inverse problems may be mitigated by assuming a priori that the unknown quantity belongs to a known low-dimensional manifold. This is a realistic assumption in many applied scenarios and is a standard setup in machine learning. The hypothesis $x\in M$ may be viewed as a prior and regularises the inverse problem, yielding stability. More precisely, 
H\"older and Lipschitz stability results were obtained with infinite-dimensional and finite-dimensional measurements. An extension to the case where $x \notin M$ was also presented. A globally-convergent reconstruction algorithm was designed. The theory was applied to several toy examples as well as to two inverse boundary value problems, for which new Lipschitz stability results with finite measurements were derived.

The abstract approach developed in this paper provides a solid foundation for the use of low-dimensional manifolds as a priori assumptions in inverse problems. However, the results are preliminary and many interesting issues and questions remain open. We outline some of them here.

\begin{itemize}
\item We have already mentioned in the introduction that the Lipschitz stability results derived in this work are nothing but the $\mathcal{M}$-RIP property, which is the basic assumption of the approach based on learning developed in \cite{hyun2020deep}. In particular, even if in principle our results are based on the knowledge of $M$, these a priori estimates are useful also in cases when $M$ is unknown and has to be learned from a training set of samples, by using unsupervised learning. It would be interesting to investigate this aspect in more detail, and consider simultaneously the manifold learning problem and the inverse problem. In the context of manifold learning, this problem is related to the works \cite{fefferman2019fitting,2016-fefferman-etal}, in which one looks for a manifold fitting a training set of noisy data.
\item In the last years, it has become very popular to rewrite iterative regularisation schemes for inverse problems as neural networks, a process called \textit{unrolling} (see \cite{2017-adler-oktem,2018-adler-oktem,2019-arridge-etal} and references therein). These networks can then be partially learned by using a training set, with the aim of optimising the recovery with the actual data. This approach has been extended to more general reconstruction algorithms for linear and nonlinear inverse problems \cite{de2019deep,bubba2020deep}. It would be interesting to see whether the reconstruction algorithm of this work may be written as a neural network, and whether supervised learning may be used to learn some of the parameters of the network, like those related to the manifold $M$ which, as mentioned before, may be unknown.
\item The manifolds considered in this work are without boundary and, by definition, of fixed dimension. However, there are situations where boundary points may be of interest, in order to cover, for instance, the case of degenerate polygons, for which the number of parameters vary. Extending the current theory to this generalised setting would enlarge the range of applicability of these results.
\item Another key assumption of the main results is the injectivity of the differential of $F$. However, in some cases, the differential of $F$ may not be injective. It would be interesting to investigate whether a workaround may be found by using higher order  derivatives of $F$, as in certain higher order inverse mapping theorems
\cite{1986-grasse,1989-Frankowska,1990-Frankowska}.
\item The examples discussed in this paper are presented with the main objective of illustrating the results in simplified settings. The applications to other inverse problems, such as inverse scattering problems or inverse boundary value problems for the wave equation, is left for future work.
\item The numerical implementation of the reconstruction algorithm presented in this work for some  applications, for instance in one of those discussed as examples, is a necessary step towards validating the approach and testing its applicability.
\item A very ambitious task would be the study of compressed sensing results in this nonlinear setting \cite{2013-Blumensath}, with signals that are sparse in a nonlinear manifold. In other words, the sparsity of $x\in M$ is measured via the charts, and corresponds to the sparsity of $\varphi_i(x)$ in the parameter space. Ideally, under this assumption, it should be possible to reduce the number of measurements $N$.
\end{itemize}

{\small
\bibliographystyle{abbrv}
\bibliography{manifold}
}


\appendix

\section{Tangent spaces and differentials}\label{sec:tangent}

In this section we recall some basic notions on the tangent space of a manifold and on the differential of a map between manifolds.

\subsection{The tangent space}
Let $X$ be a Banach space and $M\subseteq X$ be an $n$-dimensional differentiable manifold with atlas $\{(U_i,\varphi_i)\}_{i\in I}$ (\Cref{def:manifold}). For $x\in M$, we define the tangent space $T_xM$ of $M$ at $x$ as the quotient space
\[
T_x M :\,=\{\gamma\colon (-1,1)\to M:\text{$\gamma(0)=x$ and $\varphi_i\circ\gamma$ is differentiable in $0$}\}/\sim,
\]
where the equivalence relation is defined by
\[
\gamma_1 \sim \gamma_2 \iff (\varphi_i\circ\gamma_1)'(0)=(\varphi_i\circ\gamma_2)'(0),
\]
where $i\in I$ is such that $x\in U_i$. The equivalence class of $\gamma$ is denoted by $[\gamma]$. It is worth noting that, due to the differentiability of the transition maps, the definitions of $T_xM$ and of $[\gamma]$ are independent of the chart. The tangent space $T_xM$ inherits a vector space structure thanks to the bijection
\begin{equation}\label{eq:bijection}
T_x M \to \R^n,\qquad [\gamma]\mapsto (\varphi_i\circ\gamma)'(0).
\end{equation}
In this paper, we always identify  the elements of $T_xM$ with vectors $h=(\varphi_i\circ\gamma)'(0)$ in $\R^n$.

When $M$ is embedded in $X$ (\Cref{rem:embedded}), the tangent space $T_x M$ may be viewed as a subspace of $X$. This is achieved by using the identification \eqref{eq:bijection} and the immersion $(\varphi_i^{-1})'(\varphi_i(x))\colon\R^n \to X$ as follows:
\begin{equation}\label{eq:tangent-embedded}
T_x M\to X,\qquad  [\gamma] \mapsto (\varphi_i^{-1})'(\varphi_i(x))h, \qquad h=(\varphi_i\circ\gamma)'(0).
\end{equation}
Note that this map gives the standard interpretation of the tangent space as the collection of tangent vectors, since
\[
\gamma'(0) = (\varphi_i^{-1}\circ\varphi_i\circ\gamma)'(0)=(\varphi_i^{-1})'(\varphi_i(x))h.
\]
This expression also shows that, even though the identification given in \eqref{eq:bijection} depends on the chart, the embedding \eqref{eq:tangent-embedded} is intrinsic to the manifold and is independent of the chart used.

\subsection{Differential}\label{sub:diff}
Let $F\colon M\to  Y$ be a differentiable function (\Cref{F in M diffble}). The differential $dF_x$ of $F$ at $x\in M$ is defined by
\[
dF_x\colon T_x M\to Y,\qquad h \mapsto (F\circ\varphi_i^{-1})'(\varphi_i(x)) h,
\]
where $i\in I$ is such that $x\in U_i$. Here we are looking at $Y$ as an infinite-dimensional manifold modelled on $Y$ itself and identifying the tangent space to $Y$ at $F(x)$ with $Y$.

It is worth observing that the expression of $dF_x$ does depend on $i$, in contrast to the definition of $T_x M$. However, if $j\in I$ is another index for which $x\in U_j$, by the chain rule we have
\[
\begin{split}
(F\circ\varphi_i^{-1})'(\varphi_i(x)) &= (F\circ\varphi_j^{-1}\circ\varphi_j\circ\varphi_i^{-1})'(\varphi_i(x)) \\
&= (F\circ\varphi_j^{-1})'(\varphi_j(x)) (\varphi_j\circ\varphi_i^{-1})'(\varphi_i(x)).
\end{split}
\]
Condition (3) of \Cref{def:manifold} implies that the transition map $\varphi_j\circ\varphi_i^{-1}$ is a diffeomorphism, so that $(\varphi_j\circ\varphi_i^{-1})'(\varphi_i(x))\colon \R^n\to\R^n$ is an invertible linear map, which can be seen as change of variables. Thus, the two maps $(F\circ\varphi_i^{-1})'(\varphi_i(x))$ and $(F\circ\varphi_j^{-1})'(\varphi_j(x))$ coincide up to a change of variables. In particular, the injectivity of $dF_x$ is an intrinsic property, independent of the chart.

When $M$ is embedded in $X$ and $F\in C^1(A,Y)$ for some open set $A\supseteq M$, since $\varphi_i^{-1}$ is differentiable, we can apply the chain rule and obtain
\[
dF_x(h)=(F\circ\varphi_i^{-1})'(\varphi_i(x)) h=F'(x)\circ (\varphi_i^{-1})'(\varphi_i(x)) h.
\]
Thanks to the identification of $T_x M$ as a subspace of $X$ via \eqref{eq:tangent-embedded}, this identity shows that the differential of $F$ at $x$ coincides with the Fr\'echet derivative of $F$ at $x$ restricted to $T_x M$, namely,
\begin{equation}\label{eq:differential}
dF_x = F'(x)|_{T_x M}.
\end{equation}

\section{Estimates for the symmetric difference of two balls}\label{sec:estimates}

In this section we prove some estimates we need in \Cref{sec:exa}.

\begin{lemma}\label{bilipest}
Take $A>0$ and $0<\varrho_0< \varrho_1<\infty$. Then the following inequalities hold for every $a_1,a_2\in \overline{B(0,A)}\subseteq\R^d$ and $r_1,r_2\in[\varrho_0,\varrho_1]$,
\begin{equation}\label{estimate}
	\frac{1}{C}|(a_1,r_1)-(a_2,r_2)|
	\le
	|B(a_1,r_1)\triangle B(a_2,r_2)|
	\le
	C|(a_1,r_1)-(a_2,r_2)|
\end{equation}
for some $C=C(d,A,\varrho_0,\varrho_1)\ge 1$.
\end{lemma}

\begin{proof}
Observe that since $(a,r)\mapsto|a|+|r|$ defines a norm in $\R^{d+1}$, then
\begin{equation*}
	|(a,r)|
	\asymp
	|a|+|r|
\end{equation*}
and thus \eqref{estimate} is equivalent to
\begin{equation*}
	|B(a_1,r_1)\triangle B(a_2,r_2)|
	\asymp
	|a_1-a_2|+|r_1-r_2|.
\end{equation*}

For simplicity, fixed $(a_1,r_1)$  and $(a_2,r_2)$, we will denote $B_1=B(a_1,r_1)$ and $B_2=B(a_2,r_2)$. We split the proof of \eqref{estimate} in three different cases depending on the values of $|a_1-a_2|$, $|r_1-r_2|$ and $r_1+r_2$. We set $K=\overline{B(0,A)}\times [\varrho_0,\varrho_1]$.

\subsection*{Case $|a_1-a_2|\ge r_1+r_2$}

In this case $B_1\cap B_2=\emptyset$ and
\begin{equation*}
\begin{split}
	|B_1\triangle B_2|
	=
	~&
	|B_1|+|B_2|
	=
	\omega_d(r_1^d+r_2^d)
	\\
	=
	~&
	\omega_d\left[\left(\frac{r_1}{r_1+r_2}\right)^d+\left(\frac{r_2}{r_1+r_2}\right)^d\right](r_1+r_2)^d,
\end{split}
\end{equation*}
where  we denote the Lebesgue measure of the unit ball of $\R^d$ by $\omega_d$.
By the convexity of $t\mapsto t^d$ and the fact that $\frac{r_1}{r_1+r_2}<1$ we get that
\begin{equation*}
	2^{1-d}\omega_d(r_1+r_2)^d
	\le
	|B_1\triangle B_2|
	\le
	\omega_d(r_1+r_2)^d.
\end{equation*}
Moreover, since $r_1+r_2\le|a_1-a_2|$ by assumption and $|a_1-a_2|\le|a_1-a_2|+|r_1-r_2|\le c|(a_1,r_1)-(a_2,r_2)|\le c\diam K$ for some fixed constant $c$, where the diameter is computed with respect to the standard norm,
\begin{equation*}
	\frac{2^{1-d}\omega_d(r_1+r_2)^d}{c\diam K}|a_1-a_2|
	\le
	|B_1\triangle B_2|
	\le
	\omega_d(r_1+r_2)^{d-1}|a_1-a_2|.
\end{equation*}
On the other hand, using $|a_1-a_2|\ge r_1+r_2\ge |r_1-r_2|\ge 0$ we obtain the estimate
\begin{equation*}
	\frac{\omega_d(r_1+r_2)^d}{2^dc\diam K}\big(|a_1-a_2|+|r_1-r_2|\big)
	\le
	|B_1\triangle B_2|
	\le
	\omega_d(r_1+r_2)^{d-1}\big(|a_1-a_2|+|r_1-r_2|\big).
\end{equation*}
Finally, recalling that $r_1,r_2\in[\varrho_0,\varrho_1]$, we get
\begin{equation*}
	\frac{\omega_d\varrho_0^d}{c\diam K}\big(|a_1-a_2|+|r_1-r_2|\big)
	\le
	|B_1\triangle B_2|
	\le
	2^{d-1}\omega_dR^{d-1}\big(|a_1-a_2|+|r_1-r_2|\big),
\end{equation*}
so \eqref{estimate} follows.

\subsection*{Case $|a_1-a_2|<|r_1-r_2|$}

In this case, either
\begin{equation*}
	|a_1-a_2|+r_2<r_1
	\quad\Rightarrow\quad
	B_2\subseteq B_1
	\qquad \text{ or } \qquad
	|a_1-a_2|+r_1<r_2
	\quad\Rightarrow\quad
	B_1\subseteq B_2.
\end{equation*}
In any case, by the mean value theorem
\begin{equation*}
	|B_1\triangle B_2|
	=
	\big||B_1|-|B_2|\big|
	=
	\omega_d|r_1^d-r_2^d|
	=
	d\omega_d\xi^{d-1}|r_1-r_2|
\end{equation*}
for some $\xi\in(r_1,r_2)$. Thus, recalling that $r_1,r_2\in[\varrho_0,\varrho_1]$,
\begin{equation*}
	d\omega_d\varrho_0^{d-1}|r_1-r_2|
	\le
	|B_1\triangle B_2|
	\le
	d\omega_dR^{d-1}|r_1-r_2|.
\end{equation*}
Since $0\le|a_1-a_2|<|r_1-r_2|$ by assumption,
\begin{equation*}
	\frac{d}{2}\omega_d\varrho_0^{d-1}\big(|a_1-a_2|+|r_1-r_2|\big)
	\le
	|B_1\triangle B_2|
	\le
	d\omega_dR^{d-1}\big(|a_1-a_2|+|r_1-r_2|\big),
\end{equation*}
which implies \eqref{estimate}.

\subsection*{Case $|r_1-r_2|\le|a_1-a_2|<r_1+r_2$}\label{sub:B3}
In this case $B_1\cup B_2$ can be decomposed as the union of three nonempty disjoint sets: $B_1\cap B_2$, $B_1\setminus B_2$ and $B_2\setminus B_1$.

We first prove the second inequality of \eqref{estimate} by constructing two balls, one contained inside the other, such that their symmetric difference contains $B_1\triangle B_2$.

Let $B(z_1,s_1)$ and $B(z_2,s_2)$ be the smallest ball containing $B_1\cup B_2$ and the largest ball contained in $B_1\cap B_2$, respectively (see \Cref{fig:1}). Then $2s_1=\diam(B_1\cup B_2)=r_1+|a_1-a_2|+r_2$. On the other hand, $\diam(B_1\cup B_2)=(r_1+|a_1-a_2|-r_2)+2s_2+(r_2+|a_1-a_2|-r_1)=2(s_2+|a_1-a_2|)$, so $2s_2=\diam(B_1\cup B_2)-2|a_1-a_2|=r_1+r_2-|a_1-a_2|$. Therefore
\begin{equation*}
	B_1\triangle B_2
	\subseteq
	B\Big(z_1,\frac{r_1+r_2+|a_1-a_2|}{2}\Big)\setminus B\Big(z_2,\frac{r_1+r_2-|a_1-a_2|}{2}\Big).
\end{equation*}
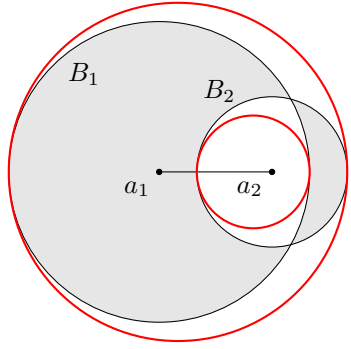
\begin{figure}
\begin{tikzpicture}[scale=1]
\fill[even odd rule,gray!20!white] (0,0) circle (2) (1.5,0) circle (1);
\draw (0,0)--(1.5,0);
\draw (0,0) circle (2);
\draw (1.5,0) circle (1);
\draw [thick,red](1.25,0) circle (0.75);
\draw [thick,red](0.25,0) circle (2.25);
\draw (-1,1.3) node {$B_1$};
\draw (0.8,0.8) node [above] {$B_2$};
\filldraw (0,0) circle (1pt) node [below left]{$a_1$};
\filldraw (1.5,0) circle (1pt) node [below left]{$a_2$};
\end{tikzpicture}
\caption{The shaded region corresponds to $B_1\triangle B_2$ and it is contained between the two red balls.}
\label{fig:1}
\end{figure}
Using this inclusion we can estimate
\begin{equation*}
	|B_1\triangle B_2|
	\le
	\frac{\omega_d}{2^d}\big[(r_1+r_2+|a_1-a_2|)^d-(r_1+r_2-|a_1-a_2|)^d\big].
\end{equation*}
Observe that, given $0<h<t$, by the mean value theorem there exists $\xi\in[t-h,t+h]$ such that $(t+h)^d-(t-h)^d=2d\xi^{d-1}h$, then we can estimate $(t+h)^d-(t-h)^d\le 2d(t+h)^{d-1}h$. Replacing this with $t=r_1+r_2$ and $h=|a_1-a_2|$ in the inequality above we obtain that
\begin{equation*}
\begin{split}
	|B_1\triangle B_2|
	\le
	~&
	\frac{d\omega_d}{2^{d-1}}(r_1+r_2+|a_1-a_2|)^{d-1}|a_1-a_2|
	\\
	\le
	~&
	2^{d-1}d\omega_dR^{d-1}|a_1-a_2|
	\\
	\le
	~&
	2^{d-1}d\omega_dR^{d-1}\big(|a_1-a_2|+|r_1-r_2|\big),
\end{split}
\end{equation*}
where in the second inequality we have used that $r_1+r_2+|a_1-a_2|<2(r_1+r_2)\leq4R$.

To prove the first inequality in \eqref{estimate} we need to distinguish two cases:

\subsection*{ Case $|a_1-a_2|>\max\{r_1,r_2\}$}
Let us consider the disjoint balls
\begin{equation*}
	B\Big(z_3,\frac{|a_1-a_2|+r_1-r_2}{2}\Big)
	\subseteq
	B_1\setminus B_2
	\quad\text{ and }\quad
	B\Big(z_4,\frac{|a_1-a_2|+r_2-r_1}{2}\Big)
	\subseteq
	B_2\setminus B_1,
\end{equation*}
(see \Cref{fig:2}).
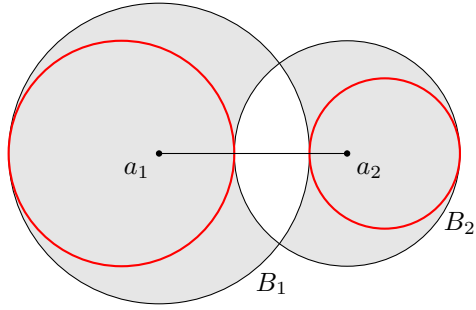
\begin{figure}
\begin{tikzpicture}[scale=1]
\fill[even odd rule,gray!20!white] (0,0) circle (2) (2.5,0) circle (1.5);
\draw (0,0) circle (2);
\draw (2.5,0) circle (1.5);
\draw [thick,red](-0.5,0) circle (1.5);
\draw [thick,red](3,0) circle (1);
\draw (1.5,-2) node [above] {$B_1$};
\draw (4,-1.2) node [above] {$B_2$};
\draw (0,0)--(2.5,0);
\filldraw (0,0) circle (1pt) node [below left]{$a_1$};
\filldraw (2.5,0) circle (1pt) node [below right]{$a_2$};
\end{tikzpicture}
\caption{$B_1\triangle B_2$ contains the two disjoint red balls.}
\label{fig:2}
\end{figure}
Then by the convexity of $t\mapsto t^d$,
\begin{equation*}
\begin{split}
	|B_1\triangle B_2|
	\ge
	~&
	\frac{\omega_d}{2^d}\Big[(|a_1-a_2|+r_1-r_2)^d+(|a_1-a_2|+r_2-r_1)^d\Big]
	\\
	\ge
	~&
	\frac{\omega_d}{2^{d-1}}|a_1-a_2|^d
	\\
	\ge
	~&
	\frac{\omega_d\varrho_0^{d-1}}{2^{d-1}}|a_1-a_2|
	\\
	\ge
	~&
	\frac{\omega_d\varrho_0^{d-1}}{2^d}\big(|a_1-a_2|+|r_1-r_2|\big),
\end{split}
\end{equation*}
where in the third inequality we used the assumption $|a_1-a_2|>\max\{r_1,r_2\}\ge\varrho_0$, and in the last inequality that $|a_1-a_2|\ge|r_1-r_2|$. Then the first inequality in \eqref{estimate} follows.

\subsection*{Case $|a_1-a_2|\le\max\{r_1,r_2\}$}
Let us assume without loss of generality that $r_1\ge r_2$. Then
\begin{equation*}
	B_1\triangle B_2
	\supseteq
	B_1\setminus B_2
	\supseteq
	B(a_1,r_1)\setminus B(a_2,r_1)
	\supseteq
	S,
\end{equation*}
where $S$ is the set enclosed by the red line in \Cref{fig:3}.
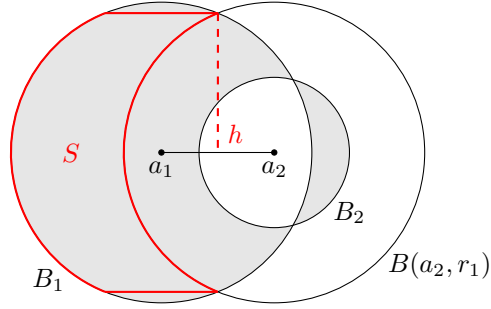
\begin{figure}
\begin{tikzpicture}[scale=1]
\draw [name path=circaux,white] (-1.5,0) circle (2);
\fill[even odd rule,gray!20!white] (0,0) circle (2) (1.5,0) circle (1);
\draw [name path=circ1] (0,0) circle (2);
\draw (1.5,0) circle (1);
\draw [name path=circ2](1.5,0) circle (2);
\draw [name intersections={of=circ1 and circ2}]  
      \foreach \i in {1,2} {(intersection-\i) coordinate (inter-\i)};
\draw [name intersections={of=circ1 and circaux}]  
      \foreach \i in {1,2} {(intersection-\i) coordinate (interaux-\i)};
\draw [thick,red, dashed] (inter-1) -- (0.75,0) node [above right] {$h$}; 
\draw [thick,red] (inter-1) -- (interaux-1);
\draw [thick,red] (inter-2) -- (interaux-2);
\coordinate(Iaux) at ($(interaux-1)$);
\coordinate(Faux) at ($(interaux-2)$);
\coordinate(I) at ($(inter-1)$);
\coordinate(F) at ($(inter-2)$);
\coordinate(C1) at (0, 0);
\coordinate(C2) at (1.5, 0);
\draw pic [draw, angle radius=2cm,thick,red] {angle=Iaux--C1--Faux};
\draw pic [draw, angle radius=2cm,thick,red] {angle=I--C2--F};
\draw (-1.5,-1.7) node{$B_1$};
\draw (2.5,-0.8) node {$B_2$};
\draw (3.7,-1.5) node {$B(a_2,r_1)$};
\draw [red] (-1.2,0) node {$S$};
\draw (0,0)--(1.5,0);
\filldraw (0,0) circle (1pt) node [below]{$a_1$};
\filldraw (1.5,0) circle (1pt) node [below]{$a_2$};
\end{tikzpicture}
\caption{$S\subseteq B_1\triangle B_2$.}
\label{fig:3}
\end{figure}

Observe that the intersection of $S$ with any line parallel to $a_1-a_2$ has length $|a_1-a_2|$. Then, by Steiner symmetrization with respect to the hyperplane orthogonal to $a_1-a_2$, we have that the measure of $S$ is equal to the measure of a cylinder of height $|a_1-a_2|$ and radius
\begin{equation*}
	h
	=
	\sqrt{r_1^2-\frac{|a_1-a_2|^2}{4}}
	\ge
	\frac{\sqrt{3}}{2}\,r_1
	\ge
	\frac{1}{2}\,\varrho_0,
\end{equation*}
where we have used $|a_1-a_2|\le r_1$ and $r_1\ge\varrho_0$. Then we can estimate
\begin{equation*}
\begin{split}
	|B_1\triangle B_2|
	\ge
	|S|
	=
	~&
	\omega_{d-1}h^{d-1}|a_1-a_2|
	\\
	\ge
	~&
	\frac{\omega_{d-1}\varrho_0^{d-1}}{2^{d-1}}|a_1-a_2|
	\\
	\ge
	~&
	\frac{\omega_{d-1}\varrho_0^{d-1}}{2^d}\big(|a_1-a_2|+|r_1-r_2|\big),
\end{split}
\end{equation*}
and the proof of \eqref{estimate} is finished. 
\end{proof}

\begin{remark}
It is noteworthy to mention that all the estimates in the proof of \Cref{bilipest} are valid for balls with respect to other norms (such as, for example, the infinity norm in $\R^d$), except the lower estimate in the case $|r_1-r_2|\le|a_1-a_2|\le\max\{r_1,r_2\}$, where the geometry of the balls has been used. However, it is possible to adapt the idea to obtain the desired estimate for norms different than the usual Euclidean norm.

For instance, if we define $Q(a,r)=\{x\in\R^d\,:\,|x-a|_\infty<r\}$ and we assume that $(a_1,r_1)$ and $(a_2,r_2)$ are points in $\R^d\times\R^+$ such that $0\le r_1-r_2\le|a_1-a_2|_\infty\le r_1$,
then $Q_1\triangle Q_2$ contains a rectangle $S$ of measure $(2r_1)^{d-1}|a_1-a_2|_\infty$ (see \Cref{fig:4}). Then the desired estimate is obtained following the same reasoning.
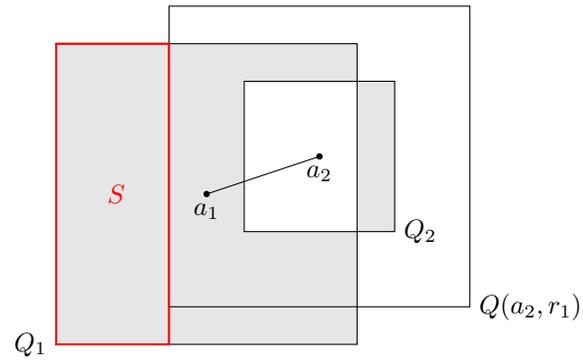
\begin{figure}
\begin{tikzpicture}[scale=1]
\fill[even odd rule,gray!20!white] (-2,-2)--(2,-2)--(2,2)--(-2,2)--cycle (2.5,-0.5)--(2.5,1.5)--(0.5,1.5)--(0.5,-0.5)--cycle;
\draw (0,0)--(1.5,0.5);
\filldraw (0,0) circle (1pt) node [below]{$a_1$};
\filldraw (1.5,0.5) circle (1pt) node [below]{$a_2$};
\draw (-2,-2)--(2,-2)--(2,2)--(-2,2)--cycle node [left] {$Q_1$};
\draw (2.5,-0.5)--(2.5,1.5)--(0.5,1.5)--(0.5,-0.5)--cycle node [right] {$Q_2$};
\draw (3.5,-1.5)--(3.5,2.5)--(-0.5,2.5)--(-0.5,-1.5)--cycle node [right] {$Q(a_2,r_1)$};
\draw [thick,red](-2,-2)--(-0.5,-2)--(-0.5,2)--(-2,2)--cycle;
\draw [red] (-1.2,0) node {$S$};
\end{tikzpicture}
\caption{$S\subseteq Q_1\triangle Q_2$.}
\label{fig:4}
\end{figure}
\end{remark}

\end{document}